\documentclass[12pt,twoside]{amsart}
\usepackage{geometry}
\usepackage{amssymb,amsmath,amsthm, amscd, enumerate, mathrsfs}
\usepackage{graphicx, hhline}
\usepackage[all]{xy}
\usepackage[dvipdfmx]{hyperref}

\title[Cone and contraction theorem]
{Cone and contraction theorem for projective morphisms 
between complex analytic spaces}
\author{Osamu Fujino}
\date{2023/8/12, version 0.13}
\address{Department of 
Mathematics, Graduate School of Science, 
Kyoto University, Kyoto 606-8502, Japan}
\email{fujino@math.kyoto-u.ac.jp}

\DeclareMathOperator{\Ass}{Ass}
\DeclareMathOperator{\xIm}{Im}
\DeclareMathOperator{\Supp}{Supp}
\DeclareMathOperator{\Nqlc}{Nqlc}
\DeclareMathOperator{\NE}{\overline{NE}}

\DeclareMathOperator{\WDiv}{WDiv}
\DeclareMathOperator{\Exc}{Exc}
\DeclareMathOperator{\Pic}{Pic}
\DeclareMathOperator{\Hom}{Hom}
\DeclareMathOperator{\Bs}{Bs}
\DeclareMathOperator{\Projan}{Projan}
\DeclareMathOperator{\Coker}{Coker}
\DeclareMathOperator{\Nlc}{Nlc}
\DeclareMathOperator{\NLC}{NLC}
\DeclareMathOperator{\Ker}{Ker}
\newtheorem{thm}{Theorem}[section]
\newtheorem{lem}[thm]{Lemma}
\newtheorem{cor}[thm]{Corollary}

\newtheorem{conj}[thm]{Conjecture}
\newtheorem{cla}{Claim}
\newtheorem*{claim}{Claim}

\theoremstyle{definition}
\newtheorem{step}{Step}
\newtheorem{defn}[thm]{Definition}
\newtheorem{rem}[thm]{Remark}
\newtheorem{ex}[thm]{Example}
\newtheorem*{ack}{Acknowledgments}  
\newtheorem{say}[thm]{}

\makeatletter
    
    \@addtoreset{equation}{section}
\makeatother

\begin{document}

\maketitle 

\begin{abstract} 
We discuss the cone and contraction theorem 
in a suitable complex analytic setting. 
More precisely, we establish the cone and contraction theorem 
of normal pairs for projective morphisms between 
complex analytic spaces. 
This result is a starting point of the minimal model program 
for complex analytic log canonical pairs. 
In this paper, we are mainly interested in normal pairs 
whose singularities are worse than kawamata log terminal 
singularities. 
\end{abstract}

\tableofcontents

\section{Introduction}
In his epoch-making paper \cite{mori}, Shigefumi Mori 
established the cone theorem for smooth projective varieties 
defined over any algebraically closed field $k$ of arbitrary 
characteristic by his ingenious method of {\em{bend and break}}. 
Then he established the contraction theorem for 
smooth projective threefolds when the characteristic of the 
base field $k$ is zero. 
After that, in characteristic zero, 
the cone and contraction theorem was generalized 
for so-called log-terminal pairs in any dimension by 
using Hironaka's resolution of singularities 
and the Kawamata--Viehweg vanishing theorem. 
For the details, see \cite{kmm}, \cite{kollar-mori} and 
references therein. Now we know that, in characteristic zero,  
the cone and contraction theorem holds 
for more general settings (see \cite{fujino-fundamental}, 
\cite[Chapter 6]{fujino-foundations}, and references therein). 
In this paper, we will discuss the cone 
and contraction theorem of normal pairs 
for projective morphisms between complex 
analytic spaces. 
For kawamata log terminal pairs, it was 
known and has played an important 
role in \cite{nakayama1}, \cite{nakayama2}, and \cite{fujino-minimal}. 
In \cite{fujino-minimal}, we have already discussed 
the minimal model program for kawamata log 
terminal pairs in a complex analytic setting 
(see also \cite{fujino-inversion-pja}). 
Roughly speaking, we showed that \cite{bchm} 
and \cite{hacon-mckernan} 
can work for projective morphisms between complex analytic spaces. 
We note that the Kawamata--Viehweg vanishing theorem 
can be formulated and proved 
for projective morphisms of complex 
analytic spaces and 
is sufficient for the study of kawamata log terminal pairs. 
We also note that $L^2$-methods can work for kawamata 
log terminal pairs. 
For an alternative approach to the minimal 
model program of kawamata log terminal 
pairs for projective morphisms between 
complex analytic spaces, see \cite{dhp}, which 
uses the idea of \cite{cascini-lazic}. 
The reader can find a new approach to the relative minimal 
model program in \cite{lyu-murayama}, 
which can work in larger categories of spaces. 
In \cite{fujino-analytic-vanishing}, 
we established some vanishing theorems and related results 
necessary for the study of complex analytic log canonical pairs and 
quasi-log structures on complex analytic spaces. 
Note that \cite{fujino-analytic-vanishing} 
depends on Morihiko Saito's theory of mixed Hodge modules 
(see \cite{saito1}, 
\cite{saito2}, \cite{saito3}, 
\cite{fujino-fujisawa-saito}, 
and \cite{saito4}) 
and Takegoshi's generalization of Koll\'ar's torsion-free and vanishing 
theorem (see \cite{takegoshi}). 
For an approach without using the theory of mixed Hodge modules, 
see \cite{fujino-fujisawa}. 
In this paper, we will discuss the cone and contraction 
theorem of normal pairs for projective morphisms 
between complex analytic spaces as an 
application of \cite{fujino-analytic-vanishing}. 
This paper can be seen as a complex analytic generalization 
of \cite{fujino-fundamental} and as a generalization of 
Nakayama's paper \cite{nakayama1}. We note that 
Nakayama only treated kawamata log terminal pairs and 
$\mathbb Q$-divisors in \cite{nakayama1}. 
Finally, this paper is independent of 
\cite{fujino-minimal} and does not use any results obtained in 
\cite{fujino-minimal}.  

\begin{say}[Standard setting]\label{c-say1.1}
One of the main difficulties to discuss the 
minimal model theory for complex analytic spaces is how to formulate it. 

Let $\pi\colon X\to Y$ be a projective morphism 
between complex analytic spaces such that 
$X$ is a normal complex variety and let $W$ be 
a compact subset of $Y$. 
In this paper, we formulate and prove almost everything 
over some open neighborhood 
of $W$. 
Let $\Delta$ be an $\mathbb R$-divisor 
on $X$ such that 
$K_X+\Delta$ is $\mathbb R$-Cartier. 
The number of the irreducible components 
of $\Supp \Delta$ is only locally finite. 
In general, the support of $\Delta$ may have infinitely many irreducible 
components. 
By shrinking $Y$ around $W$ suitably, that is, 
by replacing $Y$ with a suitable relatively compact 
open neighborhood of $W$, 
we can always assume that $\Supp \Delta$ has only finitely many 
irreducible components. Moreover, 
we can assume that 
a given $\mathbb R$-Cartier divisor 
on $X$ is a finite $\mathbb R$-linear combination of 
Cartier divisors. 
Therefore, by considering some relatively compact open 
neighborhood of $W$, 
we can avoid subtle problems caused by 
the difference between the Zariski topology and 
the Euclidean topology. In \cite{fujino-minimal}, we almost always 
assume that $W$ is a Stein compact subset of $Y$ 
such that $\Gamma (W, \mathcal O_Y)$ is noetherian. 
In this paper, however, we usually assume that $W$ is only 
a compact subset of $Y$. 
When we consider the 
Kleiman--Mori cone $\NE(X/Y; W)$ of $\pi\colon X\to Y$ 
and 
$W$, 
we further assume that the dimension of 
$N^1(X/Y; W)$ is finite. 
For the details of $\NE(X/Y; W)$ and $N^1(X/Y; W)$, 
see Section \ref{c-sec11}. 
Note that if 
$W\cap V$ has only finitely many connected 
components for any analytic subset $V$ which 
is defined over an open neighborhood of $W$ 
then the dimension of $N^1(X/Y; W)$ is finite by 
Nakayama's finiteness (see Theorem \ref{c-thm11.10}). 
Therefore, if $W$ is a compact semianalytic subset 
of $Y$, then the dimension of $N^1(X/Y; W)$ is always finite.  
Thus, we can find many compact subsets $W$ with 
$\dim _{\mathbb R}N^1(X/Y; W)<\infty$. 
\end{say}

\subsection{Main theorem}\label{c-subsec1.1} 
In this paper, we call $(X, \Delta)$ a {\em{normal pair}} 
if it consists of a normal complex variety $X$ and an effective 
$\mathbb R$-divisor $\Delta$ on $X$ such that 
$K_X+\Delta$ is $\mathbb R$-Cartier. 
The main purpose of this paper is to establish the 
following cone and contraction theorem of 
normal pairs for projective morphisms between complex analytic 
spaces. 

\begin{thm}[{Cone and contraction theorem, 
see Theorems \ref{c-thm12.1}, \ref{c-thm12.2}, \ref{c-thm13.2}, and 
\ref{c-thm14.4}}] 
\label{c-thm1.2}
Let $\pi \colon X\to Y$ be a projective 
morphism of complex 
analytic spaces such that $X$ is a normal complex variety and 
let $W$ be a compact subset of $Y$. 
Assume that the dimension of $N^1(X/Y; W)$ is finite. 
Let $\Delta$ be an effective $\mathbb R$-divisor on $X$ 
such that $K_X+\Delta$ is $\mathbb R$-Cartier. 
Then we have 
\begin{equation*}
\NE(X/Y; W)=\NE(X/Y; W)_{(K_X+\Delta)\geq 0} 
+\NE(X/Y; W)_{\Nlc (X, \Delta)}+\sum R_j
\end{equation*} 
with the following properties.  
\begin{itemize}
\item[(1)] $\Nlc (X, \Delta)$ is the non-lc locus 
of $(X, \Delta)$ and 
$\NE(X/Y; W)_{\Nlc(X, \Delta)}$ is the subcone 
of $\NE(X/Y; W)$ which is 
the closure of the convex cone spanned by 
the projective integral curves $C$ on $\Nlc(X, \Delta)$ 
such that $\pi(C)$ is a point of $W$. 
\item[(2)] 
$R_j$ is a $(K_X+\Delta)$-negative 
extremal ray of $\NE(X/Y; W)$ which satisfies  
\begin{equation*}
R_j\cap \NE(X/Y; W)_{\Nlc (X, \Delta)}=\{0\}
\end{equation*} for 
every $j$. 
\item[(3)] Let $\mathcal A$ be a $\pi$-ample $\mathbb R$-line bundle 
on $X$. 
Then there are only finitely many $R_j$'s included in 
$\NE(X/Y; W)_{(K_X+\Delta+\mathcal A)<0}$. In particular, 
the $R_j$'s are discrete in the half-space 
$\NE(X/Y; W)_{(K_X+\Delta)<0}$. 
\item[(4)] 
Let $F$ be any face of $\NE(X/Y; W)$ such 
that 
\begin{equation*} 
F\cap \left(\NE(X/Y; W)_{(K_X+\Delta)\geq 0}+
\NE(X/Y; W)_{\Nlc (X, \Delta)}\right)=\{0\}. 
\end{equation*}
Then, after shrinking $Y$ around 
$W$ suitably, there exists a contraction morphism 
$\varphi_F\colon X\to Z$ 
over $Y$ satisfying the following properties. 
\begin{itemize}
\item[(i)] 
Let $C$ be a projective integral curve on $X$ such that 
$\pi(C)$ is a point of $W$. 
Then $\varphi_F(C)$ is a point if and 
only if the numerical equivalence class 
$[C]$ of $C$ is in $F$. 
\item[(ii)] The natural map 
$\mathcal O_Z\to (\varphi_F)_*\mathcal O_X$ is an 
isomorphism.  
\item[(iii)] Let $\mathcal L$ be a line bundle on $X$ such 
that $\mathcal L\cdot C=0$ for 
every curve $C$ with $[C]\in F$. 
Then, after shrinking $Y$ around 
$W$ suitably again, 
there exists a line bundle $\mathcal L_Z$ on $Z$ such 
that $\mathcal L\simeq \varphi^*_F\mathcal L_Z$ holds. 
\end{itemize}
\item[(5)] 
Every $(K_X+\Delta)$-negative extremal 
ray $R$ with 
\begin{equation*}
R\cap \NE(X/Y; W)_{\Nlc(X, \Delta)}=\{0\}
\end{equation*} 
is spanned by a {\em{(}}possibly singular{\em{)}} rational 
curve $C$ with 
\begin{equation*}
0<-(K_X+\Delta)\cdot C\leq 2\dim X. 
\end{equation*} 
\end{itemize}

From now on, we further assume that 
$(X, \Delta)$ is log canonical, 
equivalently, $\Nlc (X, \Delta)=\emptyset$. 
Then we have the following properties. 
\begin{itemize}
\item[(6)] 
Let $H$ be an effective $\mathbb R$-Cartier 
$\mathbb R$-divisor on $X$ such that 
$K_X+\Delta+H$ is $\pi$-nef over $W$ and $(X, \Delta+H)$ 
is log canonical. 
Then, either $K_X+\Delta$ is also $\pi$-nef over $W$ or there 
exists a $(K_X+\Delta)$-negative 
extremal ray $R$ of $\NE(X/Y; W)$ such that 
\begin{equation*}
(K_X+\Delta+\lambda H)\cdot R=0, 
\end{equation*} 
where 
\begin{equation*}
\lambda:=\inf\{ t\geq 0 \, |\,  K_X+\Delta+
tH \ {\text{is $\pi$-nef over $W$}} \}. 
\end{equation*} 
Of course, $K_X+\Delta+\lambda H$ is $\pi$-nef over $W$.  
\end{itemize}
Similarly, we have: 
\begin{itemize}
\item[(7)] 
Let $\mathcal H$ be an $\mathbb R$-line bundle  
on $X$ which is $\pi$-ample over $W$ such that 
$K_X+\Delta+\mathcal H$ is $\pi$-nef over $W$. 
Then, either $K_X+\Delta$ is also $\pi$-nef over $W$ or there 
exists a $(K_X+\Delta)$-negative 
extremal ray $R$ of $\NE(X/Y; W)$ such that 
\begin{equation*}
(K_X+\Delta+\lambda \mathcal H)\cdot R=0, 
\end{equation*} 
where 
\begin{equation*}
\lambda:=\inf\{ t\geq 0 \, |\,  K_X+\Delta+
t\mathcal H \ {\text{is $\pi$-nef over $W$}} \}. 
\end{equation*} 
Note that $K_X+\Delta+\lambda \mathcal H$ is $\pi$-nef over $W$.  
\end{itemize}
\end{thm}

\begin{rem}\label{c-rem1.3}
In Theorem \ref{c-thm1.2}, the proof of (5) needs Mori's bend and break method and 
(6) is an application of (5). On the other hand, (7) is an easy consequence of (3). 
Note that $\pi$-very ample line bundles do not always have global sections. 
Hence (7) is not a special case of (6). 
We need (7) in order to discuss the minimal 
model program of log canonical pairs with ample scaling 
for projective morphisms between complex analytic spaces. 
\end{rem}

For the minimal model program, the following theorem, 
which is a supplement to Theorem \ref{c-thm1.2}, may be useful 
(see \cite{fujino-minimal}). 

\begin{thm}\label{c-thm1.4} 
Let $(X, \Delta)$ be a log canonical pair. 
Let $\pi\colon X\to Y$ be a projective morphism 
of complex analytic spaces and let $W$ be a compact 
subset of $Y$ such that the dimension of $N^1(X/Y; W)$ is 
finite. 
Suppose that $\pi\colon X\to Y$ is decomposed as 
\begin{equation*}
\xymatrix{
\pi\colon X\ar[r]^-f& Y^\flat \ar[r]^-g& Y
}
\end{equation*}
such that $Y^\flat$ is projective over $Y$. 
Let $\mathcal A_{Y^\flat}$ be a $g$-ample line bundle on $Y^\flat$. 
Let $R$ be a $\left(K_X+\Delta+(\dim X+1)
f^*\mathcal A_{Y^\flat}\right)$-negative 
extremal ray of $\NE(X/Y; W)$. 
Then $R$ is a $(K_X+\Delta)$-negative extremal ray of $\NE\left(X/{Y^\flat}; 
g^{-1}(W)\right)$, that is, $R\cdot f^*\mathcal A_{Y^\flat}=0$. 
\end{thm}

We prove Theorem \ref{c-thm1.4} as an application of 
the vanishing theorem for projective quasi-log schemes. 
We do not need Theorem \ref{c-thm1.2} (5) for the proof of 
Theorem \ref{c-thm1.4}. We have the following result as an easy 
consequence of Theorem \ref{c-thm1.4}. 

\begin{cor}[{see \cite[Corollary 1.2]{fujino-relative}}]\label{c-cor1.5}
Let $(X, \Delta)$ be a log canonical pair. 
Let $\pi\colon X\to Y$ be a projective morphism 
of complex analytic spaces 
and let $\mathcal A$ be any $\pi$-ample line bundle 
on $X$. Then $K_X+\Delta+(\dim X+1)\mathcal A$ is 
always nef over $Y$. 
\end{cor}

We make an important remark on Theorem \ref{c-thm1.2}. 
By Remark \ref{c-rem1.6}, we see that 
the cone and contraction theorem of normal pairs holds 
for projective morphisms between compact analytic 
spaces. 

\begin{rem}\label{c-rem1.6} 
Let $\pi\colon X\to Y$ be a projective morphism 
of complex analytic spaces and let $W$ be 
a compact subset of $Y$ as in Theorem \ref{c-thm1.2}. 
Then the dimension of $N^1(X/Y; W)$ 
is not always finite (see Example \ref{c-ex11.9}). 
In \cite[Chapter II.~5.19.~Lemma]{nakayama2} 
(see Theorem \ref{c-thm11.10}), 
Noboru Nakayama proved that 
if 
\begin{itemize}
\item\label{c-rem-item} $W$ is a compact subset of $Y$ 
such that $W\cap V$ has only finitely many 
connected components 
for any analytic subset $V$ which is defined 
over an open neighborhood of $W$, 
\end{itemize} 
then 
the dimension of $N^1(X/Y; W)$ is finite. 
We note that the above assumption is satisfied in the following 
cases: 
\begin{itemize}
\item[(i)] $W$ is a point of $Y$. 
\item[(ii)] $W$ is a compact semianalytic 
subset of $Y$. 
\item[(iii)] $W=Y$ when $Y$ is compact. 
\end{itemize}
Case (i) is obvious. 
In Case (ii), 
$W\cap V$ is a compact semianalytic subset of $Y$. 
Thus we see that $W\cap V$ has only finitely many 
connected components (see, for example, 
\cite[Corollary 2.7 (2)]{bierstone-milman1}). 
In Case (iii), $W\cap V=V$ is a compact analytic 
subset of $Y$. 
Hence it has only finitely many connected components.  
\end{rem}

By Remark \ref{c-rem1.6}, we see that 
there are many compact subsets $W$ of $Y$ 
such that the dimension of $N^1(X/Y; W)$ is finite. 

We note that we can formulate and prove the basepoint-free 
theorem for projective morphisms of complex 
analytic spaces as follows. 
In Theorem \ref{c-thm1.7}, $L$ is only assumed to be 
$\pi$-nef over $W$, that is, $L|_{\pi^{-1}(w)}$ is nef in the usual sense 
for every $w\in W$. Equivalently, $L\cdot C\geq 0$ for every 
projective integral curve $C$ on $X$ such that 
$\pi(C)$ is a point of $W$. However, Theorem 
\ref{c-thm1.7} claims that it is $\pi$-semiample 
over some open neighborhood of $W$. 

\begin{thm}[Basepoint-free theorem:~Theorem \ref{c-thm9.1}]\label{c-thm1.7}
Let $\pi\colon X\to Y$ be a projective morphism 
of complex analytic spaces such that 
$X$ is a normal complex variety and let $W$ be a compact 
subset of $Y$. 
Let $\Delta$ be an effective 
$\mathbb R$-divisor on $X$ 
such that $K_X+\Delta$ is 
$\mathbb R$-Cartier. 
Let $L$ be a Cartier divisor on $X$ 
which is $\pi$-nef over 
$W$. 
We assume that 
\begin{itemize}
\item[(i)] $aL-(K_X+\Delta)$ is $\pi$-ample 
over $W$ for 
some positive real number $a$, and 
\item[(ii)] $\mathcal O_{\Nlc(X, \Delta)}(mL)$ is 
$\pi|_{\Nlc(X, \Delta)}$-generated 
over some open neighborhood of $W$ for every $m\gg 0$. 
\end{itemize} 
Then there exists a relatively compact open neighborhood 
$U$ of $W$ such that 
$\mathcal O_X(mL)$ is $\pi$-generated 
over $U$ for 
every $m\gg 0$. 
\end{thm}

In Theorem \ref{c-thm1.7}, $W$ is only assumed to be 
a compact subset of $Y$. 
We do not need the assumption that $\dim_{\mathbb R} 
N^1(X/Y; W)<\infty$ holds. 
When $(X, \Delta)$ is log canonical, we will also prove 
the basepoint-free theorem for $\mathbb R$-Cartier divisors 
(see Theorem \ref{c-thm15.1}). 
In Theorem \ref{c-thm15.1}, we have to assume that 
the dimension of $N^1(X/Y; W)$ is finite since we need the 
cone theorem for the proof of Theorem \ref{c-thm15.1}. 

In the proof of Theorems \ref{c-thm1.2}, 
\ref{c-thm1.7}, and so on, the following basic properties 
of log canonical centers play an important role. 

\begin{thm}[Basic properties of log canonical 
centers:~Theorem \ref{c-thm7.1}]\label{c-thm1.8}
Let $(X, \Delta)$ be a log canonical pair. 
Then the following properties hold. 
\begin{itemize}
\item[(1)] The number of log canonical centers 
of $(X, \Delta)$ is locally finite. 
\item[(2)] The intersection of two log canonical centers 
is a union of some log canonical centers. 
\item[(3)] Let $x\in X$ be any point such that $(X, \Delta)$ 
is log canonical but is not kawamata log terminal at $x$. 
Then there exists a unique minimal {\em{(}}with respect to the 
inclusion{\em{)}} log canonical center $C_x$ passing through $x$. 
Moreover, $C_x$ is normal at $x$. 
\end{itemize}
\end{thm}

Theorem \ref{c-thm1.8} is new for complex analytic log canonical 
pairs although it is well known when $(X, \Delta)$ is algebraic. 
It will be useful for the study of complex analytic 
log canonical singularities (see also \cite{fujino-acc}). 

Theorem \ref{c-thm1.2} is a starting point of the 
minimal model program of log canonical pairs 
for projective morphisms between complex analytic 
spaces. We can formulate the minimal 
model theory of log canonical pairs for projective 
morphisms between complex analytic spaces by 
using Theorem \ref{c-thm1.2} as in the algebraic case. 
On the other hand, 
one of the main goals of the minimal model 
theory for projective morphisms between complex analytic spaces is the 
following conjecture. 

\begin{conj}[Finite generation]\label{c-conj1.9}
Let $\pi\colon X\to Y$ be a projective morphism of 
complex analytic spaces and let $\Delta$ be an 
effective $\mathbb Q$-divisor on $X$ 
such that $(X, \Delta)$ is log canonical. 
Then 
\begin{equation*}
R(X/Y, K_X+\Delta):=\bigoplus _{m\in \mathbb N}\pi_*\mathcal O_X
(\lfloor m(K_X+\Delta)\rfloor)
\end{equation*} 
is a locally finitely generated graded $\mathcal O_Y$-algebra. 
\end{conj}

We note that in \cite{fujino-minimal} Conjecture 
\ref{c-conj1.9} was already solved completely when $(X, \Delta)$ 
is kawamata log terminal. 
We also note that Conjecture \ref{c-conj1.9} is still widely open 
even when $\pi\colon X\to Y$ is algebraic (see \cite{fujino-gongyo}). 

\medskip 

The author first prepared a short manuscript which only 
explains how to modify arguments in \cite{fujino-fundamental}. 
Unfortunately, however, it seemed to be hard to read. 
Hence he made great efforts to make this paper as self-contained as 
possible except for the results established 
in \cite{fujino-analytic-vanishing} (see also \cite{fujino-vanishing-pja}). 
He sometimes repeats arguments in \cite{fujino-fundamental} and 
\cite{fujino-foundations}.  
Thus, some parts of this paper are very similar to 
those of \cite{fujino-fundamental} and \cite{fujino-foundations}. 

\begin{rem}[Quasi-log structures]\label{c-rem1.10}
By \cite[Theorems 1.1 and 1.2]{fujino-analytic-vanishing} (see 
Theorems \ref{c-thm5.5} and \ref{c-thm5.7}), 
we can formulate and discuss quasi-log structures on complex 
analytic spaces (see \cite[Chapter 6]{fujino-foundations}). 
Hence we can 
establish the cone and contraction theorem 
for highly singular complex analytic spaces. 
However, in this paper, we will only discuss the cone 
and contraction theorem of normal 
pairs (see Theorem \ref{c-thm1.2}). 
This is because Theorem \ref{c-thm1.2} is sufficient 
for many geometric applications and it is not so easy psychologically 
to treat reducible complex analytic spaces. 
We will describe the theory of quasi-log complex analytic spaces 
in \cite{fujino-quasi-log-analytic}. 
\end{rem}

We look at the organization of this paper. In Section \ref{c-sec2}, 
we collect some necessary definitions and results for the reader's 
convenience. Since we have to work in the complex analytic setting, 
some of them become much more subtle than the usual ones in the 
algebraic setting. 
In Section \ref{c-sec3}, 
we collect some basic properties of 
relatively nef and relatively 
ample $\mathbb R$-line bundles for the sake of completeness. 
They are indispensable in subsequent sections. 
In Section \ref{c-sec4}, we define non-lc ideal sheaves 
in the complex analytic setting and prove some elementary lemmas. 
In Section \ref{c-sec5}, we quickly recall the main result of 
\cite{fujino-analytic-vanishing} without proof. 
Note that the proof of the main result in \cite{fujino-analytic-vanishing} 
depends on Saito's theory of mixed Hodge modules and 
Takegoshi's generalization of Koll\'ar's 
torsion-free and vanishing theorem. 
However, we can make the main result of \cite{fujino-analytic-vanishing} 
free from the theory of mixed Hodge modules by \cite{fujino-fujisawa}. 
In Section \ref{c-sec6}, we prepare some necessary vanishing theorems 
as applications of the vanishing result explained in 
Section \ref{c-sec5}. 
In Section \ref{c-sec7}, we establish the basic properties of 
log canonical centers. They are new and very important in the theory 
of minimal models in the complex analytic setting. 
In Section \ref{c-sec8}, 
we prove the non-vanishing theorem 
in the complex analytic setting with the aid of the theory of 
quasi-log schemes. 
Note that Lemma \ref{c-lem8.2} is new and will be useful for the study of 
quasi-log structures. 
In Section \ref{c-sec9}, we establish the basepoint-free theorem 
for normal pairs in the complex analytic setting by using the 
non-vanishing theorem proved in Section \ref{c-sec8}. 
It is well known and is not difficult to prove for kawamata log 
terminal pairs. 
In Sections \ref{c-sec10}, 
we prove the rationality theorem 
for normal pairs in the complex analytic setting. 
The proof is essentially the same as 
the one for algebraic varieties explained in \cite{fujino-fundamental}. 
In Section \ref{c-sec11}, we define Kleiman--Mori cones 
for projective morphisms of complex analytic spaces. 
In Subsection \ref{c-subsec11.1}, we briefly explain Nakayama's finiteness 
without proof for 
the reader's convenience. Note that in this paper 
we do not need it except in the proof of Corollary \ref{c-cor1.5}. 
Then, in Section \ref{c-sec12}, we prove the cone theorem 
for normal pairs in the complex analytic setting. 
The results in Section \ref{c-sec12} are easy consequences of the 
basepoint-free theorem in Section \ref{c-sec9} and the rationality theorem 
in Section \ref{c-sec10}. 
In Subsection \ref{c-subsec12.1}, we prove Theorem \ref{c-thm1.4} 
as an easy application of the vanishing theorem for 
projective quasi-log schemes. 
In Section \ref{c-sec13}, we discuss lengths of extremal rational curves. 
The result in Section \ref{c-sec13} seems to be indispensable for 
the minimal model program with scaling. 
Here, we use the framework of quasi-log schemes. 
In Section \ref{c-sec14}, we discuss 
Shokurov's polytopes and some applications. 
The results in this section are well known and have 
already played an important role in the usual algebraic setting. 
In Section \ref{c-sec15}, 
we prove the basepoint-free theorem of log canonical 
pairs for $\mathbb R$-Cartier 
divisors. 
It can be seen as an application of the cone theorem. 
In Section \ref{c-sec16}, which is the final section, 
we prove the main result of this paper, that is, 
the cone and contraction theorem of 
normal pairs for projective morphisms between 
complex analytic spaces:~Theorem \ref{c-thm1.2}. 

\begin{ack}\label{c-ack}
The author was partially 
supported by JSPS KAKENHI Grant Numbers 
JP19H01787, JP20H00111, JP21H00974, JP21H04994. 
He thanks Professor Noboru Nakayama very much for useful 
suggestions and 
answering his questions. 
He also would like to thank Professors Taro Fujisawa,  
Shigefumi Mori, and Morihiko Saito very much. 
\end{ack}

In this paper, every complex analytic space is assumed to be 
{\em{Hausdorff}} and {\em{second-countable}}. 
An irreducible and reduced complex analytic space 
is called a {\em{complex variety}}. 
We will freely use the standard notation in \cite{fujino-fundamental}, 
\cite{fujino-foundations}, \cite{fujino-minimal}, and so on. 
We will also freely use the basic results on complex analytic geometry 
in \cite{banica} and \cite{fischer}. 
For the minimal model program for projective 
morphisms between complex analytic spaces, 
see \cite{nakayama1}, \cite{nakayama2}, and \cite{fujino-minimal}. 
For the traditional framework of 
the minimal model program, see \cite{kmm} and \cite{kollar-mori}. 
We note that $\mathbb Z$, $\mathbb Q$, 
and $\mathbb R$ denote the set of {\em{integers}}, 
{\em{rational numbers}}, 
and {\em{real numbers}}, respectively. We also note 
that $\mathbb N$ (resp.~$\mathbb Z_{>0}$) is the set of 
{\em{non-negative integers}} (resp.~{\em{positive integers}}). 

\section{Preliminaries}\label{c-sec2}

In this section, we collect basic definitions and results necessary for 
this paper. For the details, see \cite{fujino-fundamental}, 
\cite{fujino-foundations}, \cite{fujino-minimal}, and so on. 
Since we are working in the complex analytic setting, 
some of them become subtle. 

Let us start with the definition of {\em{singularities of pairs}}, which is 
indispensable in the theory of minimal models. 

\begin{say}[Singularities of pairs, log canonical 
centers, and non-lc loci]\label{c-say2.1}
We consider a normal complex variety 
$X$. Let $X_{\mathrm{sm}}$ denote the smooth locus 
of $X$.  
Then the {\em{canonical 
sheaf}} $\omega_X$ of $X$ is the unique reflexive sheaf 
whose restriction to $X_{\mathrm{sm}}$ is isomorphic 
to the sheaf $\Omega^n_{X_{\mathrm{sm}}}$, 
where $n=\dim X$. 
Let $\Delta$ be an $\mathbb R$-divisor on $X$, 
that is, $\Delta$ is a locally finite $\mathbb R$-linear combination 
of prime divisors on $X$. 
We say that $K_X+\Delta$ is {\em{$\mathbb R$-Cartier}} 
at $x\in X$ if there exist an open neighborhood 
$U_x$ of $x$ and 
a Weil divisor $K_{U_x}$ on $U_x$ with $\mathcal O_{U_x}
(K_{U_x})\simeq \omega_X|_{U_x}$ such that 
$K_{U_x}+\Delta|_{U_x}$ is $\mathbb R$-Cartier, 
that is, $K_{U_x}+\Delta|_{U_x}$ is 
a finite $\mathbb R$-linear combination of Cartier divisors on 
$U_x$. For any subset $L$ of $X$, 
we say that $K_X+\Delta$ is $\mathbb R$-Cartier at $L$ if 
it is $\mathbb R$-Cartier at any point $x$ of $L$. 
We simply say that $K_X+\Delta$ is $\mathbb R$-Cartier 
when $K_X+\Delta$ is $\mathbb R$-Cartier at any point 
$x\in X$. Unfortunately, however, we can not always define 
$K_X$ globally with $\mathcal O_X(K_X)\simeq 
\omega_X$. 
In general, it only exists locally on $X$. We usually use the 
symbol $K_X$ as a formal divisor class 
with an isomorphism 
$\mathcal O_X(K_X)\simeq \omega_X$ and call it 
the {\em{canonical divisor}} of $X$ if there is no danger of 
confusion. 

Let $f\colon Y\to X$ be a proper bimeromorphic morphism 
from a normal complex variety $Y$. 
Suppose that $K_X+\Delta$ is $\mathbb R$-Cartier 
in the above sense. 
We take a small Stein open subset $U$ of $X$ where 
$K_U+\Delta|_U$ is a well-defined 
$\mathbb R$-Cartier $\mathbb R$-divisor on $U$. 
In this situation, we can define $K_{f^{-1}(U)}$ and $K_U$ such that 
$f_*K_{f^{-1}(U)}=K_U$. 
Then we can write 
\begin{equation*}
K_{f^{-1}(U)}=f^*(K_U+ 
\Delta|_U)+E_U
\end{equation*}
as usual. Note that $E_U$ is a well-defined 
$\mathbb R$-divisor on $f^{-1}(U)$ such that 
$f_*E_U=\Delta|_U$. 
Then we have the following formula 
\begin{equation*}
K_Y=f^*(K_X+\Delta)+
\sum _Ea(E, X, \Delta)E
\end{equation*} 
as in the algebraic case. 
We note that $\sum _Ea(E, X, \Delta)E$ is a globally well-defined 
$\mathbb R$-divisor on $Y$ such that 
$\left(\sum_E a(E, X, \Delta)E\right)|_{f^{-1}(U)}=E_U$ 
although $K_X$ and $K_Y$ are well defined 
only locally. 

If $\Delta$ is a {\em{boundary $\mathbb R$-divisor}}, that is, 
all the coefficients of $\Delta$ are in $[0, 1]\cap \mathbb R$, and 
$a(E, X, \Delta)\geq -1$ holds for any 
$f\colon Y\to X$ and every $f$-exceptional divisor $E$, 
then $(X, \Delta)$ is called a {\em{log canonical}} pair. 
If $(X, \Delta)$ is log canonical and 
$a(E, X, \Delta)>-1$ for any $f\colon Y\to X$ and 
every $f$-exceptional divisor $E$, 
then $(X, \Delta)$ is called a {\em{purely log terminal}} pair. 
If $(X, \Delta)$ is purely log terminal and $\lfloor 
\Delta\rfloor =0$, that is, 
the coefficients of $\Delta$ are in 
$[0, 1)\cap \mathbb R$, then $(X, \Delta)$ is called a 
{\em{kawamata log terminal}} pair. 
When $\Delta=0$ and $a(E, X, 0)\geq 0$ (resp.~$>0$) 
for 
any $f\colon Y\to X$ and 
every $f$-exceptional divisor $E$, 
we simply say that $X$ has only {\em{canonical 
singularities}} (resp.~{\em{terminal singularities}}). 
In this paper, we will only use log canonical pairs and kawamata 
log terminal pairs. 

More generally, let $X$ be a normal complex variety and let $\Delta$ be an 
effective $\mathbb R$-divisor 
on $X$. 
We say that $(X, \Delta)$ is {\em{log canonical}} 
(resp.~{\em{kawamata log terminal}}) {\em{at $x\in X$}} 
if there exists an open neighborhood $U_x$ of $x$ such that 
$(U_x, \Delta|_{U_x})$ is a log canonical 
pair (resp.~kawamata log terminal pair). 
Let $L$ be any subset of $X$. 
We say that $(X, \Delta)$ is {\em{log canonical}} 
(resp.~{\em{kawamata log terminal}}) {\em{at $L$}} if 
$(X, \Delta)$ is log canonical 
(resp.~kawamata log terminal) at any point $x$ of $L$. 
We note that $(X, \Delta)$ is log 
canonical (resp.~kawamata log terminal) in the 
above sense if and only if 
$(X, \Delta)$ is log canonical 
(resp.~kawamata log terminal) at any point $x$ of $X$. 

Let $X$ be a normal complex variety and let $\Delta$ be an 
effective $\mathbb R$-divisor on $X$ such that 
$K_X+\Delta$ is $\mathbb R$-Cartier. 
The image of $E$ with $a(E, X, \Delta)= -1$ for some $f\colon Y\to X$ 
such that 
$(X, \Delta)$ is log canonical around general points of $f(E)$ is 
called a {\em{log canonical center}} of $(X, \Delta)$. 
The {\em{non-lc locus}} of $(X, \Delta)$, 
denoted by $\Nlc(X, \Delta)$, is the smallest closed 
subset $Z$ of $X$ such that the complement 
$(X\setminus Z, \Delta|_{X\setminus Z})$ is log canonical. 
We can define a 
natural complex analytic space structure 
on $\Nlc(X, \Delta)$ by the non-lc ideal 
sheaf $\mathcal J_{\NLC}(X, \Delta)$ of $(X, \Delta)$. 
For the definition of $\mathcal J_{\NLC}(X, \Delta)$, see 
Section \ref{c-sec4} below. 

\medskip 

The above definition is compatible with the usual definition 
for algebraic varieties. 

\begin{rem}\label{c-rem2.2}
Let $(X, \Delta)$ be a pair consisting of a normal algebraic variety $X$ and 
an effective $\mathbb R$-divisor on $X$ such that $K_X+\Delta$ is 
$\mathbb R$-Cartier. 
Then $(X, \Delta)$ is kawamata log terminal (resp.~log canonical) 
in the usual sense (see 
\cite{fujino-fundamental}, 
\cite{fujino-foundations}, and so on) 
if and only if $(X^{\mathrm{an}}, \Delta^{\mathrm{an}})$ 
is kawamata log terminal (resp.~log canonical) in the above sense, 
where $X^{\mathrm{an}}$ is the complex analytic space 
naturally associated to $X$ and let $\Delta^{\mathrm{an}}$ be 
the $\mathbb R$-divisor on $X^{\mathrm{an}}$ associated to $\Delta$. 
\end{rem}

The following lemma is well known for 
algebraic varieties. 

\begin{lem}\label{c-lem2.3}
Let $X$ be a normal complex variety and let $\Delta$ be 
an effective $\mathbb R$-divisor 
on $X$ such that $K_X+\Delta$ is $\mathbb R$-Cartier. 
Let $P$ be a point of $X$ and let $D_i$ be an effective 
Cartier divisor on $X$ with $P\in \Supp D_i$ for every $i$. 
If $\left(X, \Delta+\sum _{i=1}^k D_i\right)$ 
is log canonical at $P$, then $k\leq \dim X$ 
holds. 
\end{lem}

We omit the proof of Lemma \ref{c-lem2.3} here since 
the usual proof for algebraic varieties can work without 
any changes (see, for example, \cite[Lemma 13.2]{fujino-fundamental}). 
We will use Lemma \ref{c-lem2.3} in order to create a new log 
canonical center. 
\end{say}

In this paper, we sometimes implicitly use Serre's GAGA. 

\begin{say}[Serre's GAGA]\label{c-say2.4} 
Let $\pi\colon X\to Y$ be a projective morphism 
of complex analytic spaces and let $F$ be a fiber 
of $\pi\colon X\to Y$. 
Then $F$ is projective. 
Hence we can apply various results of projective schemes 
to $F$ with the aid of Serre's GAGA (see \cite{serre}). 
\end{say} 

In the theory of minimal models, we need the notion of 
{\em{$\mathbb R$-line bundles}} 
and {\em{$\mathbb Q$-line bundles}}. 

\begin{say}[Line bundles, 
$\mathbb R$-line bundles, and $\mathbb Q$-line bundles]\label{c-say2.5}
Let $X$ be a complex analytic space and let $\Pic(X)$ 
denote  
the group of line bundles on $X$, that is, the {\em{Picard group}} of $X$. 
An element of $\Pic(X)\otimes_{\mathbb Z}\mathbb R$ 
(resp.~$\Pic(X)\otimes _{\mathbb Z}\mathbb Q$) 
is called an {\em{$\mathbb R$-line bundle}} (resp.~a 
{\em{$\mathbb Q$-line bundle}}) 
on $X$. In this paper, we usually write the group law of 
$\Pic(X)\otimes _{\mathbb Z} \mathbb R$ additively for simplicity of notation. 
Hence we sometimes use $m\mathcal L$ to denote $\mathcal L^{\otimes m}$ 
for $\mathcal L\in \Pic(X)$ and $m\in \mathbb Z$.  
\end{say}

We also need the notion of {\em{$\mathbb R$-divisors}} and 
{\em{$\mathbb Q$-divisors}}. 

\begin{say}[Divisors, $\mathbb R$-divisors, 
and $\mathbb Q$-divisors]\label{c-say2.6}
Let $X$ be a reduced equidimensional complex analytic space. 
A {\em{prime divisor}} on $X$ is an irreducible 
and reduced closed analytic subspace of codimension one. 
An {\em{$\mathbb R$-divisor}} $D$ on $X$ is a formal 
sum 
\begin{equation*}
D=\sum _i a_i D_i, 
\end{equation*} 
where $D_i$ is a prime divisor on $X$ with 
$D_i\ne D_j$ for $i\ne j$, 
$a_i\in \mathbb R$ for every $i$, and the {\em{support}} 
\begin{equation*}
\Supp D:=\bigcup _{a_i\ne 0}D_i
\end{equation*} 
is a closed analytic subset of $X$. 
In other words, the formal sum $\sum _i a_i 
D_i$ is locally finite. 
If $a_i\in \mathbb Z$ (resp.~$a_i\in 
\mathbb Q$) for 
every $i$, then $D$ is called 
a {\em{divisor}} (resp.~{\em{$\mathbb Q$-divisor}}) on $X$. 
Note that a divisor is sometimes called 
an {\em{integral Weil divisor}} in 
order to emphasize the condition that $a_i\in \mathbb Z$ for every $i$. 
If $0\leq a_i\leq 1$ (resp.~$a_i\leq 1$) 
holds for every $i$, then 
an $\mathbb R$-divisor $D$ is called a {\em{boundary}} 
(resp.~{\em{subboundary}}) $\mathbb R$-divisor. 

Let $D=\sum _i a_i D_i$ be an $\mathbb R$-divisor 
on $X$ such that $D_i$ is a prime divisor 
for every $i$ with $D_i\ne D_j$ for $i\ne j$. 
The {\em{round-down}} $\lfloor D\rfloor$ of $D$ is 
defined to be the divisor 
\begin{equation*}
\lfloor D\rfloor =\sum _i \lfloor a_i\rfloor D_i, 
\end{equation*} 
where $\lfloor x\rfloor$ is 
the integer defined by $x-1<\lfloor x\rfloor \leq x$ 
for every real number $x$. 
The {\em{round-up}} and the 
{\em{fractional part}} of $D$ are defined to be 
\begin{equation*}
\lceil D \rceil :=-\lfloor -D\rfloor, \quad 
\text{and} \quad \{D\}:=D-\lfloor D\rfloor, 
\end{equation*} 
respectively. We put 
\begin{equation*}
D^{=1}:=\sum _{a_i=1}D_i, \quad 
D^{<1}:=\sum _{a_i<1} a_i D_i, \quad \text{and} \quad 
D^{>1}:=\sum _{a_i>1}a_i D_i. 
\end{equation*}

Let $D$ be an $\mathbb R$-divisor on $X$ 
and let $x$ be a point of $X$. 
If $D$ is written as a finite $\mathbb R$-linear 
(resp.~$\mathbb Q$-linear) combination of Cartier 
divisors on some open 
neighborhood of $x$, 
then $D$ is said to be {\em{$\mathbb R$-Cartier at $x$}} 
(resp.~{\em{$\mathbb Q$-Cartier at $x$}}). 
If $D$ is $\mathbb R$-Cartier 
(resp.~$\mathbb Q$-Cartier) at $x$ for every $x\in X$, 
then $D$ is said to be {\em{$\mathbb R$-Cartier}} 
(resp.~{\em{$\mathbb Q$-Cartier}}). 
More generally, for any subset $L$ of $X$, 
if $D$ is $\mathbb R$-Cartier (resp.~$\mathbb Q$-Cartier) 
at $x$ for every $x\in L$, then $D$ is said to be 
{\em{$\mathbb R$-Cartier}} (resp.~{\em{$\mathbb Q$-Cartier}}) 
{\em{at}} $L$. 
Note that a $\mathbb Q$-Cartier $\mathbb R$-divisor 
$D$ is automatically a $\mathbb Q$-Cartier 
$\mathbb Q$-divisor by definition. 
If $D$ is a finite $\mathbb R$-linear (resp.~$\mathbb Q$-linear) 
combination of Cartier divisors on $X$, 
then we sometimes say that $D$ 
is a {\em{globally $\mathbb R$-Cartier $\mathbb R$-divisor}} 
(resp.~{\em{globally $\mathbb Q$-Cartier $\mathbb Q$-divisor}}).  

Two $\mathbb R$-divisors $D_1$ and $D_2$ are said to 
be {\em{linearly equivalent}} if 
$D_1-D_2$ is a principal Cartier divisor. 
The linear equivalence is denoted by $D_1\sim D_2$. 
Two $\mathbb R$-divisors $D_1$ and $D_2$ are 
said to be {\em{$\mathbb R$-linearly equivalent}} 
(resp.~{\em{$\mathbb Q$-linearly equivalent}}) 
if $D_1-D_2$ is a {\em{finite}} $\mathbb R$-linear 
(resp.~$\mathbb Q$-linear) combination 
of principal Cartier divisors. 
When $D_1$ is $\mathbb R$-linearly (resp.~$\mathbb Q$-linearly) 
equivalent to 
$D_2$, we write $D_1\sim _{\mathbb R}D_2$ 
(resp.~$D_1\sim _{\mathbb Q}D_2$). 

\begin{ex}\label{c-ex2.7}
Let $X$ be a non-compact Riemann surface and let 
$\{P_k\}_{k=1}^\infty$ be a 
set of mutually distinct discrete points of 
$X$. 
We put $D:=\sum _{k=1}^\infty \frac{1}{k} P_k$. 
Then $D$ is obviously a $\mathbb Q$-Cartier 
$\mathbb Q$-divisor 
on $X$. 
However, $D$ is not a finite $\mathbb Q$-linear 
combination of Cartier divisors on $X$.  
\end{ex}
We note that in this paper we can almost always assume that 
$\Supp D$ has only finitely many irreducible components. 
\end{say}

\begin{say}[Hybrids of $\mathbb R$-line bundles 
and $\mathbb R$-Cartier divisors]\label{c-say2.8} 
In this paper, we usually treat 
hybrids of $\mathbb R$-line bundles and $\mathbb R$-Cartier divisors. 

Let $\pi\colon X\to Y$ be a projective morphism 
between complex analytic spaces and let $W$ be a compact 
subset of $Y$. 
Let $A$ and $B$ be $\mathbb R$-Cartier divisors 
on $X$ and let $\mathcal L$ and $\mathcal M$ be 
$\mathbb R$-line bundles on $X$. 

We sometimes say that  
\begin{equation*}
\mathcal L+A\sim _{\mathbb R} \mathcal M+B
\end{equation*} holds 
over some open neighborhood $U$ of $W$. 
This means: 
\begin{itemize}
\item[(i)] We implicitly assume that 
$A|_{\pi^{-1}(U)}$ and $B|_{\pi^{-1}(U)}$ are finite 
$\mathbb R$-linear combinations of 
Cartier divisors on $\pi^{-1}(U)$. 
Thus we can obtain $\mathbb R$-line bundles 
$\mathcal A$ and $\mathcal B$ naturally associated 
to $A|_{\pi^{-1}(U)}$ and $B|_{\pi^{-1}(U)}$, respectively. 
\item[(ii)] In $\Pic(\pi^{-1}(U))\otimes _{\mathbb Z}\mathbb R$, 
the following equality 
\begin{equation*}
\mathcal L|_{\pi^{-1}(U)}+\mathcal A=\mathcal M|_{\pi^{-1}(U)}+\mathcal B
\end{equation*} 
holds. 
\end{itemize}
If $X$ is a normal complex variety and 
$U$ is a relatively compact open subset of $Y$, then 
$A|_{\pi^{-1}(U)}$ and $B|_{\pi^{-1}(U)}$ are automatically finite 
$\mathbb R$-linear combinations of 
Cartier divisors on $\pi^{-1}(U)$. 
Therefore, (i) is harmless for applications. 

Similarly, we say that $\mathcal L+A$ is $\pi$-ample 
over some open neighborhood $U$ of $W$ if 
$A|_{\pi^{-1}(U)}$ is a finite $\mathbb R$-linear 
combination of Cartier divisors on $\pi^{-1}(U)$, 
$\mathcal A$ is the $\mathbb R$-line bundle 
naturally associated to $A|_{\pi^{-1}(U)}$, 
and $\mathcal L|_{\pi^{-1}(U)}+\mathcal A$ is $\pi$-ample 
over $U$, that is, $\mathcal L|_{\pi^{-1}(U)}+\mathcal A$ 
is a finite positive $\mathbb R$-linear combination 
of $\pi$-ample line bundles on $\pi^{-1}(U)$. 
\end{say}

\begin{say}\label{c-say2.9}
Let $\pi\colon X\to Y$ be a projective morphism 
of complex analytic spaces such that 
$X$ is a normal complex variety and 
let $\Delta$ be an $\mathbb R$-divisor 
on $X$ such that 
$K_X+\Delta$ is $\mathbb R$-Cartier. 
Let $y$ be an arbitrary point of $Y$ and let $U_y$ be 
any relatively compact Stein open neighborhood of 
$y\in Y$. 
In this case, we can always find a Weil divisor 
$K_{\pi^{-1}(U_y)}$ on $\pi^{-1}(U_y)$ such that 
$\mathcal O_{\pi^{-1}(U_y)}(K_{\pi^{-1}(U_y)})\simeq 
\omega_{\pi^{-1}(U_y)}$ holds since $\pi$ is projective 
and $U_y$ is Stein. 
Since $U_y$ is relatively compact, 
$\Supp \Delta|_{\pi^{-1}(U_y)}$ has only finitely many irreducible 
components. Thus, we can easily see that 
$K_{\pi^{-1}(U_y)}+\Delta|_{\pi^{-1}(U_y)}$ is a globally $\mathbb R$-Cartier 
$\mathbb R$-divisor on $\pi^{-1}(U_y)$. 
Moreover, for any $\mathbb R$-line bundle $\mathcal L$ on $X$, 
we can take a globally $\mathbb R$-Cartier $\mathbb R$-divisor 
$L$ on $\pi^{-1}(U_y)$ such that 
$\mathcal L|_{\pi^{-1}(U_y)}$ is the 
$\mathbb R$-line bundle naturally associated to $L$. 
\end{say}

In the theory of minimal models, 
we often use the 
following formulation. 
We will repeatedly use it in subsequent sections. 

\begin{say}\label{c-say2.10}
Let $X$ be a normal complex variety. 
A real vector space 
spanned by the prime divisors on $X$ is 
denoted by $\WDiv_{\mathbb R}(X)$, 
which has a canonical basis given 
by the prime divisors. 
Let $D$ be an element of $\WDiv_{\mathbb R}(X)$. 
Then the sup norm of $D$ with respect to 
this basis is denoted by $|\!|D|\!|$. 
Note that an $\mathbb R$-divisor $D$ on $X$ is an 
element of $\WDiv_{\mathbb R}(X)$ if and only if 
$\Supp D$ has only finitely many irreducible components. 

Let $V$ be 
a finite-dimensional affine subspace of 
$\WDiv_{\mathbb R} (X)$, which is defined over the 
rationals. 
We put 
\begin{equation*}
\mathcal R(V; x):=\{ \Delta\in V \mid {\text{$K_X+\Delta$ is 
$\mathbb R$-Cartier at $x$}}\}. 
\end{equation*} 
It is obvious that $\mathcal R(V; x)$ is an affine subspace of 
$V$. 
We take an arbitrary element $\Delta$ of $\mathcal R(V; x)$. 
Then $K_X+\Delta$ is $\mathbb R$-Cartier at $x$ by definition. 
Therefore, there exist a small open neighborhood 
$U_x$ of $x$ such that 
\begin{equation*}
K_{U_x}+\Delta|_{U_x}=\sum _{i=1}^k a_iD_i, 
\end{equation*} 
where $D_i$ is a Cartier divisor on $U_x$ and 
$a_i$ is a real number for every $i$. 
By this description, we can easily see that 
there exists an affine subspace $\mathcal T$ 
of $V$ defined over the rationals such that 
$\Delta\in \mathcal T\subset \mathcal R(V; x)$. 
Hence $\mathcal R(V; x)$ itself is an affine 
subspace of $V$ defined over the rationals. 
Let $L$ be a compact subset of $X$. 
We put 
\begin{equation*}
\mathcal R(V; L):=\{\Delta \in V\mid {\text{$K_X+\Delta$ 
is $\mathbb R$-Cartier at $L$}}\}. 
\end{equation*} 
Then the following equality 
\begin{equation*}
\mathcal R(V; L)=\bigcap _{x\in L} \mathcal R(V; x)
\end{equation*} 
obviously holds. 
Therefore, $\mathcal R(V; L)$ is an affine subspace 
of $V$ defined over the rationals. 
After shrinking $X$ around 
$L$ suitably, we may assume 
that $K_X+\Delta$ is $\mathbb R$-Cartier for 
every $\Delta\in \mathcal R(V; L)$ since 
$V$ is finite-dimensional and $L$ is compact. 
Let $\Theta$ be the union of the support of any element 
of $\mathcal R(V; L)$. 
By \cite[Theorem 13.2]{bierstone-milman}, 
after shrinking $X$ around $L$ suitably, 
we can construct a projective bimeromorphic 
morphism $f\colon Y\to X$ from a smooth 
complex analytic space $Y$ such that 
$\Exc(f)$ and $\Exc(f)\cup \Supp f^{-1}_*\Theta$ are 
simple normal crossing divisors on $Y$, where 
$\Exc(f)$ denotes the exceptional locus of $f\colon Y\to X$. 
Thus, for any $\Delta\in \mathcal R(V; L)$, 
we can write 
\begin{equation*}
K_Y+\Delta_Y:=f^*(K_X+\Delta) 
\end{equation*} 
such that $\Supp \Delta_Y$ is a simple normal crossing divisor 
on $Y$. In this situation, $(X, \Delta)$ is log canonical at 
$L$ if and only if $\Delta$ is effective at $L$ and 
the coefficients of $\Delta_Y$ are less than or equal to 
one over $L$. 
Hence, we can easily check that 
\begin{equation*}
\mathcal L(V; L):=\{\Delta\in V\, |\, 
\text{$K_X+\Delta$ is log 
canonical at $L$}\} 
\end{equation*} 
is a rational polytope contained in $\mathcal R(V; L)$. 
We can also check that there 
exists an open neighborhood $U$ of $L$ such that 
$(U, \Delta|_U)$ is log canonical 
for every $\Delta\in \mathcal L(V; L)$. 
\end{say}

\begin{say}\label{c-say2.11}
Let $X$ be a complex analytic space. 
An {\em{analytic subset}} (resp.~A {\em{locally closed 
analytic subset}}) of $X$ is the support of 
a closed analytic subspace (resp.~a locally closed 
analytic subspace) of $X$. 
A {\em{Zariski open subset}} of $X$ means 
the complement of an analytic subset. 
We note the following easy example. 
\begin{ex}\label{c-ex2.12}
We consider $\Delta:=\{z\in \mathbb C \mid |z|<1\}$ and 
$\Delta^*:=\Delta\setminus \{0\}$. 
Then $\Delta^*$ is a Zariski open subset of $\Delta$. 
We put 
\begin{equation*}
U:=\Delta^*\setminus \left\{ \frac{1}{n}\, \middle|\,  {\text{$n\in \mathbb Z$ 
with $n\geq 2$}}\right\}. 
\end{equation*} 
Then $U$ is a Zariski open subset of $\Delta^*$ since 
\begin{equation*}
\left\{ \frac{1}{n}\, \middle|\,  {\text{$n\in \mathbb Z$ 
with $n\geq 2$}}\right\}
\end{equation*} 
is a closed analytic subset of $\Delta^*$. 
However, $U$ is not a Zariski open subset of $\Delta$. 
This is because 
\begin{equation*}
\{0\} \cup \left\{ \frac{1}{n}\, \middle|\,  {\text{$n\in \mathbb Z$ 
with $n\geq 2$}}\right\}
\end{equation*} 
is not a closed analytic subset of $\Delta$. 
\end{ex}
\end{say}

\begin{say}\label{c-say2.13}
A subset $\mathcal S$ of a complex analytic space $X$ 
is said to be {\em{analytically
meagre}} if
\begin{equation*}
\mathcal S\subset \bigcup _{n\in \mathbb N} Y_n, 
\end{equation*} 
where each $Y_n$ is a locally closed analytic 
subset of $X$ of codimension $\geq 1$.

Let $X$ be a complex analytic space. We say that 
a property $P$ holds for an analytically sufficiently 
general point $x\in X$ when $P$ holds 
for every point $x$ contained in $X\setminus \mathcal S$ for some 
analytically meagre subset $\mathcal S$ of $X$. 

Let $\pi\colon X\to Y$ be a morphism of analytic spaces. 
Similarly, we say that a property $P$ holds for an 
{\em{analytically sufficiently general fiber}} of $\pi\colon X\to Y$ when 
$P$ holds for $\pi^{-1}(y)$ for 
every $y\in Y\setminus \mathcal S$, 
where $\mathcal S$ is some analytically meagre subset of $Y$. 
\end{say}

In this paper, we will freely use the following facts, which 
can be found in \cite[Chapter III]{banica}. 

\begin{say}\label{c-say2.14}
Let $\pi\colon X\to Y$ be a projective surjective morphism 
of complex analytic spaces and let $\mathcal L$ be 
a line bundle on $X$. 
If $R^p\pi_*\mathcal L=0$ holds, then $H^p(F, \mathcal L|_F)=0$ 
for an analytically sufficiently general fiber $F$ of $\pi \colon X\to Y$. 
If $H^0(F, \mathcal L|_F)\ne 0$ for an analytically sufficiently 
general fiber $F$ of $\pi\colon X\to Y$, then 
$\pi_*\mathcal L\ne 0$ holds. 
\end{say}

We will use the following convention throughout this paper. 

\begin{say}\label{c-say2.15}
The expression \lq ... for every $m\gg 0$\rq \ means 
that \lq there exists a positive real number $m_0$ such that ... 
for every $m\geq m_0$.\rq
\end{say}

\section{Basic properties of relatively 
ample and relatively nef $\mathbb R$-line bundles}\label{c-sec3}
In this section, we will collect some basic properties of 
relatively nef and relatively ample $\mathbb R$-line bundles for the 
reader's convenience. 
We will frequently use them in subsequent sections. 

Let us recall the definition of 
{\em{projective morphisms of complex analytic spaces}} 
for the sake of completeness. 

\begin{defn}[Projective morphisms of complex analytic spaces]
\label{c-def3.1}
Let $\pi\colon X\to Y$ be a proper 
morphism of complex analytic spaces and 
let $\mathcal L$ be a line bundle on $X$. 
Then $\mathcal L$ is said to be {\em{$\pi$-very 
ample}} or {\em{relatively very ample over 
$Y$}} if $\mathcal L$ is {\em{$\pi$-free}}, 
that is, 
\begin{equation*} 
\pi^*\pi_*\mathcal L\to \mathcal L
\end{equation*} 
is surjective, 
and the induced morphism 
\begin{equation*} 
X\to \mathbb P_Y(\pi_*\mathcal L)
\end{equation*} 
over 
$Y$ is a closed embedding. 
A line bundle $\mathcal L$ on $X$ is called 
{\em{$\pi$-ample}} or {\em{ample over 
$Y$}} if for any point $y\in Y$ there are an 
open neighborhood $U$ of $y$ and a positive 
integer $m$ such that $\mathcal L^{\otimes m}|_{\pi^{-1}(U)}$ 
is relatively very ample over $U$. 
Let $D$ be a Cartier divisor on $X$. Then 
we say that $D$ is {\em{$\pi$-very ample}}, {\em{$\pi$-free}}, 
and {\em{$\pi$-ample}} if the line bundle $\mathcal O_X(D)$ is 
so, respectively. 
We note that $\pi\colon X\to Y$ is said to be {\em{projective}} 
when there exists a $\pi$-ample line bundle 
on $X$. 
\end{defn}

For the basic properties of $\pi$-ample line 
bundles, see \cite[Chapter IV]{banica} and \cite[Chapter II. ~\S1.c.~Ample 
line bundles]{nakayama2}. 
Since we are mainly interested in $\mathbb R$-line bundles 
in this paper, the following easy lemma is indispensable. 

\begin{lem}\label{c-lem3.2}
Let $\pi\colon X\to Y$ be a projective morphism 
between complex analytic spaces and let $W$ be 
a compact subset of $Y$. 
Let $\mathcal L$ be an $\mathbb R$-line bundle on $X$. 
Then the following two conditions are equivalent. 
\begin{itemize}
\item[(i)] $\mathcal L$ is $\pi$-ample over $W$, 
that is, $\mathcal L|_{\pi^{-1}(w)}$ is ample in the 
usual sense for every $w\in W$. 
\item[(ii)] $\mathcal L$ is $\pi$-ample over some open neighborhood 
$U$ of $W$, that is, $\mathcal L|_{\pi^{-1}(U)}$ is a finite 
positive $\mathbb R$-linear combination of $\pi|_{\pi^{-1}(U)}$-ample 
line bundles. 
\end{itemize}
\end{lem}

\begin{proof}[Sketch of Proof of Lemma \ref{c-lem3.2}]
It is obvious that (i) follows from (ii). 
Hence it is sufficient to prove that (ii) follows from (i). 
It is an easy exercise to modify the proof of 
\cite[Lemmas 6.1 and 6.2]{fujino-miyamoto} suitably with 
the aid of \cite[Proposition 1.4]{nakayama1}. 
\end{proof}

Throughout this paper, we will freely use Lemma \ref{c-lem3.2} 
without mentioning it explicitly. 
The following lemma is more or less well known to the experts. 
We describe it here for the sake of completeness. 

\begin{lem}\label{c-lem3.3}
Let $\pi\colon X\to Y$ be a projective surjective morphism 
of complex analytic spaces such that 
$X$ and $Y$ are both irreducible. Let $\mathcal L$ be a line bundle 
on $X$. Assume that 
$\mathcal L|_{\pi^{-1}(y)}$ is ample for some $y\in Y$. 
Then there exists a Zariski open neighborhood $U$ of $y$ in $Y$ and 
a positive integer $m$ such that 
$\mathcal L^{\otimes m}|_{\pi^{-1}(U)}$ is $\pi$-very ample 
over $U$. In particular, $\mathcal L|_{\pi^{-1}(U)}$ is $\pi$-ample 
over $U$. 
\end{lem}

\begin{proof}
It is well known that there exists a small open neighborhood 
$U_1$ of $y$ in $Y$ such that 
$\mathcal L|_{\pi^{-1}(U_1)}$ is $\pi$-ample 
over $U_1$ (see \cite[Proposition 1.4]{nakayama1}). 
Therefore, we can take some positive integer $m$ such that 
$\mathcal L^{\otimes m}$ is $\pi$-very ample over some small 
open neighborhood $U_2$ of $y$ in $Y$. 
We consider $\pi^*\pi_*\mathcal L^{\otimes m} 
\to \mathcal L^{\otimes m}$. 
It is obviously surjective over $U_2$. 
Therefore, 
\begin{equation*}
\pi\left(\Supp \Coker (\pi^*\pi_*\mathcal L^{\otimes m} \to 
\mathcal L^{\otimes m})\right)\cap U_2=\emptyset. 
\end{equation*}
Then we put 
\begin{equation*}
U_3:=Y\setminus \pi\left(\Supp \Coker (\pi^*\pi_*\mathcal L^{\otimes m} \to 
\mathcal L^{\otimes m})\right). 
\end{equation*} 
Hence $U_3$ is a non-empty Zariski open subset of 
$Y$ such that $y\in U_3$ and that 
$\pi^*\pi_*\mathcal L^{\otimes m} \to \mathcal L^{\otimes m}$ 
is surjective over $U_3$. 
We put 
\begin{equation*}
\mathcal I:=\xIm \left(\pi^*\pi_*\mathcal L^{\otimes m} 
\to \mathcal L^{\otimes m} \right) \otimes \mathcal L^{\otimes (-m)} 
\subset \mathcal O_X. 
\end{equation*}
Then $\mathcal I$ is a coherent ideal sheaf on $X$. 
We take the blow-up $p\colon Z\to X$ of $X$ along the ideal sheaf 
$\mathcal I$, that is, 
$p\colon Z:=\Projan _X\bigoplus _{d=0}^{\infty} \mathcal I^d\to X$. 
By construction,  
\begin{equation*}
\mathcal M:= \xIm \left(p^*\pi^* \pi_*\mathcal L^{\otimes m} 
\to p^*\mathcal L^{\otimes m}\right)
\end{equation*} 
becomes a line bundle on $Z$. 
This gives a closed embedding 
\begin{equation*}
Z\simeq \mathbb P_Z(\mathcal M)\hookrightarrow 
\mathbb P_Y(\pi_*\mathcal L^{\otimes m})\times _Y Z. 
\end{equation*} 
Thus we obtain a morphism $\alpha\colon Z\to 
\mathbb P_Y(\pi_*\mathcal L^{\otimes m})$ over $Y$. 
By construction again, 
$p$ is an isomorphism over $U_3$ and $\alpha$ is a closed 
embedding over $U_2$. 
We can take a non-empty Zariski open subset $V$ of $\alpha (Z)$ 
such that $\alpha$ is flat over $V$. Without loss of generality,  
we may assume that $V$ contains $q^{-1}(U_2)$, 
where $q\colon \alpha (Z)\to Y$, and 
that $\alpha$ is an isomorphism over $V$. 
\begin{equation*}
\xymatrix{
& Z\ar[dl]_-p\ar[dr]^-\alpha 
&\\ 
X\ar@{-->}[rr]^-{\alpha\circ p^{-1}}\ar[dr]_-\pi& &\alpha(Z)\ar[dl]^-q \\
& Y&
}
\end{equation*}
We put 
\begin{equation*}
U:=U_3\cap \left ( Y\setminus q\left(\alpha(Z)\setminus V\right)\right). 
\end{equation*}
Then $U$ is a non-empty Zariski open subset 
of $Y$ such that $y\in U$ and $\alpha \circ p^{-1}\colon 
X\dashrightarrow 
\mathbb P_Y(\pi_*\mathcal L^{\otimes m})$ is a closed embedding 
over $U$. 
Therefore, $\mathcal L^{\otimes m}$ is $\pi$-very ample 
over $U$. 
\end{proof}

As an application of Lemma \ref{c-lem3.3}, 
we have: 

\begin{lem}\label{c-lem3.4}
Let $\pi\colon X\to Y$ be a projective surjective morphism 
of complex analytic spaces. 
Let $\mathcal L$ be an $\mathbb R$-line bundle 
on $X$. Assume that 
$\mathcal L|_{\pi^{-1}(y)}$ is ample for some $y\in Y$. 
Then there exists a Zariski open neighborhood $U$ of $y$ in $Y$ 
such that 
$\mathcal L|_{\pi^{-1}(U)}$ is $\pi$-ample 
over $U$. 
\end{lem}

In the theory of minimal models, we have to treat $\mathbb R$-line bundles. 
Therefore, Lemma \ref{c-lem3.4} is indispensable. 
Since we can not directly apply geometric arguments 
to $\mathbb R$-line bundles, Lemma \ref{c-lem3.4} is not so obvious. 

\begin{proof}[Proof of Lemma \ref{c-lem3.4}]
We can write $\mathcal L=\sum _{i\in I} a_i \mathcal L_i$ in $
\Pic(X)\otimes _{\mathbb Z} \mathbb R$ such that 
$a_i$ is a positive real number, $\mathcal L_i\in \Pic(X)$, 
and $\mathcal L_i|_{\pi^{-1}(y)}$ is ample for every $i\in I$. 
Let $X=\bigcup _{j\in J}X_j$ be the irreducible decomposition. 
We put 
\begin{equation*} 
J_1:=\{j\in J\, |\, y\in \pi(X_j)\} \quad\text{and}\quad 
J_2:=\{j\in J\, |\, y\not\in \pi(X_j)\}. 
\end{equation*}
We take an irreducible component $X_j$ of $X$ with $j\in J_1$. 
By applying Lemma \ref{c-lem3.3} to $\mathcal L_i|_{X_j}$, 
we can find a Zariski closed subset $\Sigma_j$ of $\pi(X_j)$ such that 
$y\in \pi(X_j)\setminus \Sigma_j$, 
$\mathcal L_i|_{X_j}$ is ample over 
$\pi(X_j)\setminus \Sigma_j$ for every $i\in I$. 
This implies that 
$\mathcal L|_{X_j}$ is ample over $\pi(X_j)\setminus \Sigma_j$. 
We put 
\begin{equation*}
\Sigma:=\left(\bigcup _{i\in J_1} \Sigma_j\right) 
\cup \pi\left(\bigcup_{j\in J_2}X_j\right). 
\end{equation*} 
Then $\Sigma$ is a Zariski closed subset of $Y$ such that 
$y\in Y\setminus \Sigma$ and that 
$\mathcal L$ is $\pi$-ample 
over $Y\setminus \Sigma$. 
Therefore, $U:=Y\setminus \Sigma$ is a desired Zariski open neighborhood 
of $y$ in $Y$. 
\end{proof}

By Lemma \ref{c-lem3.4}, we can easily obtain: 

\begin{lem}\label{c-lem3.5}
Let $\pi\colon X\to Y$ be a projective surjective 
morphism of complex analytic spaces. 
Let $\mathcal L$ be an $\mathbb R$-line bundle on $X$. 
Assume that $\mathcal L|_{\pi^{-1}(y_0)}$ is nef for some 
$y_0\in Y$. 
Then there exists an analytically meagre subset 
$\mathcal S$ such that 
$\mathcal L|_{\pi^{-1}(y)}$ is nef 
for every $y\in Y\setminus \mathcal S$. 
\end{lem}

Although Lemma \ref{c-lem3.5} is easy, it will play a very important 
role in our framework of the minimal model program of 
complex analytic spaces. 
We note that we can not make $\mathcal S$ a Zariski 
closed subset of $Y$ in Lemma \ref{c-lem3.5}. 

\begin{proof}[Proof of Lemma \ref{c-lem3.5}] 
We take a $\pi$-ample line bundle $\mathcal H$ on $X$. 
Then $(m\mathcal L+\mathcal H)|_{\pi^{-1}(y_0)}$ is 
ample for every positive integer $m$. 
Therefore, by Lemma \ref{c-lem3.4}, 
for each $m\in \mathbb Z_{>0}$, 
we can take a Zariski open neighborhood 
$U_m$ of $y_0$ in $Y$ such that 
$m\mathcal L+\mathcal H$ is $\pi$-ample 
over $U_m$. 
We put $\mathcal S:=\bigcup _{m\in \mathbb Z_{>0}} 
(Y\setminus U_m)$. 
Then $(m\mathcal L+\mathcal H)|_{\pi^{-1}(y)}$ is ample 
for every $m\in \mathbb Z_{>0}$ and every 
$y\in Y\setminus \mathcal S$. 
This means that $\mathcal L|_{\pi^{-1}(y)}$ is nef for 
every $y\in Y\setminus \mathcal S$. 
\end{proof}

The following obvious corollary of Lemma \ref{c-lem3.5} 
is also useful for geometric applications. 
We note that we sometimes have to treat a countably infinite set of 
line bundles. 

\begin{cor}\label{c-cor3.6}
Let $\pi\colon X\to Y$ be a projective surjective 
morphism of complex analytic spaces. 
Let $\mathcal L_i$ be an $\mathbb R$-line bundle on $X$ for 
$i\in \mathbb N$. 
Assume that $\mathcal L_i|_{\pi^{-1}(y_0)}$ is nef for some 
$y_0\in Y$ and for every $i\in \mathbb N$. 
Then there exists an analytically meagre subset 
$\mathcal S$ such that 
$\mathcal L_i|_{\pi^{-1}(y)}$ is nef 
for every $y\in Y\setminus \mathcal S$ and every $i\in \mathbb N$. 
Therefore, if $\mathcal H$ is 
any $\pi$-ample $\mathbb R$-line bundle 
on $X$, then $(\mathcal H+\mathcal L_i)|_{\pi^{-1}(y)}$ 
is ample for every $y\in Y\setminus \mathcal S$ and 
every $i\in \mathbb N$. 
\end{cor}
\begin{proof}
By Lemma \ref{c-lem3.5}, for each $i\in \mathbb N$, 
we can find an analytically meagre subset 
$\mathcal S_i$ of $Y$ such that $\mathcal L_i|_{\pi^{-1}(y)}$ is 
nef for every $y\in Y\setminus \mathcal S_i$. 
We put $\mathcal S:=\bigcup _{i\in \mathbb N} \mathcal S_i$. 
Then it is easy to see that $\mathcal S$ is a desired 
analytically meagre subset of $Y$. 
\end{proof}

By the proof of Lemma \ref{c-lem3.5} and Corollary \ref{c-cor3.6}, 
we have: 

\begin{rem}\label{c-rem3.7} In Lemma \ref{c-lem3.5} and 
Corollary \ref{c-cor3.6}, we can make $Y\setminus \mathcal S$ 
a countable intersection of non-empty Zariski open subsets of 
$Y$. 
\end{rem}

\section{Non-lc ideal sheaves}\label{c-sec4}

Let us recall the notion of {\em{non-lc ideal sheaves}}. 
It is well defined even in the complex analytic setting. 

\begin{defn}[{Non-lc ideal sheaves, 
see \cite[Definition 7.1]{fujino-fundamental}}]\label{c-def4.1}
Let $X$ be a normal complex variety and let $\Delta$ be 
an effective $\mathbb R$-divisor on $X$ such that 
$K_X+\Delta$ is $\mathbb R$-Cartier. 
Let $f\colon Z\to X$ be a projective bimeromorphic 
morphism from a smooth complex variety $Z$ with 
$K_Z+\Delta_Z:=f^*(K_X+\Delta)$ such that 
$\Supp \Delta_Z$ is a simple normal crossing divisor on $Z$. 
Then we put 
\begin{equation*}
\mathcal J_{\NLC} (X, \Delta):=f_*\mathcal O_Z(\lceil 
-(\Delta^{<1}_Z)\rceil -\lfloor \Delta^{>1}_Z\rfloor)
=f_*\mathcal O_Z(-\lfloor \Delta_Z\rfloor +\Delta^{=1}_Z)
\end{equation*} 
and call it the {\em{non-lc ideal sheaf associated 
to 
$(X, \Delta)$}}. We put 
\begin{equation*}
\mathcal J(X, \Delta):=
f_*\mathcal O_Z(-\lfloor \Delta_Z\rfloor). 
\end{equation*} 
Then $\mathcal J(X, \Delta)$ is the well-known {\em{multiplier ideal 
sheaf associated to $(X, \Delta)$}}. 
By definition, the following inclusion 
\begin{equation*}
\mathcal J(X, \Delta)\subset \mathcal J_{\NLC}(X, \Delta)
\end{equation*} 
always holds. 
By definition again, we can easily see that 
the support of $\mathcal O_X/\mathcal J_{\NLC}(X, \Delta)$ 
is the non-lc locus $\Nlc(X, \Delta)$ of $(X, \Delta)$. 
\end{defn}

By the standard argument 
(see, for example, \cite[Lemma 7.2]{fujino-fundamental}), 
there are 
no difficulties to check the following lemma. 

\begin{lem}\label{c-lem4.2}
In Definition \ref{c-def4.1}, $\mathcal J_{\NLC}(X, \Delta)$ and 
$\mathcal J(X, \Delta)$ are  
independent of the choice of the resolution $f\colon Z\to X$. 
Hence $\mathcal J_{\NLC}(X, \Delta)$ and 
$\mathcal J(X, \Delta)$ are well-defined 
coherent ideal sheaves on $X$. 
\end{lem}
\begin{proof}[Sketch of Proof of Lemma \ref{c-lem4.2}]
Since we do not use $\mathcal J(X, \Delta)$ in this paper and 
the proof for $\mathcal J(X, \Delta)$ is simpler than for 
$\mathcal J_{\NLC}(Z, \Delta)$, we only treat $\mathcal J_{\NLC}(X, \Delta)$ 
here. 
Let $f_1\colon Z_1\to X$ and $f_2\colon Z_2\to X$ be two 
resolutions with $K_{Z_1}+\Delta_{Z_1}=f^*_1(K_X+\Delta)$ and 
$K_{Z_2}+\Delta_{Z_2}=f^*_2(K_X+\Delta)$ as in Definition \ref{c-def4.1}. 
We take an arbitrary point $x\in X$. It is sufficient to 
prove that 
\begin{equation*}
{f_1}_*\mathcal O_{Z_1}(-\lfloor \Delta_{Z_1}\rfloor 
+\Delta^{=1}_{Z_1})
={f_2}_*\mathcal O_{Z_2}(-\lfloor \Delta_{Z_2}\rfloor +\Delta^{=1}_{Z_2})
\end{equation*}
holds on some open neighborhood of $x$. 
Therefore, by shrinking $X$ around $x$ and taking an elimination 
of indeterminacy of $Z_2\dashrightarrow Z_1$, 
we may further assume that $f_2$ decomposes as 
\begin{equation*}
\xymatrix{
f_2\colon Z_2\ar[r]&  Z_1\ar[r]^-{f_1}& X.  
}
\end{equation*} 
Then, by \cite[Proposition 6.3.1]{fujino-foundations}, 
we can directly check that ${f_1}_*\mathcal O_{Z_1}(-\lfloor \Delta_{Z_1}\rfloor 
+\Delta^{=1}_{Z_1})
={f_2}_*\mathcal O_{Z_2}(-\lfloor \Delta_{Z_2}\rfloor +\Delta^{=1}_{Z_2})
$ holds. We finish the proof. 
\end{proof}

In this paper, we need the following Bertini-type theorem 
for $\mathcal J_{\NLC}(X, \Delta)$. 

\begin{lem}[{\cite[Proposition 7.5]{fujino-fundamental}}]\label{c-lem4.3}
Let $X$ be a normal complex variety and let $\Delta$ be 
an effective $\mathbb R$-divisor on $X$ such that 
$K_X+\Delta$ is $\mathbb R$-Cartier. 
Let $\Lambda$ {\em{(}}$\simeq \mathbb P^N${\em{)}} 
be a finite-dimensional linear system on $X$. 
Let $X^\dag$ be any relatively compact open subset 
of $X$. 
Then there exists an analytically meagre subset 
$\mathcal S$ of $\Lambda$ such that 
\begin{equation*}
\mathcal J_{\NLC}(X^\dag, \Delta+tD)
=\mathcal J_{\NLC}(X^\dag, \Delta)
\end{equation*} 
holds outside the base locus $\Bs\Lambda$ of 
$\Lambda$ for every element 
$D$ of $\Lambda \setminus \mathcal S$ and 
every $0\leq t\leq 1$. 
\end{lem}
\begin{proof}
Without loss of generality, 
we can freely replace $X$ with a relatively compact open neighborhood 
of $\overline{X^\dag}$. 
Therefore, by the desingularization theorem (see 
\cite[Theorem 13.2]{bierstone-milman}), 
we can take a projective bimeromorphic 
morphism $f\colon Z\to X$ from a smooth 
complex variety $Z$ with $K_Z+\Delta_Z=f^*(K_X+\Delta)$ such that 
$\Supp \Delta_Z$ is a simple normal crossing divisor on $Z$. 
By replacing $X$ with $X\setminus \Bs \Lambda$, 
we may further assume that $\Bs\Lambda=\emptyset$. 
By Bertini's theorem, 
there exists an analytically meagre subset $\mathcal S$ of $\Lambda$ 
such that $f^*D$ is smooth, 
$f^*D=f^{-1}_*D$, $f^*D$ and $\Supp \Delta_Z$ 
have no common irreducible components, and 
the support of $f^*D+\Supp \Delta_Z$ is a simple 
normal crossing divisor on $Z$ for every element 
$D$ of $\Lambda \setminus \mathcal S$. 
Then $K_Z+\Delta_Z+f^*tD=f^*(K_X+\Delta+tD)$ holds over 
$X^\dag$ with $f^*tD=tf^{-1}_*D$. 
Thus, 
\begin{equation*}
\lceil 
-(\Delta^{<1}_Z)\rceil -\lfloor \Delta^{>1}_Z\rfloor
=\lceil 
-(\Delta_Z+f^*tD)^{<1}\rceil -\lfloor (\Delta_Z+f^*tD)^{>1}\rfloor
\end{equation*}
holds over $X^\dag$ for every $0\leq t\leq 1$ and 
every element $D$ of $\Lambda\setminus \mathcal S$. 
Thus, we obtain 
\begin{equation*}
\mathcal J_{\NLC}(X^\dag, \Delta+tD)
=\mathcal J_{\NLC}(X^\dag, \Delta) 
\end{equation*} 
by definition. 
This is what we wanted. 
\end{proof}

We need the following lemma in order to reduce the 
problems for $\mathbb R$-divisors to simpler problems for 
$\mathbb Q$-divisors. 

\begin{lem}\label{c-lem4.4}
Let $X$ be a normal complex variety and let 
$L$ be a compact subset of $X$. 
Let $\Delta$ be an effective $\mathbb R$-divisor 
on $X$ such that $K_X+\Delta$ is $\mathbb R$-Cartier at $L$. 
Then, after shrinking $X$ around $L$ suitably, 
there exist effective $\mathbb Q$-divisors $\Delta_1, \ldots, 
\Delta_k$ on $X$ and positive real numbers $r_1, \ldots, r_k$ with 
$\sum _{i=1}^k r_i =1$ such that 
$K_X+\Delta_i$ is $\mathbb Q$-Cartier 
for every $i$, 
$\Delta=\sum _{i=1}^k r_i \Delta_i$, and 
$\mathcal J_{\NLC}(X, \Delta_i)=\mathcal J_{\NLC}(X, \Delta)$ 
holds for every $i$. In particular, 
if $(X, \Delta)$ is log canonical, then 
$(X, \Delta_i)$ is log canonical for every $i$. 
\end{lem}

\begin{proof}
By shrinking $X$ around $L$ suitably, 
we may assume that 
$\Supp \Delta$ has only finitely many irreducible components. 
Let $\Supp \Delta:=\sum _{j=1}^l D_j$ be the irreducible 
decomposition. We consider the $\mathbb R$-vector space 
$V:=\bigoplus_{j=1}^l\mathbb R D_j$. 
We put 
\begin{equation*}
\mathcal R (V; L):=\{D\in V\, |\, {\text{$K_X+D$ is 
$\mathbb R$-Cartier at $L$}}\}. 
\end{equation*}
Then $\mathcal R(V; L)$ is an affine subspace of $V$ defined over 
the rationals (see \ref{c-say2.10}). 
By shrinking $X$ around $L$ suitably again, 
we may assume that $K_X+D$ is $\mathbb R$-Cartier for 
every $D\in \mathcal R(V; L)$. 
By \cite[Theorem 13.2]{bierstone-milman}, 
we may further assume that there exists a projective bimeromorphic 
morphism $f\colon Z\to X$ from a smooth complex 
analytic space $Z$ such that $\Exc(f)$ and $\Exc(f)\cup 
\sum _{j=1}^l \Supp f^{-1}_*D_j$ are simple 
normal crossing divisors on $Z$. 
We put 
\begin{equation*} 
\mathcal S_{\Delta}(V; L):=
\left\{D \in \mathcal R(V; L) \middle| 
\begin{array}{l}  {\text{$a(E, X, D)=a(E, X, \Delta)$ holds for every}}\\
{\text{divisor $E$ on $Z$ with $a(E, X, \Delta)\in \mathbb Q$}} 
\end{array} \right\}.  
\end{equation*} 
Then $\mathcal S_{\Delta}(V; L)$ is an affine subspace of $V$ defined 
over the rationals with $\Delta\in \mathcal S_{\Delta}(V; L)$. 
Since $\mathcal S_{\Delta}(V; L)$ is defined over the 
rationals, 
we can take effective $\mathbb Q$-divisors $\Delta_1, 
\ldots, \Delta_k$ from $\mathcal S_{\Delta}(V; L)$ such that 
they are close to $\Delta$ in $\mathcal S_{\Delta}(V; L)$ and 
positive real numbers $r_1, \ldots, r_k$ with all the 
desired properties. 
\end{proof}

Although we do not use multiplier ideal sheaves in this paper, 
we note: 

\begin{rem}\label{c-rem4.5}
In Lemma \ref{c-lem4.4}, 
we see that $\mathcal J(X, \Delta_i)=\mathcal J(X, \Delta)$ holds 
for every $i$ by construction. 
In particular, $(X, \Delta_i)$ is kawamata log terminal for 
every $i$ if $(X, \Delta)$ is kawamata log terminal. 
\end{rem}

\section{Quick review of vanishing theorems}\label{c-sec5}

In this section, let us quickly recall 
the vanishing theorems established in 
\cite{fujino-analytic-vanishing} (see also 
\cite{fujino-fujisawa} and \cite{fujino-vanishing-pja}). 
Let us start with the definition of 
{\em{analytic simple normal crossing pairs}}. 

\begin{defn}[Analytic simple normal crossing pairs]\label{c-def5.1}
Let $X$ be a simple normal crossing divisor 
on a smooth complex analytic space $M$ and 
let $B$ be an $\mathbb R$-divisor on $M$ such that 
$\Supp (B+X)$ is a simple normal crossing divisor on $M$ and 
that $B$ and $X$ have no common irreducible components. 
Then we put $D:=B|_X$ and 
consider the pair $(X, D)$. 
We call $(X, D)$ an {\em{analytic globally embedded simple 
normal crossing pair}} and $M$ the {\em{ambient space}} 
of $(X, D)$. 

If the pair $(X, D)$ is locally isomorphic to an analytic 
globally embedded 
simple normal crossing pair at any point of $X$ and the irreducible 
components of $X$ and $D$ are all smooth, 
then $(X, D)$ is called an {\em{analytic simple normal crossing 
pair}}. 

When $(X, D)$ is an analytic simple normal crossing pair, 
$X$ has an invertible dualizing sheaf $\omega_X$ since it 
is Gorenstein. We use the symbol $K_X$ as a formal 
divisor class with an isomorphism $\mathcal O_X(K_X)\simeq 
\omega_X$ if there is no danger of confusion. 
Note that we can not always define $K_X$ globally with 
$\mathcal O_X(K_X)\simeq \omega_X$. 
In general, it only exists locally on $X$. 
\end{defn}

\begin{rem}\label{c-rem5.2}
Let $X$ be a smooth complex analytic space and let $D$ be 
an $\mathbb R$-divisor on $X$ such that $\Supp D$ is a simple 
normal crossing divisor on $X$. 
Then, by considering $M:=X\times \mathbb C$, 
we can see $(X, D)$ as an analytic globally embedded simple 
normal crossing pair. 
\end{rem}
The notion of {\em{strata}}, 
which is a generalization of that of 
log canonical centers, plays a crucial role. 

\begin{defn}[Strata]\label{c-def5.3}
Let $(X, D)$ be an analytic simple normal crossing pair 
such that $D$ is effective. 
Let $\nu\colon X^\nu\to X$ be the normalization. 
We put 
\begin{equation*} 
K_{X^\nu}+\Theta=\nu^*(K_X+D). 
\end{equation*}  
This means that $\Theta$ is the union of $\nu^{-1}_*D$ and the 
inverse image of the singular locus of $X$. 
If $S$ is an irreducible component of $X$ 
or the $\nu$-image 
of some log canonical center of $(X^\nu, \Theta)$, 
then $S$ is called a {\em{stratum}} of $(X, D)$. 
\end{defn}

We recall Siu's theorem on coherent analytic sheaves, 
which is a special case of \cite[Theorem 4]{siu}. 

\begin{thm}\label{c-thm5.4} 
Let $\mathcal F$ be a coherent sheaf on a complex 
analytic space $X$. 
Then there exists a locally finite family $\{Y_i\}_{i\in I}$ 
of complex analytic subvarieties of $X$ such that 
\begin{equation*}
\Ass _{\mathcal O_{X,x}}(\mathcal F_x)=\{\mathfrak{p}_{x, 1}, 
\ldots, \mathfrak{p}_{x, r(x)}\}
\end{equation*}
holds for every point $x\in X$, where 
$\mathfrak{p}_{x, 1}, 
\ldots, \mathfrak{p}_{x, r(x)}$ are the prime ideals 
of $\mathcal O_{X, x}$ associated to the irreducible components 
of the germs $Y_{i, x}$ of $Y_i$ at $x$ with $x\in Y_i$. 
We note that each $Y_i$ is called an {\em{associated subvariety}} 
of $\mathcal F$. 
\end{thm}
Now we are ready to state the main result of \cite{fujino-analytic-vanishing}. 

\begin{thm}[{\cite[Theorem 1.1]{fujino-analytic-vanishing}}]\label{c-thm5.5}
Let $(X, \Delta)$ be an analytic simple 
normal crossing pair such that $\Delta$ is a boundary 
$\mathbb R$-divisor on $X$. 
Let $f\colon X\to Y$ be a projective morphism 
to a complex analytic space $Y$ and let $\mathcal L$ 
be a line bundle on $X$. 
Let $q$ be an arbitrary non-negative integer. 
Then we have the following properties. 
\begin{itemize}
\item[(i)] $($Strict support condition$)$.~If 
$\mathcal L-(\omega_X+\Delta)$ is $f$-semiample,  
then every 
associated subvariety of $R^qf_*\mathcal L$ is the $f$-image 
of some stratum of $(X, \Delta)$. 
\item[(ii)] $($Vanishing theorem$)$.~If 
$\mathcal L-(\omega_X+\Delta)\sim _{\mathbb R} f^*\mathcal H
$ holds 
for some $\pi$-ample 
$\mathbb R$-line bundle $\mathcal H$ on $Y$, where 
$\pi\colon Y\to Z$ is a 
projective morphism to a complex analytic space 
$Z$, then we have 
\begin{equation*}
R^p\pi_*R^qf_*\mathcal L=0
\end{equation*}
for every $p>0$. 
\end{itemize} 
\end{thm}

We make a supplementary remark on Theorem \ref{c-thm5.5}. 

\begin{rem}\label{c-rem5.6} 
In Theorem \ref{c-thm5.5} (and Theorem \ref{c-thm5.7} below), 
we always assume that $\Delta$ is {\em{globally $\mathbb R$-Cartier}}, 
that is, $\Delta$ is a finite $\mathbb R$-linear combination 
of Cartier divisors. We note that if the support of $\Delta$ 
has only finitely many irreducible components 
then it is globally $\mathbb R$-Cartier. 
Since we are mainly interested in the standard setting explained in 
\ref{c-say1.1}, this assumption is harmless to 
geometric applications.  
Under this assumption, we can obtain an $\mathbb R$-line bundle 
$\mathcal N$ 
on $X$ naturally associated to $\mathcal L-(\omega_X+\Delta)$. 
The assumption in (i) means that $\mathcal N$ is a finite 
positive $\mathbb R$-linear combination of $\pi$-semiample 
line bundles on $X$. 
The assumption in (ii) says that $\mathcal N=f^*\mathcal H$ holds 
in $\Pic(X)\otimes _{\mathbb Z} \mathbb R$. 
\end{rem}

We do not prove Theorem \ref{c-thm5.5} here. 
For the details of the proof of Theorem \ref{c-thm5.5}, see 
\cite{fujino-analytic-vanishing}, which depends on 
Saito's theory of mixed Hodge modules (see 
\cite{saito1}, \cite{saito2}, \cite{saito3}, 
\cite{fujino-fujisawa-saito}, 
and \cite{saito4}) and 
Takegoshi's analytic generalization of Koll\'ar's torsion-free and 
vanishing theorem (see \cite{takegoshi}). 
The reader can find an alternative approach 
to Theorem \ref{c-thm5.5} without using Saito's theory 
of mixed Hodge modules in \cite{fujino-fujisawa}. 
We note that 
Theorem \ref{c-thm5.5} is one of the main ingredients of 
this paper. Or, we can see this paper as an application of 
Theorem \ref{c-thm5.5}. 

\subsection{Vanishing theorems of Reid--Fukuda type}\label{c-subsec5.1}
Although we do not need vanishing theorems of Reid--Fukuda type 
in this paper, we will shortly discuss them here for future usage. 

\begin{thm}[{Vanishing theorem of Reid--Fukuda type, 
see \cite[Theorem 1.2]{fujino-analytic-vanishing}}]\label{c-thm5.7}
Let $(X, \Delta)$ be an analytic simple 
normal crossing pair such that $\Delta$ is a boundary 
$\mathbb R$-divisor on $X$. 
Let $f\colon X\to Y$ and $\pi\colon Y\to Z$ be projective morphisms 
between complex analytic spaces and let $\mathcal L$ 
be a line bundle on $X$. 
If $\mathcal L-(\omega_X+\Delta)\sim _{\mathbb R} f^*\mathcal H$ 
holds such that $\mathcal H$ is an $\mathbb R$-line bundle, 
which is nef and 
log big over $Z$ with respect to $f\colon (X, \Delta)\to Y$, on $Y$, then  
\begin{equation*}
R^p\pi_*R^qf_*\mathcal L=0
\end{equation*} 
holds for every $p>0$ and every $q$. 
\end{thm}

Theorem \ref{c-thm5.7} is obviously a generalization of 
Theorem \ref{c-thm5.5} (ii). 
The reader can find the detailed proof of 
Theorem \ref{c-thm5.7} in \cite{fujino-analytic-vanishing}, 
which is harder than that of Theorem \ref{c-thm5.5} (ii). 
As an easy application of Theorem \ref{c-thm5.7}, 
we can establish the vanishing theorem of Reid--Fukuda type 
for log canonical pairs in the complex analytic setting. 

\begin{thm}[Vanishing theorem of Reid--Fukuda type for 
log canonical pairs]\label{c-thm5.8}
Let $(X, \Delta)$ be a log canonical pair and let $\pi\colon X\to Y$ 
be a projective morphism of complex analytic spaces. 
Let $L$ be a $\mathbb Q$-Cartier integral Weil divisor on $X$. 
Assume that $L-(K_X+\Delta)$ is nef and 
big over $Y$ and 
that 
$(L-(K_X+\Delta))|_C$ is big 
over $\pi(C)$ for 
every log canonical center $C$ of $(X, \Delta)$. 
Then 
\begin{equation*}
R^q\pi_*\mathcal O_X(L)=0
\end{equation*} 
holds for every $q>0$. 
\end{thm}
\begin{proof}
The proof of \cite[Theorem 5.7.6]{fujino-foundations} works 
by Theorem \ref{c-thm5.7}. 
\end{proof}

We leave the details of Theorems \ref{c-thm5.7} and 
\ref{c-thm5.8} for the interested readers. 

\section{Vanishing theorems for normal pairs}\label{c-sec6}

In this section, we will prepare some vanishing theorems, 
which are suitable for geometric applications. 
They will play a crucial role in subsequent sections. 
We note that the results in this section follow from 
Theorem \ref{c-thm5.5}. 

\begin{thm}[{see \cite[Theorem 8.1]{fujino-fundamental}}]\label{c-thm6.1}
Let $\pi\colon X\to Y$ be a projective 
morphism of complex analytic spaces such that 
$X$ is a normal complex variety and let $W$ be a compact 
subset of $Y$. 
Let $\Delta$ be an effective $\mathbb R$-divisor on 
$X$ such that $K_X+\Delta$ is $\mathbb R$-Cartier 
and let $\mathcal L$ be a line bundle on $X$. 
We assume that $\mathcal L-(K_X+\Delta)$ is 
$\pi$-ample over $W$, that is, 
$\left(\mathcal L-(K_X+\Delta)\right)|_{\pi^{-1}(w)}$ is 
ample for every $w\in W$. 
Let $\{C_i\}_{i\in I}$ be any set of log canonical 
centers of the pair $(X, \Delta)$. 
We put $V:=\bigcup _{i\in I}C_i$ with the 
reduced structure. 
We further assume that $V$ is disjoint from the 
non-lc locus $\Nlc(X, \Delta)$ of $(X, \Delta)$. 
Then there exists some open neighborhood 
$U$ of $W$ such that 
\begin{equation*}
R^i\pi_*(\mathcal J\otimes \mathcal L)=0
\end{equation*} 
holds 
on $U$ for every $i>0$, 
where $\mathcal J :=\mathcal I_V\cdot \mathcal J_{\NLC}(X, \Delta)
\subset \mathcal O_X$ 
and $\mathcal I_V$ is the defining ideal sheaf 
of $V$ on $X$. 
Therefore, the restriction 
map 
\begin{equation*}
\pi_*\mathcal L\to \pi_*(\mathcal L|_V)\oplus 
\pi_*(\mathcal L|_{\Nlc(X, \Delta)})
\end{equation*}
is surjective on $U$ and 
\begin{equation*}
R^i\pi_*(\mathcal L|_V)=0 
\end{equation*}
holds on $U$ for every $i>0$. 
In particular, the restriction maps 
\begin{equation*}
\pi_*\mathcal L\to \pi_*(\mathcal L|_V)
\end{equation*} 
and 
\begin{equation*}
\pi_*\mathcal L\to \pi_*(\mathcal L|_{\Nlc(X, \Delta)})
\end{equation*} 
are surjective on $U$. 
\end{thm}

The result and argument in Step \ref{c-thm6.1-step1} 
in the proof of Theorem \ref{c-thm6.1} 
is the most important part of this paper. 
We will use them repeatedly in subsequent sections. 

\begin{proof}[Proof of Theorem \ref{c-thm6.1}]
In Steps \ref{c-thm6.1-step1} and \ref{c-thm6.1-step2}, we will use 
the strict support condition (see Theorem \ref{c-thm5.5} (i)) 
and the vanishing theorem (see Theorem \ref{c-thm5.5} (ii)), 
respectively. The assumption that $\mathcal L-(K_X+\Delta)$ is 
$\pi$-ample over $W$ will be used only in Step \ref{c-thm6.1-step2}. 

We take an arbitrary point $w\in W$. 
Then it is sufficient to prove the desired vanishing theorem 
over some open neighborhood of $w$ by the compactness 
of $W$. 
Therefore, we may replace $Y$ with a relatively compact 
Stein open neighborhood of $w$ 
and may assume that $\mathcal L-(K_X+\Delta)$ is 
$\pi$-ample over $Y$. 
\setcounter{step}{0}
\begin{step}\label{c-thm6.1-step1}
We can take a resolution of singularities $f\colon Z\to X$ 
of $X$ such that $f$ is projective and that 
$\Supp f^{-1}_*\Delta\cup \Exc(f)$ is a simple normal 
crossing divisor on $Z$. 
We may further assme that $f^{-1}(V)$ is a simple 
normal crossing divisor on $Z$. 
Then we can write 
\begin{equation*}
K_Z+\Delta_Z=f^*(K_X+\Delta). 
\end{equation*}
Let $T$ be the union of the irreducible 
components of $\Delta^{=1}_Z$ that 
are mapped into $V$ by $f$. 
We consider 
the following 
short exact sequence 
\begin{equation*}
0\to \mathcal O_Z(A-N-T)\to \mathcal O_Z(A-N)\to \mathcal 
O_T(A-N)\to 0, 
\end{equation*}
where $A:=\lceil -(\Delta^{<1}_Z)\rceil$ 
and $N:=\lfloor \Delta^{>1}_Z\rfloor$. 
By definition, $A$ is an effective 
$f$-exceptional 
divisor on $Z$. 
We obtain the following long exact sequence 
\begin{equation*}
\begin{split}
0&\to f_*\mathcal O_Z(A-N-T)\to f_*\mathcal O_Z(A-N)
\to f_*\mathcal O_T(A-N) 
\\&\overset{\delta}\to R^1f_*\mathcal O_Z(A-N-T)\to \cdots.  
\end{split}
\end{equation*}
Since 
\begin{equation*}
A-N-T-(K_Z+\{\Delta_Z\}+\Delta^{=1}_Z-T)=
-(K_Z+\Delta_Z)\sim_{\mathbb R}
-f^*(K_X+\Delta), 
\end{equation*}
every associated subvariety of $R^1f_*\mathcal O_Z(A-N-T)$ 
is the $f$-image of some stratum of $(Z, \{\Delta_Z\}+\Delta^{=1}_Z-T)$ 
by the strict support condition in Theorem \ref{c-thm5.5} (i). 
Since $f^{-1}(V)$ is a simple normal crossing divisor, 
there are no strata of $(Z, \{\Delta_Z\}+\Delta^{=1}_Z-T)$ that 
are mapped into $V$ by $f$. 
On the other hand, $V=f(T)$ holds by construction. 
Thus, 
the connecting homomorphism $\delta$ is a zero map. 
Hence we have a short 
exact sequence 
\begin{equation}\label{c-eq6.1}
0\to f_*\mathcal O_Z(A-N-T)\to f_*\mathcal O_Z(A-N)\to 
f_*\mathcal O_T(A-N)\to 0. 
\end{equation}
We put $\mathcal J:=f_*\mathcal O_Z(A-N-T)\subset \mathcal O_X$. 
Since $V$ is disjoint from $\Nlc(X, \Delta)$ by assumption, 
the ideal 
sheaf 
$\mathcal J$ coincides with $\mathcal I_V$ and $\mathcal J_{\NLC}(X, \Delta)$ 
in a neighborhood of $V$ and $\Nlc (X, \Delta)$, 
respectively. Therefore, 
we have 
$\mathcal J=\mathcal I_V\cdot \mathcal J_{\NLC}(X, \Delta)$. 
We note that if $V$ is empty then $\mathcal I_V=\mathcal O_X$ 
and $\mathcal J=\mathcal J_{\NLC}(X, \Delta)$. 
We put $X^*:=X\setminus \Nlc (X, \Delta)$ and $Z^*:=f^{-1}(X^*)$. 
By restricting (\ref{c-eq6.1}) to $X^*$, 
we obtain 
\begin{equation*}
0\to f_*\mathcal O_{Z^*}(A-T)
\to f_*\mathcal O_{Z^*}(A)\to f_*\mathcal O_T(A)\to 0. 
\end{equation*}
Since $f_*\mathcal O_{Z^*}(A)\simeq \mathcal O_{X^*}$, we 
have $f_*\mathcal O_T(A)\simeq \mathcal O_V$. 
This implies that $\mathcal O_V\simeq f_*\mathcal O_T$ holds. 
\end{step}
\begin{step}\label{c-thm6.1-step2}
Since 
\begin{equation*}
f^*\mathcal L+A-N-T-(K_Z+\{\Delta_Z\}+\Delta^{=1}_Z-T)\sim_{\mathbb R}
f^*(\mathcal L-(K_X+\Delta)), 
\end{equation*}
we have 
\begin{equation*}
R^i\pi_*(\mathcal J\otimes \mathcal L)\simeq 
R^i\pi_*(f_*\mathcal O_Z(A-N-T)\otimes \mathcal L)=0 
\end{equation*} 
for every $i>0$ by the vanishing theorem in 
Theorem \ref{c-thm5.5} (ii). 
If we put $V=\emptyset$, 
then we have $\mathcal J=\mathcal J_{\NLC}(X, \Delta)$. 
Therefore, 
\begin{equation*}
R^i\pi_*\left(\mathcal J_{\NLC}(X, \Delta)\otimes \mathcal L\right)=0
\end{equation*} 
holds for every $i>0$ as a special case. 
By considering the short exact sequence 
\begin{equation*}
0\to \mathcal J\to \mathcal J_{\NLC}(X, \Delta)\to \mathcal O_V\to 0, 
\end{equation*}
we obtain 
\begin{equation*}
\cdots \to R^i\pi_*(\mathcal J_{\NLC}(X, \Delta)\otimes 
\mathcal L)\to R^i\pi_*(\mathcal L|_V)\to 
R^{i+1}\pi_*\mathcal (\mathcal J\otimes \mathcal L)\to \cdots.  
\end{equation*} 
Since we have already 
checked 
\begin{equation*}
R^i\pi_*(\mathcal J_{\NLC}(X, \Delta)\otimes \mathcal L)
=R^i\pi_*(\mathcal J\otimes \mathcal L)=0 
\end{equation*}
for every $i>0$, we have 
$R^i\pi_*(\mathcal L|_V)=0$ for 
all $i>0$. 
Finally, we consider the following 
short exact sequence 
\begin{equation*}
0\to \mathcal J\to \mathcal O_X\to \mathcal O_V
\oplus \mathcal O_{\Nlc(X, \Delta)}
\to 0. 
\end{equation*}
By taking $\otimes \mathcal L$ and $R^i\pi_*$, 
we obtain that 
\begin{equation*}
0\to \pi_*(\mathcal J\otimes \mathcal L)\to
\pi_*\mathcal L\to \pi_*(\mathcal L|_V)
\oplus \pi_*(\mathcal L|_{\Nlc(X, \Delta)})\to 0 
\end{equation*}
is exact. 
\end{step}
We finish the proof. 
\end{proof}

The following remark is obvious by the proof of 
Theorem \ref{c-thm6.1}. 

\begin{rem}\label{c-rem6.2}
If $\left(\mathcal L-(K_X+\Delta)\right)|_{\pi^{-1}(y)}$ is ample 
for every $y\in Y$, then the proof of Theorem \ref{c-thm6.1} 
shows that Theorem \ref{c-thm6.1} holds over $Y$. 
This means that we can take $U=Y$ in Theorem \ref{c-thm6.1}. 
\end{rem}

We prepare one more vanishing theorem. 

\begin{thm}[{see \cite[Theorem 11.1]{fujino-fundamental}}]\label{c-thm6.3}
Let $\pi\colon X\to Y$ be a projective morphism of complex 
analytic spaces such that $X$ is a normal complex variety and 
let $W$ be a compact subset 
of $Y$. 
Let $\Delta$ be an effective $\mathbb R$-divisor 
on $X$ such that 
$K_X+\Delta$ is $\mathbb R$-Cartier. 
Let $\{C_i\}_{i\in I}$ be any set of log canonical 
centers of the pair $(X, \Delta)$. 
We put $V:=\bigcup _{i\in I}C_i$ with the 
reduced structure. 
We assume that $V$ is disjoint from the 
non-lc locus $\Nlc(X, \Delta)$ of $(X, \Delta)$. 
Let $\mathcal M$ be a line bundle on $V$ 
such that $\mathcal M-(K_X+\Delta)|_V$ is $\pi$-ample 
over $W$. 
Then there exists some open neighborhood 
$U$ of $W$ such that 
$R^i\pi_*\mathcal M=0$ holds on $U$ for every $i>0$. 
\end{thm}

\begin{proof} 
As in Theorem \ref{c-thm6.1}, it is sufficient to prove the desired 
vanishing theorem for some open neighborhood 
of any point $w\in W$. 
We will use the same notation as in the 
proof of Theorem \ref{c-thm6.1}. 
We note that 
\begin{equation*}
A-N-(K_Z+\{\Delta_Z\}+\Delta^{=1}_Z)
\sim_{\mathbb R}
-f^*(K_X+\Delta) 
\end{equation*} holds. 
We put $f_T:=f|_T\colon T\to V$. 
Then 
\begin{equation*}
f^*_T\mathcal M+A|_T-(K_T+(\{\Delta_Z\}+\Delta^{=1}_Z-T)|_T)
\sim _{\mathbb R} f^*_T(\mathcal M-(K_X+\Delta)|_V)
\end{equation*} 
holds. Note that $(T, (\{\Delta_Z\}+\Delta^{=1}_Z-T)|_T)$ 
is an analytic globally embedded simple normal crossing pair. 
Thus, by the vanishing theorem in Theorem \ref{c-thm5.5} (ii),  
\begin{equation*}
R^i\pi_*\mathcal M\simeq R^i\pi_*\left(\mathcal M
\otimes (f_T)_*\mathcal O_T(A|_T)\right)=0
\end{equation*} 
for every $i>0$. 
Here, we used the following isomorphism 
$(f|_T)_*\mathcal O_T(A|_T)\simeq \mathcal O_V$ 
obtained in Step \ref{c-thm6.1-step1} in the proof of 
Theorem \ref{c-thm6.1}. 
\end{proof}

We make two remarks on Theorem \ref{c-thm6.3}. 

\begin{rem}\label{c-rem6.4}
If $\left(\mathcal M-(K_X+\Delta)|_V\right)|_{\pi^{-1}_V(y)}$ is 
ample for every $y\in \pi(V)$, where 
$\pi_V:=\pi|_V$, then Theorem \ref{c-thm6.3} holds true over $Y$, 
that is, we can take $U=Y$ in Theorem \ref{c-thm6.3}. 
We can check it by the proof of Theorems \ref{c-thm6.1} and 
\ref{c-thm6.3}. 
\end{rem}

\begin{rem}\label{c-rem6.5}
In \cite[Theorem 11.1]{fujino-fundamental}, 
$V$ is assumed to be a {\em{minimal}} log canonical 
center of $(X, \Delta)$ which 
is disjoint from $\Nlc(X, \Delta)$. 
Moreover, the proof of \cite[Theorem 11.1]{fujino-fundamental} 
depends on \cite{bchm}. 
For the details, see \cite[Remark 11.2]{fujino-fundamental}. 
\end{rem}

\section{On log canonical centers}\label{c-sec7} 
The main purpose of this section is to prove the 
following very fundamental theorem on log canonical 
centers, which is 
an easy application of Theorem \ref{c-thm6.1} and 
its proof. 
It will play an important role in this paper. 

\begin{thm}[Basic properties of log canonical 
centers]\label{c-thm7.1}
Let $(X, \Delta)$ be a log canonical pair. 
Then the following properties hold. 
\begin{itemize}
\item[(1)] The number of log canonical centers 
of $(X, \Delta)$ is locally finite. 
\item[(2)] The intersection of two log canonical centers 
is a union of some log canonical centers. 
\item[(3)] Let $x\in X$ be any point such that $(X, \Delta)$ 
is log canonical but is not kawamata log terminal at $x$. 
Then there exists a unique minimal {\em{(}}with respect to the 
inclusion{\em{)}} log canonical center $C_x$ passing through $x$. 
Moreover, $C_x$ is normal at $x$. 
\end{itemize}
\end{thm}

\begin{proof}
We note that (1) is almost obvious by definition. 
We take an arbitrary point $x\in X$ and shrink $X$ around $x$ suitably. 
Then we may assume that there exists a projective bimeromorphic 
morphism $f\colon Y\to X$ from a smooth 
complex analytic space $Y$ such that 
$K_Y+\Delta_Y:=f^*(K_X+\Delta)$, $\Supp \Delta_Y$ is a 
simple normal crossing divisor on $Y$, 
and $\Supp \Delta_Y$ has only finitely many irreducible components 
(see \cite[Theorem 13.2]{bierstone-milman}). 
Let $\Delta^{=1}_Y:=\sum _{i\in I}\Delta_i$ be the irreducible 
decomposition. Then $C$ is a log canonical 
center of $(X, \Delta)$ if and only if $C=f(S)$, where 
$S$ is an 
irreducible component of $\Delta_{i_1}
\cap \cdots \cap \Delta_{i_k}$ for 
some $\{i_1, \ldots, i_k\}\subset I$. 
Therefore, there exists only finitely many log canonical centers 
on some open neighborhood of $x$. Thus we obtain (1). 

From now on, we will use the same notation as in 
the proof of Theorem \ref{c-thm6.1} with $Y=X$. 
Let $C_1$ and $C_2$ be two log canonical centers of $(X, \Delta)$. 
We fix a closed point $P\in C_1\cap C_2$. 
We replace $X$ with a relatively compact Stein 
open neighborhood of $P\in X$ and apply 
the argument in the proof of Theorem \ref{c-thm6.1}. 
For the 
proof of (2), it is enough to find a log canonical center $C$ 
such that $P\in C\subset C_1\cap C_2$. 
We put 
$V:=C_1\cup C_2$. 
By Step \ref{c-thm6.1-step1} in the proof of Theorem \ref{c-thm6.1}, 
we obtain $f_*\mathcal O_T\simeq \mathcal O_V$. This means 
that $f\colon T\to V$ has connected fibers. 
We note that $T$ is a simple normal crossing divisor on $Z$. 
Thus, there exist irreducible 
components $T_1$ and $T_2$ of $T$ such that  
$T_1\cap T_2\cap f^{-1}(P)\ne \emptyset$ and that 
$f(T_i)\subset C_i$ for $i=1, 2$. 
Therefore, we can find a log canonical center $C$ with 
$P\in C\subset C_1\cap C_2$. 
We finish the proof of (2). 
Finally, we will prove (3). 
The existence and the uniqueness of 
the minimal log canonical center follow from (2). 
We take the unique minimal log canonical center $C=C_x$ passing through 
$x$. We put $V:=C$. We may replace $X$ with 
a relatively compact Stein open neighborhood of $x\in X$. 
Then, by Step \ref{c-thm6.1-step1} in the proof of 
Theorem \ref{c-thm6.1}, 
we have $f_*\mathcal O_T\simeq \mathcal O_V$. 
By shrinking $V$ around $x$, we can assume that 
every stratum of $T$ dominates $V$. Let $\nu\colon 
V^\nu\to V$ be the normalization of $V$. 
By applying Hironaka's flattening theorem (see \cite{hironaka}) 
to the graph of $T\dashrightarrow V^\nu$ and 
then using the desingularization theorem 
(see \cite[Theorems 13.3 and 12.4]{bierstone-milman}), 
we can obtain the following commutative diagram: 
\begin{equation*}
\xymatrix{
T^\dag \ar[d]_-q\ar[r]^-p & T\ar[d]^-f \\ 
V^\nu \ar[r]_-\nu& V, 
}
\end{equation*}
where $p\colon T^\dag \to T$ is a projective 
bimeromorphic morphism such that 
$T^\dag$ is simple normal crossing 
with $p_*\mathcal O_{T^\dag}\simeq \mathcal O_T$ 
(see \cite[Lemma 2.15]{fujino-analytic-vanishing}). 
Hence 
\begin{equation*}
\mathcal O_V\hookrightarrow \nu_* \mathcal O_{V^\nu} 
\hookrightarrow \nu_*q_*\mathcal O_{T^\dag}\simeq 
f_*p_*\mathcal O_{T^\dag}\simeq f_*\mathcal O_T\simeq \mathcal O_V. 
\end{equation*} 
This implies that $\mathcal O_V\simeq \nu_*\mathcal O_{V^\nu}$ 
holds, that is, $V$ is normal. 
Thus we obtain (3). 
\end{proof}

By the above proof of Theorem \ref{c-thm7.1}, 
we see that Theorem \ref{c-thm7.1} (2) and (3) are consequences of 
the strict support condition in Theorem \ref{c-thm5.5} (i). 

\section{Non-vanishing theorem}\label{c-sec8}
In this section, we will explain the non-vanishing theorem 
for projective morphisms between complex analytic spaces. 

\begin{thm}[{see \cite[Theorem 12.1]{fujino-fundamental}}]\label{c-thm8.1}
Let $\pi\colon X\to Y$ be a projective morphism 
of complex analytic spaces such that 
$X$ is a normal complex variety and let $W$ be a compact 
subset of $Y$. 
Let $\Delta$ be an effective 
$\mathbb R$-divisor on $X$ 
such that $K_X+\Delta$ is 
$\mathbb R$-Cartier. 
Let $L$ be a Cartier divisor on $X$ 
which is $\pi$-nef over $W$, that is, 
$L|_{\pi^{-1}(w)}$ is nef for every $w\in W$. 
We assume that 
\begin{itemize}
\item[(i)] $aL-(K_X+\Delta)$ is $\pi$-ample over $W$ for 
some positive real number $a$, and 
\item[(ii)] $\mathcal O_{\Nlc(X, \Delta)}(mL)$ is 
$\pi|_{\Nlc(X, \Delta)}$-generated 
over some open neighborhood of $W$ 
for every $m\gg 0$. 
\end{itemize} 
Then for every $m\gg 0$ 
there exists a relatively compact open neighborhood 
$U_m$ of $W$ over which  
the relative base locus $\Bs _{\pi}|mL|$ contains 
no log canonical centers of $(X, \Delta)$ 
and is disjoint from $\Nlc(X, \Delta)$. 
We note that the open subset $U_m$ depends on $m$. 
\end{thm}

We first prepare the following useful lemma, which is new, for the 
proof of Theorem \ref{c-thm8.1}. For the 
details of the theory of quasi-log schemes, 
see \cite[Chapter 6]{fujino-foundations}, 
\cite{fujino-cone}, and \cite{fujino-quasi}. 

\begin{lem}\label{c-lem8.2}
Let $\pi\colon X\to Y$ be a projective morphism 
between complex analytic spaces such that 
$X$ is a normal complex variety and let $f\colon (Z, \Delta_Z)\to X$ be a 
projective morphism from an analytic globally embedded 
simple normal crossing pair $(Z, \Delta_Z)$ such that 
$\Delta_Z$ is a subboundary $\mathbb R$-divisor 
on $Z$ and is a finite $\mathbb R$-linear combination 
of Cartier divisors, 
the natural map $\mathcal O_X\to f_*\mathcal O_Z(\lceil -(\Delta^{<1}_Z)
\rceil)$ is an isomorphism, and $K_Z+\Delta_Z\sim _{\mathbb R} 
f^*\omega$ holds for some $\mathbb R$-line bundle 
$\omega$ on $X$. 
Let $y$ be an analytically sufficiently general point of $\pi(X)$. 
Then 
\begin{equation*}
\left(X_y, \omega|_{X_y}, f_y\colon (Z_y, \Delta_{Z_y})\to X_y\right)
\end{equation*}
is a projective quasi-log canonical pair, where 
$X_y:=\pi^{-1}(y)$, $Z_y:=(\pi\circ f)^{-1}(y)$, 
$f_y:=f|_{X_y}$, and $\Delta_{Z_y}:=\Delta_Z|_{Z_y}$. 
\end{lem}

This lemma is also a consequence of the 
strict support condition in Theorem \ref{c-thm5.5} (i). 

\begin{proof}[Proof of Lemma \ref{c-lem8.2}]
By replacing $Y$ with $\pi(X)$, we may assume that 
$\pi$ is surjective and $Y$ is a complex variety. 
By replacing $Y$ with a Zariski open subset of $Y$, 
we may further assume that 
$Y$ is smooth. By replacing $Y$ with a suitable Zariski open subset, 
we may assume that $\pi\circ f$ is flat. 
Then, by replacing $Y$ with a suitable Zariski open subset again, 
we may assume that 
every stratum of $(Z, \Supp \Delta_Z)$ is smooth over $Y$. 
We take an arbitrary point $y\in Y$. 
Then $(Z_y, \Delta_{Z_y})$ is an analytic simple normal crossing pair. 
From now on, we will prove that 
\begin{equation*}
\left(X_y, \omega|_{X_y}, f_y\colon (Z_y, \Delta_{Z_y})\to X_y\right)
\end{equation*}
is a projective quasi-log canonical pair. 
Without loss of generality, we may assume that $Y$ is a polydisc $\Delta^m$ 
with $y=0\in \Delta^m$. 
Let $(z_1, \ldots, z_m)$ be the local coordinate system of $\Delta^m$. 
Then $\left((\pi\circ f)^*z_i=0\right)$ does not contain any strata of 
$(Z, \Supp \Delta_Z)$. 
Therefore, 
\begin{equation*}
\pi^*z_i\times\colon R^pf_*\mathcal O_Z(\lceil -(\Delta^{<1}_Z)
\rceil)\to R^pf_*\mathcal O_Z(\lceil -(\Delta^{<1}_Z)
\rceil)\end{equation*} 
is injective for every $i$ and every $p$ since 
$\lceil -(\Delta^{<1}_Z)\rceil-(K_Z+\{\Delta_Z\}
+\Delta^{=1}_Z)\sim _{\mathbb R} -f^*\omega$. We put $X_1:=(\pi^*z_1=0)$ 
and $Z_1:=\left((\pi\circ f)^*z_1=0\right)$. 
Since 
\begin{equation*}
\pi^*z_1\times\colon R^1f_*\mathcal O_Z(\lceil -(\Delta^{<1}_Z)
\rceil)\to R^1f_*\mathcal O_Z(\lceil -(\Delta^{<1}_Z)
\rceil)\end{equation*} 
is injective, we obtain the following short exact sequence: 
\begin{equation*}
\xymatrix{
0\ar[r]& f_*\mathcal O_Z(\lceil -(\Delta^{<1}_Z)\rceil)
\ar[r]^-{\times \pi^*z_1}
&
f_*\mathcal O_Z(\lceil -(\Delta^{<1}_Z)\rceil) 
\ar[r]
& f_*\mathcal O_{Z_1}(\lceil -(\Delta^{<1}_{Z_1})\rceil)
\ar[r]
& 0, 
}
\end{equation*}
where $\Delta_{Z_1}=\Delta_Z|_{Z_1}$. 
This implies that 
the natural map 
$\mathcal O_{X_1}\to f_*\mathcal O_{Z_1}(\lceil -(\Delta^{<1}_{Z_1})\rceil)$ 
is an isomorphism. 
By repeating this argument, we finally obtain that 
$\mathcal O_{X_y}\simeq (f_y)_*\mathcal O_{Z_y}(\lceil 
-(\Delta^{<1}_{Z_y})\rceil)$ holds. 
By \cite[Theorem 4.9]{fujino-pull-back}, this means that 
\begin{equation*}
\left(X_y, \omega|_{X_y}, f_y\colon (Z_y, \Delta_{Z_y})\to X_y\right)
\end{equation*}
is a projective quasi-log canonical pair. 
\end{proof}

Let us prove Theorem \ref{c-thm8.1}. 

\begin{proof}[Proof of Theorem \ref{c-thm8.1}]
We divide the proof into several small steps. 
\setcounter{step}{0}
\begin{step}\label{c-thm8.1-step1}
By shrinking $Y$ suitably, 
we may assume that there exists a positive integer $m_1$ 
such that $\mathcal O_{\Nlc(X, \Delta)}(mL)$ is 
$\pi|_{\Nlc(X, \Delta)}$-generated for every $m\geq m_1$ by 
(ii). We may further assume that 
$aL-(K_X+\Delta)$ is $\pi$-ample over $Y$. 
\end{step}
\begin{step}\label{c-thm8.1-step2}
In this step, we will prove the following claim. 
\begin{claim}\label{c-thm8.1-claim}
There exists a positive integer $m_2$ such that 
$\pi_*\mathcal O_V(mL)\ne 0$ 
holds for every $m\geq m_2$, where 
$V$ is any minimal log canonical center 
of $(X, \Delta)$ such that 
$\pi(V)\cap W\ne \emptyset$ and that $V\cap \Nlc(X, \Delta)
=\emptyset$ over some open neighborhood of $W$. 
\end{claim}
\begin{proof}[Proof of Claim]
We note that the number of the minimal log canonical 
centers $V$ of $(X, \Delta)$ 
with $\pi(V)\cap W\ne \emptyset$ is finite. 
We take a minimal log canonical center $V$ such that 
$\pi(V)\cap W\ne \emptyset$ and 
that $V\cap \Nlc(X, \Delta)=\emptyset$ 
over some open neighborhood of $W$. 
Let $y$ be an arbitrary point of $\pi(V)\cap W$. 
It is sufficient to prove $\pi_*\mathcal O_V(mL)\ne 0$ 
on a small open neighborhood of $y$. 
Therefore, we can replace $Y$ with a small 
relatively compact Stein open neighborhood 
of $y$.  
Thus, by Step \ref{c-thm6.1-step1} in the proof of Theorem \ref{c-thm6.1}, 
we can construct a projective surjective morphism 
$f\colon (T, \Delta_T)\to V$ from an analytic globally 
embedded simple normal crossing pair $(T, \Delta_T)$ such that 
$\Delta_T$ is a subboundary $\mathbb R$-divisor on $T$, 
the natural map $\mathcal O_V\to f_*\mathcal O_T(\lceil 
-(\Delta^{<1}_T)\rceil)$ is an isomorphism, and 
$K_T+\Delta_T\sim _{\mathbb R} f^*(K_X+\Delta)|_V$ holds. 
Thus, by Lemma \ref{c-lem8.2}, an analytically sufficiently general fiber 
$F$ of $\pi\colon V\to \pi(V)$ is a projective 
quasi-log canonical pair. 
By Lemma \ref{c-lem3.5}, we may assume that 
$L|_F$ is nef.  
Therefore, by the basepoint-free theorem 
for quasi-log canonical pairs (see \cite[Theorem 6.5.1]{fujino-foundations}), 
there 
exists a positive integer $m_2$ such that 
$|mL|_F|$ is basepoint-free 
for every $m\geq m_2$.  
This implies that $\pi_*\mathcal O_V(mL)\ne 0$ for 
every $m\geq m_2$. This is what we wanted. 
\end{proof}
\end{step}
\begin{step}\label{c-thm8.1-step3}
We put $m_0:=\max \{a, m_1, m_2\}$. 
Let $m$ be any positive integer with $m\geq m_0$. 
Since $aL-(K_X+\Delta)$ is $\pi$-ample 
over $W$ and $L$ is $\pi$-nef over $W$, 
$mL-(K_X+\Delta)$ is $\pi$-ample over $W$. 
By Theorem \ref{c-thm6.1}, we can find 
an open neighborhood $U_m$ of $W$ such that 
the vanishing theorem holds for $mL$ over $U_m$. 
Without 
loss of generality, 
we may assume that every minimal log canonical center 
$V$ of $(X, \Delta)$ with $\pi(V)\cap U_m\ne \emptyset$ always 
satisfies $\pi(V)\cap W\ne \emptyset$ by 
shrinking $U_m$ suitably. 
\end{step}
\begin{step}\label{c-thm8.1-step4}
In this final step, we will prove 
that over $U_m$ the relative base locus 
$\Bs_{\pi}|mL|$ contains no log canonical 
centers of $(X, \Delta)$ and is disjoint 
from $\Nlc(X, \Delta)$ for every $m\geq m_0$.  

By the vanishing theorem (see Theorem \ref{c-thm6.1}), 
we have $R^1\pi_*\left(\mathcal J_{\NLC}(X, \Delta)\otimes 
\mathcal O_X(mL)\right)=0$ on $U_m$. 
Thus the restriction map 
\begin{equation*}
\pi_*\mathcal O_X(mL)\to \pi_*\mathcal O_{\Nlc(X, \Delta)}(mL)
\end{equation*} 
is surjective on $U_m$. This implies that 
the relative base locus  
$\Bs_{\pi}|mL|$ is disjoint from $\Nlc(X, \Delta)$ over 
$U_m$. Let $V$ be a minimal log canonical 
center of $(X, \Delta)$ with $\pi(V)\cap U_m\ne \emptyset$. 
If $V\cap \Nlc(X, \Delta)\ne \emptyset$ over $U_m$, then 
$V\not\subset \Bs_{\pi}|mL|$ since 
$\Nlc(X, \Delta)\cap \Bs_{\pi}|mL|=\emptyset$ over 
$U_m$. 
Hence we may assume that $V\cap 
\Nlc(X, \Delta)=\emptyset$ over $U_m$. 
In this case, $\pi_*\mathcal O_V(mL)\ne 0$ by 
Claim in Step \ref{c-thm8.1-step2}. 
On the other hand, by the vanishing theorem 
(see Theorem \ref{c-thm6.1}), 
the restriction map 
\begin{equation*}
\pi_*\mathcal O_X(mL) \to \pi_*\mathcal O_V(mL)
\end{equation*} 
is surjective on $U_m$. 
This implies that $V\not\subset \Bs_{\pi}|mL|$. 
Hence $\Bs_{\pi}|mL|$ contains no log canonical centers of $(X, \Delta)$ 
over $U_m$. 
\end{step}
We finish the proof. 
\end{proof}

We make an important remark on Theorem \ref{c-thm8.1}. 

\begin{rem}\label{c-rem8.3}
In Step \ref{c-thm8.1-step3} in the proof of 
Theorem \ref{c-thm8.1}, the condition that 
$mL-(K_X+\Delta)$ is $\pi$-ample 
over $w\in W$ only implies that $mL-(K_X+\Delta)$ 
is $\pi$-ample over some open neighborhood 
$\mathcal U^m_w$ of $w$ in $Y$. 
We note that $\mathcal U^m_w$ depends on $m$. 
Therefore, $U_m$ in 
Theorem \ref{c-thm8.1} also depends on $m$. 
\end{rem}

For kawamata log terminal pairs, the non-vanishing theorem 
is formulated as follows. 

\begin{thm}[Non-vanishing theorem for 
kawamata log terminal pairs]\label{c-thm8.4}
Let $\pi\colon X\to Y$ be a projective morphism 
of complex analytic spaces such that 
$X$ is a normal complex variety and let $W$ be a compact 
subset of $Y$. 
Let $\Delta$ be an effective 
$\mathbb R$-divisor on $X$ 
such that $(X, \Delta)$ is kawamata log terminal. 
Let $L$ be a Cartier divisor on $X$ 
which is $\pi$-nef over $W$. 
We assume that 
$aL-(K_X+\Delta)$ is $\pi$-ample over $W$ for 
some positive real number $a$. 
Then $\pi_*\mathcal O_X(mL)\ne 0$ holds for every $m\gg 0$ 
\end{thm}

\begin{proof} 
By shrinking $Y$ around $W$ suitably, 
we may assume that $aL-(K_X+\Delta)$ is $\pi$-ample 
over $Y$. 
Let $F$ be an analytically sufficiently general fiber 
of $\pi\colon X\to \pi(X)$. 
Then $(F, \Delta|_F)$ is kawamata log terminal. 
By Lemma \ref{c-lem3.5}, 
we may assume that $L|_F$ is nef.  
Hence $|mL|_F|$ is basepoint-free for every $m\gg 0$ by 
the usual Kawamata--Shokurov basepoint-free theorem for 
projective kawamata log terminal pairs. 
Thus, we obtain that $\pi_*\mathcal O_X(mL)\ne 0$ for every $m\gg 0$. 
This is what we wanted. 
\end{proof}

We will use Theorems \ref{c-thm8.1} and \ref{c-thm8.4} 
in the proof of the basepoint-freeness in Section \ref{c-sec9}. 

\section{Basepoint-free theorem}\label{c-sec9}
In this section, we will explain the basepoint-free 
theorem in the complex analytic setting. 

\begin{thm}[{see \cite[Theorem 13.1]{fujino-fundamental}}]\label{c-thm9.1}
Let $\pi\colon X\to Y$ be a projective morphism 
of complex analytic spaces such that 
$X$ is a normal complex variety and let $W$ be a compact 
subset of $Y$. 
Let $\Delta$ be an effective 
$\mathbb R$-divisor on $X$ 
such that $K_X+\Delta$ is 
$\mathbb R$-Cartier. 
Let $L$ be a Cartier divisor on $X$ 
which is $\pi$-nef over 
$W$. 
We assume that 
\begin{itemize}
\item[(i)] $aL-(K_X+\Delta)$ is $\pi$-ample 
over $W$ for 
some positive real number $a$, and 
\item[(ii)] $\mathcal O_{\Nlc(X, \Delta)}(mL)$ is 
$\pi|_{\Nlc(X, \Delta)}$-generated 
over some open neighborhood of $W$ for every $m\gg 0$. 
\end{itemize} 
Then there exists a relatively compact open neighborhood 
$U$ of $W$ such that 
$\mathcal O_X(mL)$ is $\pi$-generated 
over $U$ for 
every $m\gg 0$. 
\end{thm}

Theorem \ref{c-thm9.1} is a consequence of the vanishing 
theorem (see 
Theorem \ref{c-thm6.1}) 
and the non-vanishing theorem (see Theorems \ref{c-thm8.1} and 
\ref{c-thm8.4}). 

\begin{proof}[Proof of Theorem \ref{c-thm9.1}]
We take an arbitrary point $y\in W$. 
It is sufficient to prove that $\mathcal O_X(mL)$ is 
$\pi$-generated for every $m \gg 0$ 
over some relatively compact Stein open neighborhood 
of $y$. 
Hence, we will sometimes freely replace $Y$ with a suitable 
relatively compact Stein open neighborhood 
of $y$ without mentioning it explicitly throughout this proof. 
So, from now on, we assume that $Y$ is Stein and that $\pi$ is 
surjective. 

\setcounter{step}{0}
\begin{step}\label{c-thm9.1-step1}
Let $p$ be a prime number. 
In this step, we will prove that 
there exists a positive integer $k$ such that 
$\Bs_{\pi}|p^kL|=\emptyset$ holds 
over some open neighborhood 
of $y$. 

By putting $W:=\{y\}$ and 
using the non-vanishing theorem (see Theorems \ref{c-thm8.1} and 
\ref{c-thm8.4}), 
we obtain $|p^{k_1}L|\ne 
\emptyset$ for some positive integer $k_1$. 
If $\Bs_{\pi}|p^{k_1}L|=\emptyset$, 
then there is nothing to prove. 
Hence we may assume that 
$\Bs_{\pi}|p^{k_1}L|\ne \emptyset$. 
By Theorem \ref{c-thm8.1}, 
we may assume that $\Bs_{\pi}|p^{k_1}L|$ contains no 
log canonical centers of $(X, \Delta)$ and is disjoint from $\Nlc(X, \Delta)$ 
after shrinking $Y$ suitably. 
By shrinking $Y$ around $W=\{y\}$, 
we may assume that $\pi(V)\cap W\ne \emptyset$, 
where $V$ is any irreducible component 
of $\Bs_{\pi}|p^{k_1}L|$. 
Without loss of generality, we may further assume that 
$aL-(K_X+\Delta)$ is $\pi$-ample over $Y$. 
We take general members $D_1, \ldots, D_{n+1}$ 
of $|p^{k_1}L|$ with $n=\dim X$. 
We put $D:=\sum _{i=1}^{n+1}D_i$. 
We may assume that $(X, \Delta+D)$ is log canonical 
outside $\Bs_{\pi}|p^{k_1}L|\cup 
\Nlc (X, \Delta)$. 
Let $x\in X$ be any point of $\Bs_{\pi}|p^{k_1}L|$. 
Then, by Lemma \ref{c-lem2.3}, $(X, \Delta+D)$ is not log canonical 
at $x$.  
We put 
\begin{equation*}
c:=\sup \left\{t\in \mathbb R \mid \text{$(X, \Delta+tD)$ is log canonical 
at $\pi^{-1}(y)\cap \left(X\setminus \Nlc(X, \Delta)\right)$}\right\}. 
\end{equation*} 
Then we can check that $0<c<1$. 
By shrinking $Y$ around $W=\{y\}$ suitably again, 
we may assume that $(X, \Delta+cD)$ is 
log canonical outside $\Nlc(X, \Delta)$. 
By Lemma \ref{c-lem4.3} and its proof, 
we see that $\mathcal J_{\NLC}(X, \Delta+cD)=\mathcal J_{\NLC}(X, 
\Delta)$ holds. 
By construction, 
\begin{equation*}
\left(c(n+1)p^{k_1}+a\right)L-(K_X+\Delta+cD)\sim 
_{\mathbb R} aL-(K_X+\Delta)
\end{equation*} 
is $\pi$-ample over $Y$. 
By construction again, 
there exists a log canonical center $V$ of $(X, \Delta+cD)$ 
which is contained in $\Bs_{\pi}|p^{k_1}L|$ 
such that $\pi(V)\cap W\ne \emptyset$. 
By the non-vanishing theorem (see Theorem \ref{c-thm8.1}), 
we can find $k_2>k_1$ such that 
$\Bs_{\pi}|p^{k_2}L|\subsetneq \Bs_{\pi} |p^{k_1}L|$. 
Here, we replaced $Y$ with a smaller open neighborhood of 
$y$. 
By repeating this process finitely many times, 
we can find $k$ such that $\Bs_{\pi} |p^kL|=\emptyset$ 
over some open neighborhood of $y$. 
\end{step}
\begin{step}\label{c-thm9.1-step2}
We take another prime number $p'$. 
Then there exists $k'$ such that 
$\Bs_{\pi}|p'^{k'}L|=\emptyset$ over some 
open neighborhood of $y$ by Step \ref{c-thm9.1-step1}. 
This implies that there exist a positive integer $m_0$ and 
some open neighborhood $U_y$ of $y$ such that 
for every $m\geq m_0$ the relative base locus 
$\Bs_{\pi}|mL|$ is empty over $U_y$. 
\end{step}
Since $W$ is compact, we obtain a desired open 
neighborhood $U$ of $W$. We finish the proof. 
\end{proof}

\begin{rem}\label{c-rem9.2}
Although the non-vanishing theorem (see Theorems \ref{c-thm8.1} 
and \ref{c-thm8.4}) and 
the basepoint-free theorem (see Theorem \ref{c-thm9.1}) 
were formulated for Cartier divisors $L$, 
it is obvious that they hold true even for line bundles $\mathcal L$. 
We will sometimes use the basepoint-free theorem for line bundles 
in subsequent sections. 
\end{rem}

\section{Rationality theorem}\label{c-sec10}

In this section, we will 
explain the rationality theorem 
in the complex analytic setting. 
Although the proof of \cite[Theorem 15.1]{fujino-fundamental}, 
which is the rationality theorem in the algebraic setting,  
works with some suitable modifications, 
we will explain the details for the reader's convenience. 
This is because the proof of the rationality theorem is complicated. 

\begin{thm}[{Rationality theorem, see 
\cite[Theorem 15.1]{fujino-fundamental}}]\label{c-thm10.1}
Let $\pi\colon X\to Y$ be a projective morphism 
of complex analytic spaces such that 
$X$ is a normal complex variety and let 
$W$ be a compact subset of $Y$. 
Let $\Delta$ be an effective 
$\mathbb Q$-divisor on $X$ such that 
$K_X+\Delta$ is $\mathbb Q$-Cartier. 
Let $H$ be a $\pi$-ample 
Cartier divisor on $X$. 
Assume that 
$K_X+\Delta$ is not $\pi$-nef over $W$ and that 
$r$ is a positive number such that 
\begin{itemize}
\item[(i)] $H+r(K_X+\Delta)$ is $\pi$-nef 
over $W$ but is not $\pi$-ample over $W$, and 
\item[(ii)] $\left(H+r(K_X+\Delta)\right)|_{\Nlc(X, \Delta)}$ 
is $\pi|_{\Nlc(X, \Delta)}$-ample over $W$. 
\end{itemize} 
Then $r$ is a rational number, and in reduced form, 
it has denominator at most $a(d+1)$, where 
$d:=\max_{w\in W}\dim \pi^{-1}(w)$ and 
$a$ is a positive integer such that 
$a(K_X+\Delta)$ is a Cartier divisor in a neighborhood of $\pi^{-1}(W)$. 
\end{thm}

In the proof of Theorem \ref{c-thm10.1}, we will use the 
following elementary lemmas. 
We do not prove Lemma \ref{c-lem10.2} here. 
For the proof, see, for example, \cite{fujino-fundamental}. 

\begin{lem}[{\cite[Lemma 3.19]{kollar-mori}}]\label{c-lem10.2}
Let $P(x, y)$ be a non-trivial 
polynomial of degree $\leq d$ and 
assume that 
$P$ vanishes for all sufficiently large integral solutions of 
$0<ay-rx<\varepsilon$ for some fixed positive integer $a$ and 
positive $\varepsilon$ for some $r\in \mathbb R$. 
Then $r$ is rational, and 
in reduced form, $r$ has denominator $\leq 
a(d+1)/\varepsilon$. 
\end{lem}

\begin{lem}\label{c-lem10.3}  
Let $F$ be a projective variety and let $D_1$ and $D_2$ be Cartier 
divisors on $X$. 
Let us consider the Hilbert polynomial 
\begin{equation*}
P(u_1, u_2):=\chi (F, \mathcal O_F(u_1 D_1+u_2D_2)). 
\end{equation*} 
If $D_1$ is ample, then 
$P(u_1, u_2)$ is a non-trivial 
polynomial of total degree $\leq \dim F$. 
It is because $P(u_1, 0)=\dim_{\mathbb C} 
H^0(F, \mathcal O_F(u_1D_1))\not\equiv 0$ if $u_1$ 
is sufficiently large. 
\end{lem}

Let us start the proof of Theorem \ref{c-thm10.1}. 

\begin{proof}[Proof of Theorem \ref{c-thm10.1}] 
Throughout this proof, we can freely shrink $Y$ around $W$ suitably. 
Hence we sometimes will replace $Y$ with a small 
open neighborhood of $W$ without mentioning it explicitly. 

Let $m$ be a positive integer such that 
$H':=mH$ is $\pi$-very ample after shrinking $Y$ around $W$ suitably. 
If $H'+r'(K_X+\Delta)$ is $\pi$-nef over $W$ but is not 
$\pi$-ample over $W$, 
and $\left(H'+r'(K_X+\Delta)\right)|_{\Nlc(X, \Delta)}$ 
is $\pi|_{\Nlc(X, \Delta)}$-ample over $W$, 
then 
\begin{equation*}
H+r(K_X+\Delta)=\frac{1}{m}(H'+r'(K_X+\Delta))
\end{equation*} 
holds. 
This implies that $r=(1/m)r'$ holds. 
Therefore, it is obvious that 
$r$ is rational if and only if $r'$ is rational. 
We further assume that $r'$ has denominator $v$. 
Then $r$ has denominator dividing $mv$. 
Since $m$ can be an arbitrary sufficiently large positive integer, 
this means that $r$ has denominator dividing $v$. 
Hence, by replacing $H$ with $mH$, we may assume that 
$H$ is $\pi$-very ample. 

For each $(p, q)\in \mathbb Z^2$, 
we put $M(p, q):=pH+qa(K_X+\Delta)$ and define 
\begin{equation*}
L(p, q):=\Supp (\Coker (\pi^*\pi_*
\mathcal O_X(M(p, q))\to \mathcal O_X(M(p, q)))). 
\end{equation*}
By definition, $L(p, q)=X$ holds 
if and only if $\pi_*\mathcal O_X(M(p, q))= 0$. 

\setcounter{cla}{0}
\begin{cla}\label{c-thm10.1-claim1}
Let $\varepsilon$ be a positive number. 
For $(p, q)$ sufficiently 
large and $0<aq-rp<\varepsilon$, $L(p, q)$ is the same 
subset of $X$ after shrinking $Y$ around $W$ suitably. 
We call this subset $L_0$. 
Let $I\subset \mathbb Z^2$ be the set of $(p, q)$ for 
which $0<aq-rp<1$ and $L(p, q)=L_0$. 
Then we note that $I$ contains all sufficiently 
large $(p, q)$ with $0<aq-rp<1$. 
\end{cla}

\begin{proof}[Proof of Claim \ref{c-thm10.1-claim1}]
We fix $(p_0, q_0)\in \mathbb Z^2$ such that 
$p_0>0$ and $0<aq_0-rp_0<1$. 
Since $H$ is $\pi$-very ample, 
there exists a positive integer 
$m_0$ such that 
$\mathcal O_X(mH+ja(K_X+\Delta))$ is $\pi$-generated 
for every $m>m_0$ and 
every $0\leq j\leq q_0-1$ after shrinking $Y$ around $W$ suitably. 
Let $M$ be the round-up of 
\begin{equation*}
\left.\left(m_0+\frac{1}{r}\right)\right\slash\left(\frac{a}{r}-\frac{p_0}{q_0}\right). 
\end{equation*}
If $(p', q')\in \mathbb Z^2$ such that 
$0<aq'-rp'<1$ and $q'\geq M+q_0-1$, 
then we can write 
\begin{equation*}
p'H+q'a(K_X+\Delta)=k(p_0H+q_0a(K_X+\Delta))+(lH+ja(K_X+\Delta))
\end{equation*} 
for some $k\geq 0$, $0\leq j\leq q_0-1$ 
with $l> m_0$. 
It is because we can uniquely write $q'=kq_0+j$ with $0\leq j\leq q_0-1$. Thus, 
we have $kq_0\geq M$. So, 
we obtain 
\begin{equation*}
l=p'-kp_0>\frac{a}{r}q'-\frac{1}{r}-(kq_0)\frac{p_0}{q_0}
\geq \left(\frac{a}{r}-\frac{p_0}{q_0}\right)M-\frac{1}{r}\geq m_0. 
\end{equation*}
Therefore, 
$L(p', q')\subset L(p_0, q_0)$. 
We note that we can use the noetherian induction over a relatively 
compact open neighborhood of $W$ (see \cite[0.40. Corollary]{fischer}). 
Therefore, after shrinking $Y$ around $W$ suitably again, 
we obtain the desired closed subset $L_0\subset X$. 
We can check that the subset 
$I\subset \mathbb Z^2$ contains all sufficiently large $(p, q)$ with 
$0<aq-rp<1$ without any difficulties. 
\end{proof}

\begin{cla}\label{c-thm10.1-claim2} 
We have $L_0\cap \Nlc (X, \Delta)=\emptyset$. 
\end{cla}
\begin{proof}[Proof of Claim \ref{c-thm10.1-claim2}]
We take $(\alpha, \beta)\in \mathbb Q^2$ such that 
$\alpha>0$, $\beta>0$, 
and $\beta a/\alpha>r$ is sufficiently close to 
$r$. 
Then $(\alpha H+\beta a(K_X+\Delta))|_{\Nlc(X, \Delta)}$ is 
$\pi|_{\Nlc(X, \Delta)}$-ample over $W$ because 
$(H+r(K_X+\Delta))|_{\Nlc(X, \Delta)}$ is 
$\pi|_{\Nlc(X, \Delta)}$-ample over $W$. 
We take any point $w\in W$. 
Then it is sufficient to prove that $L_0\cap \Nlc(X, \Delta)=\emptyset$ 
holds over some open neighborhood of $w$. 
From now on, we will freely shrink $Y$ around $w$ without mentioning it 
explicitly. We take a sufficiently large and divisible positive 
integer $m'$ such that 
\begin{equation*} 
m'(\alpha H+\beta a (K_X+\Delta))|_{\Nlc(X, \Delta)}
\end{equation*} 
is $\pi|_{\Nlc(X, \Delta)}$-very ample. 
We put $(p_0, q_0):=(m'\alpha, m'\beta)$ and 
apply the argument in the proof of Claim \ref{c-thm10.1-claim1}. 
Thus, if $0<aq-rp<1$ and $(p, q)\in \mathbb Z^2$ is 
sufficiently large, then we can write 
\begin{equation*} 
M(p, q)=mM(\alpha, \beta)+(M(p, q)-mM(\alpha, \beta))
\end{equation*} 
such 
that $M(p, q)-mM(\alpha, \beta)$ is $\pi$-very ample and 
that 
\begin{equation*} 
m(\alpha H+\beta a (K_X+\Delta))|_{\Nlc(X, \Delta)}
\end{equation*} 
is 
also $\pi|_{\Nlc(X, \Delta)}$-very ample. 
Hence, $\mathcal O_{\Nlc(X, \Delta)}(M(p, q))$ is $\pi$-very 
ample. 
We note that 
\begin{equation*}
M(p, q)-(K_X+\Delta)=pH+(qa-1)(K_X+\Delta)
\end{equation*} 
is $\pi$-ample over some open neighborhood of $w$ because 
$(p, q)$ is sufficiently large and $aq-rp<1$.  
Thus, by the vanishing theorem:~Theorem \ref{c-thm6.1}, 
the restriction map 
\begin{equation*} 
\pi_*\mathcal O_X(M(p, q))\to 
\pi_*\mathcal O_{\Nlc(X, \Delta)}(M(p, q))
\end{equation*} 
is surjective. 
Therefore, $L(p, q)\cap \Nlc(X, \Delta)=\emptyset$ holds 
over some open neighborhood of $w$. 
By Claim \ref{c-thm10.1-claim1}, we have 
$L_0\cap \Nlc(X, \Delta)=\emptyset$ over some 
open neighborhood of $w$. 
Since $w$ is an arbitrary point of $W$, 
$L_0\cap \Nlc(X, \Delta)=\emptyset$ holds 
over some open neighborhood of $W$. 
This is what we wanted. 
\end{proof}

\begin{cla}\label{c-thm10.1-claim3} 
We assume that $r$ is not rational or that 
$r$ is rational and has denominator $>a(d+1)$ in reduced form. 
Then, for $(p, q)$ sufficiently large and $0<aq-rp<1$, 
$\mathcal O_X(M(p, q))$ is $\pi$-generated 
at general points of every log canonical center 
of $(X, \Delta)$. 
\end{cla}

We will explain the proof of Claim \ref{c-thm10.1-claim3} in detail 
because we have to change the proof of Claim 3 
in the proof of \cite[Theorem 15.1]{fujino-fundamental} slightly. 

\begin{proof}[Proof of Claim \ref{c-thm10.1-claim3}]
After shrinking $Y$ around $W$ suitably, 
it is sufficient to consider minimal log canonical centers $C$ of $(X, \Delta)$ 
such that $\pi(C)\cap W\ne \emptyset$. 
By Claim \ref{c-thm10.1-claim2}, we may assume that 
$C\cap \Nlc(X, \Delta)=\emptyset$ holds. 
We take a point $w\in \pi(C) \cap W$. 
It is sufficient to consider everything 
over some small open enighborhood 
of $w$ in $Y$. 
We take an analytically sufficiently general fiber $F$ of $C\to 
\pi(C)$. 
Then we may assume that 
$(H+r(K_X+\Delta))|_F$ and 
$(H+r(K_X+\Delta))|_{\pi^{-1}(\pi(F))}$ are 
both nef by Lemma \ref{c-lem3.5} (see also Remark \ref{c-rem3.7}). 
We note that 
\begin{equation*}
\begin{split}
M(p, q)-(K_X+\Delta)&=pH+(qa-1)(K_X+\Delta)
\\ &=\left(p-\frac{qa-1}{r}\right)H+\frac{(qa-1)}{r}(H+r(K_X+\Delta))
\end{split}
\end{equation*} 
holds. 
Therefore, if 
$aq-rp<1$ and $(p, q)$ is sufficiently large, 
then we see that 
\begin{equation*}
\left(M(p, q)-(K_X+\Delta)\right)|_{\pi^{-1}(\pi(F))}
\end{equation*} 
is ample. 
We note that 
\begin{equation*}
P_F(p, q):=\chi (F, \mathcal 
O_F(M(p, q)))
\end{equation*} 
is a non-zero polynomial of degree 
at most $\dim F\leq d$ by Lemma \ref{c-lem10.3}. 
We also note 
that $F$ is an analytically sufficiently general fiber of $C\to \pi(C)$. 
By Lemma \ref{c-lem10.2}, there exists $(p, q)$ such that 
$P_F(p, q)\ne 0$, 
$(p, q)$ sufficiently large, and $0<aq-rp<1$. 
By the $\pi$-ampleness of $M(p, q)-(K_X+\Delta)$ 
over some open neighborhood of $\pi(F)$, 
\begin{equation*}
P_F
(p, q)=
\chi (F, \mathcal O_F(M(p, q)))
=\dim _{\mathbb C}H^0(F, \mathcal O_F(M(p, q)))
\end{equation*} 
and 
\begin{equation*}
\pi_*\mathcal O_X(M(p, q))\to \pi_*\mathcal O_C(M(p, q))
\end{equation*} 
is surjective over some open neighborhood of 
$\pi(F)$ by Theorem \ref{c-thm6.1} (see also 
\ref{c-say2.14}). 
We note that $\pi_*\mathcal O_C(M(p, q))\ne 0$ by 
$P_F(p, q)\ne 0$ and that 
$C\cap \Nlc (X, \Delta)=\emptyset$ by assumption. 
Therefore, $\mathcal O_X(M(p, q))$ 
is $\pi$-generated 
at general points of $C$. 
By combining this fact 
with Claim \ref{c-thm10.1-claim1}, $\mathcal O_X(M(p, q))$ is 
$\pi$-generated at general points of every 
log canonical center of $(X, \Delta)$ if 
$(p, q)$ is sufficiently large with 
$0<aq-rp<1$. 
Hence we obtain Claim \ref{c-thm10.1-claim3}. 
\end{proof}

Note that $\mathcal O_X(M(p, q))$ is not $\pi$-generated 
for $(p, q)\in I$ 
because $M(p, q)$ is not $\pi$-nef over $W$. 
Therefore, $L_0\ne \emptyset$ with $\pi(L_0)\cap W\ne \emptyset$. 
We take a point $w\in \pi(L_0)\cap W$ and 
replace $Y$ with a relatively compact Stein open neighborhood of $w$. 
From now on, we will freely shrink $Y$ around $w$ suitably. 
Let $D_1, \ldots, D_{n+1}$ 
be general members of \begin{equation*}
\pi_*\mathcal O_X(M(p_0, q_0))
=H^0(X, \mathcal O_X(M(p_0, q_0)))
\end{equation*} 
with 
$(p_0, q_0)\in I$. 
We put $D:=\sum _{i=1}^{n+1}D_i$. 
Let $x\in X$ be any point of $L_0$. 
Then, by Lemma \ref{c-lem2.3}, 
$K_X+\Delta+D$ is not log canonical at $x$. 
On the other hand, we may assume that 
$K_X+\Delta+D$ 
is log canonical outside $L_0\cup \Nlc(X, \Delta)$ since 
$D_i$ is a general member of $|M(p_0, q_0)|$ for every $i$. 
We put 
\begin{equation*}
c:=\sup \left\{t\in \mathbb R\, \mid \text{$(X, \Delta+tD)$ is log canonical 
at $\pi^{-1}(w)\cap \left(X\setminus \Nlc(X, \Delta)\right)$}\right\}. 
\end{equation*} 
Then we can check that $0<c<1$ by Claim \ref{c-thm10.1-claim3}. 
We note that $w\in \pi(L_0)\cap W$. 
Thus, the pair $(X, \Delta+cD)$ has some 
log canonical centers contained in $L_0$ 
and intersecting $\pi^{-1}(w)$. 
By shrinking $Y$ around $w$, we may assume that 
$(X, \Delta+cD)$ is log canonical 
outside $\Nlc(X, \Delta)$. 
Let $C$ be a log canonical center 
contained in $L_0$ and intersecting $\pi^{-1}(w)$. 
We note that $\mathcal J_{\NLC}(X, \Delta+cD)=\mathcal 
J_{\NLC}(X, \Delta)$ by Lemma \ref{c-lem4.3} and its proof 
and that $C\cap \Nlc (X, \Delta+cD)=C\cap \Nlc (X, \Delta)=\emptyset$. 
We consider 
\begin{equation*}
K_X+\Delta+cD=c(n+1)p_0H+(1+c(n+1)q_0a)(K_X+\Delta). 
\end{equation*}  
Thus we have 
\begin{align*}
&pH+qa(K_X+\Delta)-(K_X+\Delta+cD)\\ &
=(p-c(n+1)p_0)H+(qa-(1+c(n+1)q_0a))(K_X+\Delta).
\end{align*} 
If $p$ and $q$ are large enough and $0<aq-rp\leq aq_0-rp_0$, 
then $$pH+qa(K_X+\Delta)-(K_X+\Delta+cD)$$ is $\pi$-ample 
over $w$. 
It is because 
\begin{align*}
&(p-c(n+1)p_0)H+(qa-(1+c(n+1)q_0a))(K_X+\Delta)\\
&=(p-(1+c(n+1))p_0)H+(qa-(1+c(n+1))q_0a)(K_X+\Delta)
\\ &\ +p_0H+(q_0a-1)(K_X+\Delta).  
\end{align*}
By shrinking $Y$ around $w$ suitably, 
we may further assume that 
it is $\pi$-ample over $Y$. 
We consider an analytically sufficiently general fiber $F$ of 
$C\to \pi(C)$ as in the proof of Claim \ref{c-thm10.1-claim3}. 
We note that $\left(H+r(K_X+\Delta)\right)|_{\pi^{-1}(\pi(F))}$ is 
nef by the choice of $F$. 

Suppose that $r$ is not rational. 
There exists an arbitrarily large $(p, q)\in \mathbb Z^2$ such that 
$0<aq-rp<\varepsilon =aq_0-rp_0$ and 
$\chi(F, \mathcal O_{F}(M(p, q)))\ne 0$ 
by Lemma \ref{c-lem10.2} because 
$P_{F}(p, q)=\chi (F, 
\mathcal O_{F}(M(p, q)))$ is 
a non-trivial 
polynomial of degree at most $\dim F\leq d$ by Lemma \ref{c-lem10.3}. 
Since 
\begin{equation*} 
\left(M(p, q)-(K_X+\Delta+cD)\right)|_{\pi^{-1}(\pi(F))}
\end{equation*} 
is ample by 
$0<aq-rp<aq_0-rp_0$, 
we have 
\begin{equation*}
\dim _{\mathbb C}H^0(F, \mathcal O_{F}(M(p, q)))=
\chi (F, \mathcal O_{F}(M(p, q)))\ne 0
\end{equation*} 
by the 
vanishing theorem:~Theorem \ref{c-thm6.1} 
(see also \ref{c-say2.14}). 
By the vanishing theorem:~Theorem \ref{c-thm6.1}, 
\begin{equation*}
\pi_*\mathcal O_X(M(p, q))\to \pi_*\mathcal O_C(M(p, q))
\end{equation*}
is surjective over a neighborhood 
of $\pi(F)$ because $\left(M(p, q)-(K_X+\Delta+cD)\right)|_{\pi^{-1}(\pi(F))}$ 
is ample.
We note that $C\cap \Nlc (X, \Delta+cD)=\emptyset$. 
Thus $C$ is not contained in $L(p, q)$. Therefore, 
$L(p, q)$ is a proper subset of $L(p_0, q_0)=L_0$, 
which gives the desired contradiction. 
Hence we know that $r$ is rational. 

We next suppose that the assertion of the theorem concerning 
the denominator of $r$ is false. 
We choose $(p_0, q_0)\in I$ such that 
$aq_0-rp_0$ is the maximum, 
say it is equal to $e/v$. 
If $0<aq-rp\leq e/v$ and 
$(p, q)$ is sufficiently large, then 
\begin{equation*}
\chi (F, 
\mathcal O_{F}(M(p, q)))=\dim _{\mathbb C} H^0(F, 
\mathcal O_{F}(M(p, q)))
\end{equation*} 
since 
$\left(M(p, q)-(K_X+\Delta+cD)\right)|_{\pi^{-1}(\pi(F))}$ is ample. 
There exists sufficiently 
large $(p, q)\in \mathbb Z^2$ in the strip 
$0<aq-rp<1$ with $\varepsilon =1$ for which 
\begin{equation*}
\dim_{\mathbb C} H^0(F, \mathcal O_{F}(M(p, q)))
=\chi (F, \mathcal O_{F}(M(p, q)))\ne 0
\end{equation*} 
by 
Lemma \ref{c-lem10.2} since 
$P_F(p, q)=\chi (F, \mathcal O_{F}(M(p, q)))$ is a 
non-trivial 
polynomial of degree at most $\dim F\leq d$ by Lemma \ref{c-lem10.3}. 
Note that $aq-rp\leq e/v=aq_0-rp_0$ holds automatically for 
$(p, q)\in I$. 
Since 
\begin{equation*}
\pi_*\mathcal O_X(M(p, q))
\to \pi_*\mathcal O_C(M(p,q))
\end{equation*}  
is surjective over some open neighborhood of 
$\pi(F)$ by the ampleness of 
\begin{equation*}
\left(M(p,q)-
(K_X+\Delta+cD)\right)|_{\pi^{-1}(\pi(F))}, 
\end{equation*}  
we 
obtain the desired contradiction by the same reason as above. 

Thus, we finish the proof of the rationality theorem. 
\end{proof}

We close this section with an easy remark. 

\begin{rem}\label{c-rem10.4}
The proof of Theorem \ref{c-thm10.1} shows that 
Theorem \ref{c-thm10.1} holds true under the assumption that 
$H$ is a $\pi$-ample line bundle. 
\end{rem}

\section{Kleiman--Mori cones}\label{c-sec11} 

In this section, we will define Kleiman--Mori cones for 
projective morphisms between complex analytic spaces 
under some suitable assumption. 

\begin{say}\label{c-say11.1}
Throughout this section, 
let $\pi\colon X\to Y$ be a projective 
morphism of complex analytic spaces and let 
$W$ be a compact subset of $Y$. 
Let $Z_1(X/Y; W)$ be the free abelian group 
generated by the projective integral curves $C$ on $X$ such that 
$\pi(C)$ is a point of $W$. 
Let $U$ be any open neighborhood of $W$. 
Then we can consider the following intersection pairing 
\begin{equation*} 
\cdot :
\Pic\!\left(\pi^{-1}(U)\right)\times Z_1(X/Y; W)\to \mathbb Z
\end{equation*}  
given by $\mathcal L\cdot C\in \mathbb Z$ for 
$\mathcal L\in \Pic\!\left(\pi^{-1}(U)\right)$ and 
$C\in Z_1(X/Y; W)$. 
We say that $\mathcal L$ is {\em{$\pi$-numerically 
trivial over $W$}} when $\mathcal L\cdot C=0$ for 
every $C\in Z_1(X/Y; W)$. 
We take $\mathcal L_1, \mathcal L_2\in 
\Pic\!\left(\pi^{-1}(U)\right)$. 
If $\mathcal L_1\otimes \mathcal L_2^{-1}$ 
is $\pi$-numerically trivial over $W$, 
then we write $\mathcal L_1\equiv_W\mathcal L_2$ 
and say that $\mathcal L_1$ is numerically equivalent to 
$\mathcal L_2$ over $W$. 
We put 
\begin{equation*}
\widetilde A(U, W):=\Pic\!\left(\pi^{-1}(U)\right)/{\equiv_W}
\end{equation*}  
and define 
\begin{equation*}
A^1(X/Y; W):=\underset{W\subset U}\varinjlim
\widetilde A(U, W), 
\end{equation*}  
where $U$ runs through all the open neighborhoods of 
$W$. 
\end{say}

\begin{say}\label{c-say11.2}
We assume that $A^1(X/Y; W)$ is a finitely generated 
abelian group. Then we can define the {\em{relative Picard number}} 
$\rho(X/Y; W)$ to be the rank of 
$A^1(X/Y; W)$. 
We put 
\begin{equation*}
N^1(X/Y; W):=A^1(X/Y; W)\otimes _{\mathbb Z} \mathbb R. 
\end{equation*} 
Let $A_1(X/Y; W)$ be the image of 
\begin{equation*} 
Z_1(X/Y; W)\to \Hom_{\mathbb Z} \left(A^1(X/Y; W), 
\mathbb Z\right)
\end{equation*} 
given by the above intersection pairing. 
Then we set 
\begin{equation*} 
N_1(X/Y; W):=A_1(X/Y; W)\otimes _{\mathbb Z}\mathbb R. 
\end{equation*}  
As usual, we can define the {\em{Kleiman--Mori cone}} 
\begin{equation*} 
\NE(X/Y; W)
\end{equation*} 
of $\pi\colon X\to Y$ over $W$, that is, 
$\NE(X/Y; W)$ is the closure of the convex cone in 
$N_1(X/Y; W)$ spanned by the projective 
integral curves $C$ on $X$ such that 
$\pi(C)$ is a point of $W$. 
An element $\zeta\in N^1(X/Y; W)$ is called 
{\em{$\pi$-nef over $W$}} or {\em{nef over $W$}} 
if $\zeta\geq 0$ on $\NE(X/Y; W)$, equivalently, 
$\zeta|_{\pi^{-1}(w)}$ is nef in the usual sense for 
every $w\in W$. 
\end{say}

\begin{rem}\label{c-rem11.3}
We assume that $\pi\colon X\to Y$ decomposes as 
\begin{equation*}
\xymatrix{
\pi\colon X\ar[r]^-\varphi & Z \ar[r]^-{\pi_Z}&Y, 
}
\end{equation*}
where $\pi_Z\colon Z\to Y$ is a projective morphism 
of complex analytic spaces. 
Then $\varphi$ is always projective and $\pi^{-1}_Z(W)$ is 
a compact subset of $Z$. 
Therefore, we can define $A^1\left(X/Z; \pi^{-1}_Z(W)\right)$ and 
$N^1\left(X/Z; \pi^{-1}_Z(W)\right)$ as above. 
By definition, $N^1\left(X/Z; \pi^{-1}_Z(W)\right)$ is a quotient 
vector space of $N^1\left(X/Y; W\right)$. 
Hence, if $\dim _{\mathbb R} N^1(X/Y; W)<\infty$, 
then we see that 
$\dim _{\mathbb R}N^1\left(X/Z; \pi^{-1}_Z(W)\right)<\infty$ holds. 
\end{rem}

\begin{lem}[{\cite[Proposition 4.7 (2)]{nakayama1}}]\label{c-lem11.4}
$\NE(X/Y; W)$ contains no lines of $N_1(X/Y; W)$. 
\end{lem}
\begin{proof}
Suppose that $\NE(X/Y; W)$ contains a line of 
$N_1(X/Y; W)$. 
Then we can take $\Gamma \in \NE(X/Y; W)$ such that 
$\Gamma, -\Gamma \in \NE(X/Y; W)$. 
We take a $\pi$-ample $\mathbb R$-line bundle $\mathcal A$ on $X$. 
By definition, $\mathcal A$ is $\pi$-nef over $W$. 
Therefore, we obtain $\mathcal A\cdot \Gamma \geq 0$ and 
$-\mathcal A\cdot \Gamma \geq 0$. 
This means that $\mathcal A\cdot \Gamma =0$. 
On the other hand, after shrinking $Y$ around 
$W$ suitably, we can take a line bundle $\mathcal M$ 
on $X$ such that $\Gamma \cdot \mathcal M>0$ since $\Gamma 
\ne 0$ in $N_1(X/Y; W)$. 
Since $\mathcal A$ is $\pi$-ample, 
$m\mathcal A-\mathcal M$ is also 
$\pi$-ample over some open neighborhood of $W$, where  
$m$ is a large positive integer. 
This implies that $(m\mathcal A-\mathcal M)\cdot \Gamma \geq 0$ 
since $\Gamma \in \NE(X/Y; W)$. 
Thus, $m\mathcal A\cdot \Gamma \geq \mathcal M\cdot \Gamma >0$ holds. 
Hence we obtain $\mathcal A\cdot \Gamma >0$. 
This is a contradiction. 
Therefore, there are no lines in $\NE(X/Y; W)$. 
\end{proof}

The following theorem is Kleiman's ampleness criterion for 
projective morphisms between complex analytic spaces 
(see \cite[Proposition 4.7]{nakayama1}). 

\begin{thm}[Kleiman's ampleness criterion]\label{c-lem11.5}
Let $\pi\colon X\to Y$ be a projective morphism 
between complex analytic spaces and let $W$ be a compact 
subset of $Y$ such that the dimension of $N^1(X/Y; W)$ is finite. 
Let $\mathcal L$ be an $\mathbb R$-line bundle 
on $X$. 
Then the following conditions are equivalent. 
\begin{itemize}
\item[(i)] $\mathcal L$ is $\pi$-ample 
over $W$. 
\item[(ii)] $\mathcal L$ is $\pi$-ample over some open neighborhood 
$U$ of $W$. 
\item[(iii)] $\mathcal L$ is positive on $\NE(X/Y; W)\setminus \{0\}$. 
\end{itemize} 
\end{thm}

\begin{proof}
We have already proved the equivalence of (i) and (ii) in 
Lemma \ref{c-lem3.2} without assuming $\dim _{\mathbb R} 
N^1(X/Y; W)<\infty$. 

\setcounter{step}{0}
\begin{step}\label{c-lem11.5-step1}
In this step, we will prove that (i) follows from (iii). 

We assume that $\mathcal L$ is positive on $\NE(X/Y; W)\setminus \{0\}$. 
Then we can take a $\pi$-ample 
$\mathbb R$-line bundle $\mathcal A$ on $X$ such that 
$\mathcal N:=\mathcal L-\mathcal A$ is non-negative on $\NE(X/Y; W)$. 
This means that $\mathcal N|_{\pi^{-1}(w)}$ is nef for every $w\in W$. 
Since $\mathcal L|_{\pi^{-1}(w)}=\mathcal N|_{\pi^{-1}(w)}
+\mathcal A|_{\pi^{-1}(w)}$, $\mathcal N|_{\pi^{-1}(w)}$ is nef, 
and $\mathcal A|_{\pi^{-1}(w)}$ is ample, 
$\mathcal L|_{\pi^{-1}(w)}$ is ample by the usual Kleiman's ampleness 
criterion. 
Hence (i) follows from (iii). 
\end{step}
\begin{step}\label{c-lem11.5-step2}
In this step, we will prove that (iii) follows from (ii). 

We assume that $\mathcal L$ is $\pi$-ample over some open neighborhood 
$U$ 
of $W$. By replacing $Y$ with $U$, we may assume that $Y=U$. 
In the proof of Lemma \ref{c-lem11.4}, we have already checked that 
$\mathcal L\cdot \Gamma >0$ for every $\Gamma \in 
\NE(X/Y; W)\setminus \{0\}$. 
This means that (iii) follows from (ii). 
\end{step}
We finish the proof. 
\end{proof}

From now on, we always assume that the dimension of 
$N^1(X/Y; W)$ is finite. 
In order to formulate the cone and contraction theorem, 
we need the following definitions. 

\begin{defn}\label{c-def11.6}
Let $\pi\colon X\to Y$ be a projective morphism of complex 
analytic spaces and let $W$ be a compact subset of $Y$. 
Let $X$ be a normal complex variety and let $\Delta$ be an effective 
$\mathbb R$-divisor on $X$ such 
that $K_X+\Delta$ is $\mathbb R$-Cartier. 
We assume that the dimension of $N^1(X/Y; W)$ is finite. 
Then we define a subcone 
\begin{equation*}
\NE(X/Y; W)_{\Nlc(X, \Delta)}
\end{equation*} 
of $\NE(X/Y; W)$ 
as the closure of the convex cone spanned by 
the projective integral curves $C$ on $\Nlc(X, \Delta)$ 
such that $\pi(C)$ is a point of $W$. 
Let $D$ be an element of $N^1(X/Y; W)$. 
We define 
$$
D_{\geq 0}:=\{z\in N_1(X/Y; W) \,|\, D\cdot z\geq 0\}. 
$$ 
Similarly, we can define $D_{>0}$, 
$D_{\leq 0}$, and 
$D_{<0}$. 
We also define 
$$
D^{\perp}:=\{ z\in N_1(X/Y; W) \, | \, D\cdot z=0\}. 
$$ 
We use the following notation 
$$
\NE(X/Y; W)_{D\geq 0} :=\NE(X/Y; W)\cap 
D_{\geq 0}, 
$$ 
and similarly for $>0$, $\leq 0$, 
and $<0$. 
\end{defn}

\begin{defn}\label{c-def11.7}
An {\em{extremal face}} of the 
Kleiman--Mori cone $\NE(X/Y; W)$ is a non-zero 
subcone $F\subset \NE(X/Y; W)$ such that 
$z, z'\in \NE(X/Y; W)$ and $z+z'\in F$ imply that 
$z, z'\in F$. Equivalently, 
$F=\NE(X/Y; W)\cap \mathcal H^{\perp}$ for some 
$\mathbb R$-line bundle $\mathcal H$ which is defined on 
some open neighborhood of $\pi^{-1}(W)$ and 
is $\pi$-nef over $W$. 
We call $\mathcal H$ 
a {\em{support function}} of $F$. 
An {\em{extremal ray}} 
is a one-dimensional 
extremal face. 
\begin{itemize}
\item[(1)] An extremal face $F$ is called 
{\em{$(K_X+\Delta)$-negative}} 
if 
\begin{equation*}
F\cap \NE(X/Y; W)_{(K_X+\Delta)\geq 0} 
=\{ 0\}. 
\end{equation*}
\item[(2)] An extremal face $F$ is called {\em{rational}} 
if we can choose 
a $\mathbb Q$-line bundle $\mathcal H$, which is 
defined on some open neighborhood 
of $\pi^{-1}(W)$ and is $\pi$-nef over $W$,  
as a support 
function of $F$. 
\item[(3)] An extremal face $F$ is 
called {\em{relatively ample at 
$\Nlc (X, \Delta)$}} if 
\begin{equation*}
F\cap \NE(X/Y; W)_{\Nlc (X, \Delta)}=\{0\}. 
\end{equation*}
Equivalently, $\mathcal H|_{\Nlc (X, \Delta)}$ is 
$\pi|_{\Nlc (X, \Delta)}$-ample over $W$ for 
every support function $\mathcal H$ of $F$. 
\item[(4)] An extremal face $F$ is called {\em{contractible 
at $\Nlc (X, \Delta)$}} if it has a rational support function 
$\mathcal H$ such that 
$\mathcal H|_{\Nlc (X, \Delta)}$ is $\pi|_{\Nlc (X, \Delta)}$-semiample 
over some open neighborhood of $W$. 
\end{itemize}
\end{defn}

We make a remark on (3) in Definition \ref{c-def11.7}. 

\begin{rem}\label{c-rem11.8}
In (3) in Definition \ref{c-def11.7}, 
the condition that $F$ is relatively ample at $\Nlc(X, \Delta)$ implies 
that the support function $\mathcal H$ of $F$ is 
positive on $\NE(X/Y; W)_{\Nlc(X, \Delta)}\setminus \{0\}$. 
Let $\mathcal A$ be a $\pi$-ample 
$\mathbb R$-line bundle on $X$. 
Then $\mathcal N:=\mathcal H-\varepsilon \mathcal A$ is 
positive on $\NE(X/Y; W)_{\Nlc(X, \Delta)}\setminus \{0\}$ for 
some $0<\varepsilon \ll 1$. 
Thus, it is easy to see that $\mathcal N|_{\Nlc(X, \Delta)}$ is 
$\pi|_{\Nlc(X, \Delta)}$-nef over $W$. 
Note that $\mathcal A|_{\Nlc(X, \Delta)}$ is obviously 
$\pi|_{\Nlc(X, \Delta)}$-ample. 
Hence $\mathcal H|_{\Nlc(X, \Delta)}$ is $\pi|_{\Nlc(X, \Delta)}$-ample 
over $W$. 
\end{rem}

\subsection{Nakayama's finiteness}\label{c-subsec11.1} 
In this subsection, we quickly recall Nakayama's finiteness. 
As we saw above, 
for the cone and contraction 
theorem in this paper, we need the assumption that 
the dimension of $N^1(X/Y; W)$ is finite. 
In general, the dimension of $N^1(X/Y; W)$ may be infinite. 
The author learned the following example from 
Noboru Nakayama. 

\begin{ex}[Nakayama]\label{c-ex11.9}
Let $\pi\colon X\to Y$ be a projective surjective morphism 
of complex analytic spaces 
such that $X$ is a normal 
complex variety with $\dim X\geq 2$ and 
that $Y=\{z\in \mathbb C\, |\, |z|<2\}$. 
We put 
\begin{equation*}
W:=\left\{\left. \frac{1}{n}\, \right|\, n \in 
\mathbb Z_{>0}\right\}\cup \{0\}. 
\end{equation*}
Then $W$ is a compact subset of $Y$. 
It is obvious that $W$ has infinitely many connected components. 
In this case, we can see that the abelian group 
$A^1(X/Y; W)$ is not finitely generated. 
Hence we have $\dim _{\mathbb R} N^1(X/Y; W)=\infty$. 
\end{ex}

The following theorem gives an important and useful 
sufficient condition for the finite-dimensionality of 
$N^1(X/Y; W)$. 
We state it here for the reader's convenience. 

\begin{thm}
[{Nakayama's finiteness, 
see \cite[Chapter II.~5.19.~Lemma]{nakayama2}}]
\label{c-thm11.10}
Let $\pi\colon X\to Y$ be a projective surjective 
morphism of complex analytic spaces such that 
$W$ is a compact subset of $Y$. 
We assume that $W\cap V$ has only finitely 
many connected components 
for every analytic subset $V$ defined over an 
open neighborhood of $W$. 
Then $A^1(X/Y; W)$ is a finitely generated 
abelian group. 
\end{thm}

\begin{proof} 
For the details, see \cite[4.1.~Nakayama's finiteness]{fujino-minimal}. 
\end{proof}

In this paper, we do not need Theorem \ref{c-thm11.10} 
except in the proof of Corollary \ref{c-cor1.5}. 
We only need the assumption that the dimension of $N^1(X/Y; W)$ is 
finite. 

\begin{rem}\label{c-rem11.11}
Let $\pi\colon X\to Y$ be a smooth projective 
surjective morphism between smooth irreducible complex 
analytic spaces. 
Let $W$ be a compact subset of $Y$. 
Assume that $W$ is connected. 
Then we can easily check that the dimension of 
$N^1(X/Y; W)$ is finite. 
However, $W\cap V$ may have infinitely many connected components 
for some analytic subset $V$ defined over an 
open neighborhood of $W$.  
\end{rem}

We close this subsection with some remarks on Nakayama's 
fundamental paper \cite{nakayama1}. 

\begin{rem}\label{c-rem11.12}
Example \ref{c-ex11.9} shows that \cite[Proposition 4.3]{nakayama1} 
is not correct. In \cite[Section 4]{fujino-kawamata}, we gave an 
alternative simple proof of \cite[Theorem 5.5]{nakayama1} 
(see also \cite[5.3]{fujino-kawamata}). 
\end{rem}

\section{Cone theorem}\label{c-sec12}
In this section, we will explain the cone and contraction 
theorem of 
normal pairs for projective morphisms between 
complex analytic spaces. 
The proof given in this section is essentially the same as 
that in \cite{fujino-fundamental} for algebraic varieties. 
The main ingredients of this section is the 
basepoint-free theorem (see Theorem \ref{c-thm9.1}) and 
the rationality theorem (see Theorem \ref{c-thm10.1}). 

We first treat the contraction theorem, which is 
a direct consequence of 
the basepoint-free theorem:~Theorem \ref{c-thm9.1}. 
We will use it in the proof of the cone 
theorem:~Theorem \ref{c-thm12.2}. 

\begin{thm}[Contraction theorem]\label{c-thm12.1}
Let $\pi\colon X\to Y$ be a projective morphism of complex 
analytic spaces such that $X$ is a normal complex 
variety and let $W$ be a compact subset of $Y$ 
such that the dimension of $N^1(X/Y; W)$ is 
finite. 
Let $\Delta$ be an effective $\mathbb R$-divisor 
on $X$ such that $K_X+\Delta$ is $\mathbb R$-Cartier. 
Let $\mathcal H$ be a 
line bundle which is defined on some open neighborhood of 
$\pi^{-1}(W)$ and is 
$\pi$-nef over $W$ such that 
the extremal face 
\begin{equation*}
F=\mathcal H^{\perp}\cap \NE(X/Y; W)
\end{equation*} 
is $(K_X+\Delta)$-negative 
and contractible at $\Nlc (X, \Delta)$. 
Then, after shrinking $Y$ around $W$ suitably, 
there exists a projective morphism 
$\varphi_F\colon X\to Z$ over $Y$ with the following properties. 
\begin{itemize}
\item[$(1)$] 
Let $C$ be a projective integral curve on $X$ such that 
$\pi(C)$ is a point of $W$. 
Then $\varphi_F(C)$ is a point if and 
only if the numerical equivalence class 
$[C]$ of $C$ is in $F$. 
\item[$(2)$] The natural map 
$\mathcal O_Z\to (\varphi_F)_*\mathcal O_X$ is an isomorphism.  
\item[$(3)$] Let $\mathcal L$ be a line bundle on $X$ such 
that $\mathcal L\cdot C=0$ for 
every curve $C$ with $[C]\in F$. 
Assume that $\mathcal L^{\otimes m}|_{\Nlc(X, \Delta)}$ 
is $\varphi_{F}|_{\Nlc (X, \Delta)}$-generated for every $m \gg 0$. 
Then, after shrinking $Y$ around $W$ suitably again, 
there exists a line bundle $\mathcal L_Z$ on $Z$ such 
that $\mathcal L\simeq \varphi^*_F\mathcal L_Z$ holds. 
\end{itemize} 
\end{thm}

As we mentioned above, Theorem \ref{c-thm12.1} 
easily follows from the basepoint-free theorem:~Theorem \ref{c-thm9.1}. 

\begin{proof}[Proof of Theorem \ref{c-thm12.1}]
Since $F$ is contractible at $\Nlc(X, \Delta)$ 
by assumption, 
we may assume that $\mathcal H|_{\Nlc(X, \Delta)}$ 
is $\pi|_{\Nlc(X, \Delta)}$-semiample over some 
open neighborhood of $W$. 
Since $F$ is $(K_X+\Delta)$-negative by assumption, 
we can take some positive integer $a$ such that 
$a\mathcal H-(K_X+\Delta)$ is $\pi$-ample 
over $W$. 
By the basepoint-free theorem (see 
Theorem \ref{c-thm9.1}), after shrinking $Y$ around 
$W$ suitably, $\mathcal H^{\otimes m}$ is $\pi$-generated 
for some positive integer $m$. 
We take the Stein factorization of the associated morphism. 
Then we can obtain a 
contraction morphism 
$\varphi_F\colon X\to Z$ over $Y$ satisfying the properties (1) and (2). 
We consider $\varphi_F\colon X\to Z$ and $\NE(X/Z; \pi_Z^{-1}(W))$, 
where $\pi_Z\colon Z\to Y$ is the structure morphism. 
Then $\NE(X/Z; \pi_{Z}^{-1}(W))=F$ holds 
by construction, $\mathcal L$ is numerically 
trivial over $\pi_Z^{-1}(W)$, and $-(K_X+\Delta)$ is $\varphi_F$-ample 
over $\pi_{Z}^{-1}(W)$. 
We use the basepoint-free theorem over $Z$ 
(see Theorem \ref{c-thm9.1}). 
Then, after shrinking $Z$ around 
$\pi^{-1}_Z(W)$ suitably, 
both $\mathcal L^{\otimes m}$ and $\mathcal L^{\otimes (m+1)}$ 
are pull-backs of line bundles on $Z$. 
Their difference gives a line bundle $\mathcal L_Z$ 
on $Z$ such that 
$\mathcal L\simeq \varphi^*_F\mathcal L_Z$ holds. 
We finish the proof of Theorem \ref{c-thm12.1}.  
\end{proof}

The following theorem is the main result of this section, which 
is the cone theorem of normal pairs 
for projective morphisms between complex 
analytic spaces. 

\begin{thm}[{Cone theorem, see 
\cite[Theorem 16.6]{fujino-fundamental}}]\label{c-thm12.2}
Let $\pi\colon X\to Y$ be a projective morphism of complex 
analytic spaces such that $X$ is a normal complex variety 
and let $W$ be a compact subset of $Y$ 
such that the dimension of $N^1(X/Y; W)$ is finite. 
Let $\Delta$ be an effective $\mathbb R$-divisor 
on $X$ such that $K_X+\Delta$ is $\mathbb R$-Cartier. 
Then we have the following 
properties. 
\begin{itemize}
\item[$(1)$] We can write 
\begin{equation*}
\NE(X/Y; W)=\NE(X/Y; W)_{(K_X+\Delta)\geq 0} 
+\NE(X/Y; W)_{\Nlc (X, \Delta)}+\sum R_j, 
\end{equation*} 
where $R_j$'s are the $(K_X+\Delta)$-negative 
extremal rays of $\NE(X/Y; W)$ that are 
rational and relatively ample at $\Nlc (X, \Delta)$. 
In particular, each $R_j$ is spanned by 
an integral curve $C_j$ on $X$ such that 
$\pi(C_j)$ is a point of $W$.  
\item[$(2)$] Let $\mathcal A$ be a $\pi$-ample $\mathbb R$-line 
bundle defined on some open neighborhood of $\pi^{-1}(W)$. 
Then there are only finitely many $R_j$'s included in 
$\NE(X/Y; W)_{(K_X+\Delta+\mathcal A)<0}$. In particular, 
the $R_j$'s are discrete in the half-space 
$\NE(X/Y; W)_{(K_X+\Delta)<0}$. 
\item[$(3)$] Let $F$ be a $(K_X+\Delta)$-negative extremal 
face of $\NE(X/Y; W)$ that is 
relatively ample at $\Nlc (X, \Delta)$. 
Then $F$ is a rational face. 
In particular, $F$ is contractible at $\Nlc (X, \Delta)$. 
\end{itemize}
\end{thm}

By combining Theorem \ref{c-thm12.2} with Theorem \ref{c-thm12.1}, 
we obtain the cone and contraction theorem 
of normal pairs for projective morphisms between complex 
analytic spaces. 

\begin{proof}[Proof of Theorem \ref{c-thm12.2}]
Without loss of generality, we can freely shrink $Y$ around 
$W$ suitably throughout this proof. 
From Step \ref{c-thm12.2-step1} to Step \ref{c-thm12.2-step5}, 
we will prove Theorem \ref{c-thm12.2} under 
the extra assumption that $K_X+\Delta$ is $\mathbb Q$-Cartier. 
Then, in Step \ref{c-thm12.2-step6}, we will treat the general case. 
We note that we may assume that $\dim _{\mathbb R}
N_1(X/Y; W)\geq 2$ and $K_X+\Delta$ is not $\pi$-nef over $W$. 
Otherwise, the theorem is obvious. 

\setcounter{step}{0}
\begin{step}\label{c-thm12.2-step1} 
In this step, we will prove: 
\setcounter{cla}{0}
\begin{cla}\label{c-thm12.2-claim1} 
When $K_X+\Delta$ is $\mathbb Q$-Cartier, 
the following equality 
\begin{equation*}
\NE(X/Y; W)=\overline 
{\NE(X/Y; W)_{(K_X+\Delta)\geq 0} 
+\NE(X/Y; W)_{\Nlc (X, \Delta)}+\sum _F F} 
\end{equation*} 
holds, 
where $F$'s vary among all rational 
proper $(K_X+\Delta)$-negative extremal faces that are relatively 
ample at $\Nlc (X, \Delta)$. 
\end{cla}
We note that in Claim \ref{c-thm12.2-claim1} 
$\raise0.5ex\hbox{\textbf{-----}}$ 
denotes the closure 
with respect to the real topology. 
\begin{proof}[Proof of Claim \ref{c-thm12.2-claim1}]We put 
\begin{equation*}
\mathfrak B=\overline {\NE(X/Y; W)_{(K_X+\Delta)\geq 0} 
+\NE(X/Y; W)_{\Nlc (X, \Delta)}+\sum _F F}. 
\end{equation*}
The inclusion $\NE(X/Y; W)\supset \mathfrak B$ obviously holds 
by definition. 
We note that each $F$ is spanned by curves 
on $X$ mapped to points in $W$ by 
Theorem \ref{c-thm12.1} (1). 
From now on, we suppose 
$\NE(X/Y; W)\ne \mathfrak B$. 
Then we will derive a contradiction. 
We can take a separating function $M$ which 
is a line bundle 
on some open neighborhood of $\pi^{-1}(W)$ 
and is not a multiple of $K_X+\Delta$ in $N^1(X/Y; W)$ such 
that $M>0$ on $\mathfrak B\setminus \{0\}$ and 
$M\cdot z_0<0$ for some $z_0\in \NE(X/Y; W)$. 
Let $\mathcal C$ be the dual 
cone of $\NE(X/Y; W)_{(K_X+\Delta)\geq 0}$, 
that is, 
\begin{equation*}
\mathcal C=\{D\in N^1(X/Y; W)\, | \, D\cdot z\geq 0 \ \text{for}\ z\in 
\NE(X/Y; W)_{(K_X+\Delta) \geq 0}\}. 
\end{equation*}
Then $\mathcal C$ is generated by $K_X+\Delta$ and 
$\mathbb R$-line bundles on $X$ 
which are $\pi$-nef over $W$. 
Since $M>0$ on $\NE(X/Y; W)_{(K_X+\Delta)\geq 0}\setminus \{0\}$, 
$M$ is in the interior of $\mathcal C$. 
Hence there 
exists a $\pi$-ample $\mathbb R$-line bundle $A$ such that 
\begin{equation*}
M-A=L'+p(K_X+\Delta)
\end{equation*} 
in $N^1(X/Y; W)$, where 
$L'$ is an $\mathbb R$-line bundle on 
some open neighborhood of $\pi^{-1}(W)$ which is $\pi$-nef 
over $W$, and $p$ is a non-negative rational number. 
Therefore, $M$ is expressed in the form 
\begin{equation*}
M=H+p(K_X+\Delta)
\end{equation*} in 
$N^1(X/Y; W)$, where $H=A+L'$ is a 
$\mathbb Q$-line bundle on $X$ which is $\pi$-ample 
over $W$. 
The rationality theorem (see Theorem \ref{c-thm10.1}) 
implies that there exists a positive 
rational number $r<p$ such that 
\begin{equation*}
L=H+r(K_X+\Delta)
\end{equation*} 
is 
$\pi$-nef over $W$ but not 
$\pi$-ample over $W$, and $L|_{\Nlc (X, \Delta)}$ is 
$\pi|_{\Nlc (X, \Delta)}$-ample over $W$. 
We note that $L\ne 0$ in $N^1(X/Y; W)$ since 
$M$ is not a multiple of $K_X+\Delta$. 
Thus the extremal face $F_L$ associated to 
the support function $L$ is contained 
in $\mathfrak B$, which implies $M>0$ on $F_L$. 
Therefore, $p<r$. It is a contradiction. 
This completes the proof of Claim \ref{c-thm12.2-claim1}. 
\end{proof}
\end{step} 

\begin{step}\label{c-thm12.2-step2}
In this step, we will prove: 
\begin{cla}\label{c-thm12.2-claim2}
In the equality in Claim \ref{c-thm12.2-claim1}, 
we can assume that every extremal face $F$ is one-dimensional. 
\end{cla}
\begin{proof}[Proof of Claim \ref{c-thm12.2-claim2}]
Let $F$ be a rational proper $(K_X+\Delta)$-negative extremal face 
that is relatively ample at $\Nlc (X, \Delta)$. 
We assume that $\dim F\geq 
2$. After shrinking $Y$ around $W$ suitably, 
we can take the contraction morphism 
$\varphi_F:X\to Z$ over $Y$ associated to $F$ 
(see Theorem \ref{c-thm12.1}). 
We note that $F=\NE(X/Z; \pi^{-1}_Z(W))$, 
where $\pi_Z\colon Z\to Y$ is the structure morphism, 
and that $-(K_X+\Delta)$ is $\varphi_F$-ample 
over $\pi^{-1}_Z(W)$ by construction. 
By Claim \ref{c-thm12.2-claim1} in Step \ref{c-thm12.2-step1}, 
we obtain 
\begin{equation}\label{c-thm12.2-eq1}
F=\NE(X/Z; \pi^{-1}_Z(W))=\overline {\sum _G G}, 
\end{equation} 
where the $G$'s in \eqref{c-thm12.2-eq1} are the rational proper 
$(K_X+\Delta)$-negative extremal faces of $\NE(X/Z; \pi^{-1}_Z(W))$. 
We note that $\NE(X/Z; \pi^{-1}_Z(W))_{\Nlc (X, \Delta)}=0$ holds 
because 
$\varphi_F$ embeds $\Nlc (X, \Delta)$ into $Z$. 
The $G$'s are also 
$(K_X+\Delta)$-negative 
extremal faces of $\NE(X/Y; W)$ that are 
ample at $\Nlc (X, \Delta)$ 
with $\dim G<\dim F$. By induction, we finally obtain 
\begin{equation}\label{c-thm12.2-eq2} 
\NE(X/Y; W)=\overline {\NE(X/Y; W)_{(K_X+\Delta)\geq 0} 
+\NE(X/Y; W)_{\Nlc (X, \Delta)} +\sum R_j},
\end{equation}
where the $R_j$'s are $(K_X+\Delta)$-negative rational 
extremal rays. 
Note that each $R_j$ does not 
intersect 
$\NE(X/Y; W)_{\Nlc (X, \Delta)}$. 
We finish the proof of Claim \ref{c-thm12.2-claim2}. 
\end{proof}
\end{step}
\begin{step}\label{c-thm12.2-step3} 
In this step, we still assume that $K_X+\Delta$ is 
$\mathbb Q$-Cartier. 
We will finish the proof of (1) when $K_X+\Delta$ is $\mathbb Q$-Cartier. 

The contraction theorem (see Theorem \ref{c-thm12.1}) 
guarantees that 
for each extremal ray $R_j$, 
which is $(K_X+\Delta)$-negative, rational, and 
relatively ample at $\Nlc(X, \Delta)$, there exists a 
projective integral curve $C_j$ on $X$ such that 
$[C_j]\in R_j$. 
Let $\psi_j:X\to Z_j$ be the contraction 
morphism of $R_j$ over $Y$ after shrinking $Y$ 
around $W$ suitably, and 
let $A$ be a $\pi$-ample 
line bundle on $X$. 
Let $\pi_{Z_j}\colon Z_j\to Y$ be the structure morphism. 
We set 
\begin{equation*}
r_j=-\frac{A\cdot  C_j}{(K_X+\Delta)\cdot C_j}. 
\end{equation*} 
Then $A+r_j(K_X+\Delta)$ is $\psi_j$-nef 
over $\pi^{-1}_{Z_j}(W)$ but not $\psi_j$-ample 
over $\pi^{-1}_{Z_j}(W)$, 
and 
\begin{equation*}
(A+r_j (K_X+\Delta))|_{\Nlc (X, \Delta)}
\end{equation*} 
is $\psi_j|_{\Nlc (X, \Delta)}$-ample over $\pi^{-1}_{Z_j}(W)$. 
By the rationality theorem (see Theorem \ref{c-thm10.1}), 
expressing 
$r_j =u_j /v_j$ with 
$u_j, v_j\in \mathbb Z_{>0}$ and 
$(u_j, v_j)=1$, we have 
the inequality $v_j\leq a(\dim X+1)$. 
After shrinking $Y$ around $W$ suitably, we take $\pi$-ample line bundles 
$H_1, H_2, \ldots, H_{\rho-1}$ on $X$ such that 
$K_X+\Delta$ and the $H_i$'s form a basis of $N^1(X/Y; W)$, where 
$\rho =\dim _{\mathbb R}N^1(X/Y; W)<\infty$. 
As we saw above, the intersection of  the extremal 
rays $R_j$ with the hyperplane 
\begin{equation*}
\{z\in N_1(X/Y; W)\, | \, a(K_X+\Delta)\cdot z=-1\}
\end{equation*} in 
$N_1(X/Y; W)$ lie on the lattice 
\begin{equation*}
\Lambda=
\{z\in N_1(X/Y; W)\, | \, a(K_X+\Delta)
\cdot z=-1, H_i \cdot z\in (a(a(\dim 
X+1))!)^{-1}\mathbb Z\}. 
\end{equation*}
This implies that the extremal rays are discrete in the 
half space 
\begin{equation*}
\{z \in N_1(X/Y; W) \, | \, (K_X+\Delta)\cdot z<0\}. 
\end{equation*} 
Thus we can omit the closure sign 
$\raise0.5ex\hbox{\textbf{-----}}$ 
from 
the formula (\ref{c-thm12.2-eq2}) 
and this completes the proof of (1) when $K_X+\Delta$ is 
$\mathbb Q$-Cartier. 
\end{step} 

\begin{step}\label{c-thm12.2-step4}
In this step, we will prove (2) under the assumption that 
$K_X+\Delta$ is $\mathbb Q$-Cartier. 

Let $\mathcal A$ be a $\pi$-ample $\mathbb R$-line bundle on $X$. 
We choose $0<\varepsilon_i\ll 1$ for $1\leq i\leq \rho-1$ such 
that $\mathcal A-\sum_{i=1}^{\rho-1}\varepsilon _i H_i$ is still 
$\pi$-ample. 
Then the $R_j$'s included in $(K_X+\Delta+\mathcal A)_{<0}$ correspond to 
some elements of the above lattice $\Lambda$ in 
Step \ref{c-thm12.2-step3} for which 
$\sum_{i=1}^{\rho-1}\varepsilon _i H_i\cdot z<1/a$.
Therefore, we obtain (2) when $K_X+\Delta$ is $\mathbb Q$-Cartier. 
\end{step}

\begin{step}\label{c-thm12.2-step5} 
In this step, we will prove (3) under the extra 
assumption that $K_X+\Delta$ is $\mathbb Q$-Cartier. 

Let $F$ be a $(K_X+\Delta)$-negative extremal face as in (3). 
The vector space $V=F^{\perp}\subset N^1(X/Y; W)$ is 
defined over $\mathbb Q$ because 
$F$ is generated by some of the $R_j$'s. 
There exists a $\pi$-ample $\mathbb R$-line bundle 
$\mathcal A$ such that $F$ is contained in $(K_X+\Delta+\mathcal A)_{<0}$. 
Let $\langle F\rangle$ be the vector space 
spanned by $F$. 
We put 
\begin{equation*}
\mathcal C_F:=\NE(X/Y; W)_{(K_X+\Delta+\mathcal A)\geq 0}+
\NE(X/Y; W)_{\Nlc (X, \Delta)}+\sum _{R_j\not\subset F}R_j. 
\end{equation*} 
Then $\mathcal C_F$ is a closed cone, 
\begin{equation*}
\NE(X/Y; W)=\mathcal C_F+F, 
\end{equation*} 
and 
\begin{equation*} 
\mathcal C_F\cap \langle F\rangle=\{0\}. 
\end{equation*} 
The support 
functions of $F$ are the elements of $V$ that are 
positive on $\mathcal C_F\setminus \{0\}$. 
This is a non-empty open subset of $V$ and thus it 
contains a rational element that, after scaling, 
gives a line bundle $L$ defined over 
some open neighborhood of $W$ such that $L$ is $\pi$-nef over $W$ and 
that 
$F=L^{\perp}\cap \NE(X/Y; W)$. Therefore, 
$F$ is rational. 
Hence, we obtain (3) when $K_X+\Delta$ is $\mathbb Q$-Cartier. 
\end{step}

We finish the proof of Theorem \ref{c-thm12.2} under the 
extra assumption that $K_X+\Delta$ is $\mathbb Q$-Cartier. 
Therefore, from now on, we can freely use Theorem \ref{c-thm12.2} 
when $K_X+\Delta$ is $\mathbb Q$-Cartier. 

\begin{step}\label{c-thm12.2-step6}
In this final step, we will treat the general case. 
This means that we will treat the case where 
$K_X+\Delta$ is $\mathbb R$-Cartier. 

Let $\mathcal A$ be a $\pi$-ample $\mathbb R$-line bundle on $X$. 
First we will prove (2). 
By Lemma \ref{c-lem4.4}, 
after shrinking $Y$ around $W$ suitably, 
we can take effective $\mathbb Q$-divisors $\Delta_1, \ldots, \Delta_k$ 
on $X$ and positive real numbers $r_1, \ldots, r_k$ with $\sum _{i=1}^kr_i=1$ 
such that 
\begin{equation*}
K_X+\Delta=\sum _{i=1}^k r_i (K_X+\Delta_i)
\end{equation*} 
and 
that $\mathcal J_{\NLC}(X, \Delta_i)=\mathcal J_{\NLC}(X, 
\Delta)$ holds for every $i$. 
Since $K_X+\Delta_i$ is $\mathbb Q$-Cartier, 
there are only finitely many 
$(K_X+\Delta_i+\mathcal A)$-negarive extremal 
rays of $\NE(X/Y; W)$ which are 
rational and relatively ample at $\Nlc(X, \Delta_i)=\Nlc(X, \Delta)$ for 
every $i$. 
Therefore, since 
\begin{equation*}
K_X+\Delta+\mathcal A=\sum _{i=1}^k 
r_i (K_X+\Delta_i+\mathcal A)
\end{equation*} 
holds, 
there exist only finitely many $(K_X+\Delta+\mathcal A)$-negative 
extremal rays of $\NE(X/Y; W)$ which are rational 
and relatively ample at $\Nlc(X, \Delta)$. 
Thus we obtain (2) in full generality. 
The statement (1) is a direct and formal consequence of 
(2). For the details, see, for example, the proof of 
\cite[Chapter III.~1.2 Theorem]{kollar-rational}. 
Finally, we will prove (3). 
Let $F$ be a $(K_X+\Delta)$-negative extremal face of 
$\NE(X/Y; W)$ as in (3). 
By using Lemma \ref{c-lem4.4}, after shrinking $Y$ around $W$ suitably, 
we can 
take an effective $\mathbb Q$-divisor $\Delta^\dag$ on $X$, 
which is sufficiently close to $\Delta$, such that 
$K_X+\Delta^\dag$ is $\mathbb Q$-Cartier, 
$\mathcal J_{\NLC}(X, \Delta^\dag)
=\mathcal J_{\NLC}(X, \Delta)$ holds, 
and $F$ is $(K_X+\Delta^\dag)$-negative. 
Therefore, we see that 
$F$ is a rational face of $\NE(X/Y; W)$. 
This is what we wanted.  
\end{step}
We finish the proof of the cone theorem. 
\end{proof}

\subsection{Proof of Theorem \ref{c-thm1.4} and 
Corollary \ref{c-cor1.5}}\label{c-subsec12.1}

In this subsection, we will prove Theorem \ref{c-thm1.4} as an 
application of the vanishing theorem for projective quasi-log 
schemes (see \cite[Theorem 6.3.5 (ii)]{fujino-foundations}). 
Note that 
Corollary \ref{c-cor1.5} is an easy consequence of Theorem \ref{c-thm1.4}. 
For the details of the framework of quasi-log schemes, 
see \cite[Chapter 6]{fujino-foundations}, 
\cite{fujino-pull-back}, \cite{fujino-cone}, and \cite{fujino-quasi}. 
Let us start with an easy lemma. 

\begin{lem}[{see \cite[Lemma 4.2]{fujino-miyamoto1}}]
\label{c-lem12.3}Let $[V, \omega]$ be an irreducible 
positive-dimensional projective quasi-log scheme 
with $\dim \Nqlc(V, \omega)=0$ or $\Nqlc(V, \omega)=\emptyset$ 
and 
let $\mathcal M$ be an ample line bundle 
on $V$. Assume that 
$\omega+r\mathcal M$ is numerically trivial for some 
real number $r$. 
Then $r\leq \dim V+1$ holds. 
\end{lem}

\begin{proof}
If $r\leq 0$, then $r\leq \dim V+1$ is obvious. 
Hence we may assume that $r$ is positive. 
We consider the following short exact sequence: 
\begin{equation}\label{c-eq12.3}
0\to \mathcal I_{\Nqlc(V, \omega)}\to \mathcal O_V\to 
\mathcal O_{\Nqlc(V, \omega)}\to 0, 
\end{equation} 
where $\mathcal I_{\Nqlc(V, \omega)}$ is the defining ideal sheaf 
of $\Nqlc(V, \omega)$ on $V$. 
Since $l\mathcal M-\omega$ is ample 
for $l>-r$, 
we have 
\begin{equation*} 
H^i(V, \mathcal I_{\Nqlc(V, \omega)}\otimes \mathcal M^{\otimes l})
=0
\end{equation*} 
for every $i\ne 0$ and $l>-r$ by 
the vanishing theorem 
for quasi-log schemes (see \cite[Theorem 6.3.5 (ii)]{fujino-foundations}). 
Since $\dim \Nqlc(V, \omega)=0$ or 
$\Nqlc(V, \omega)=\emptyset$, 
\begin{equation*} 
H^i(V, \mathcal O_{\Nqlc(V, \omega)}\otimes \mathcal M^{\otimes l})=0
\end{equation*} 
for every $i\ne 0$ and every $l$. 
Therefore, we obtain $H^i(V, \mathcal M^{\otimes l})=0$ for 
every $i\ne 0$ and $l>-r$ by \eqref{c-eq12.3}. 
Let $V'$ be the unique maximal (with respect to the 
inclusion) qlc stratum of 
$[V, \omega]$. Then we have the following short exact sequence: 
\begin{equation*}
\xymatrix{
0\ar[r]& \Ker \alpha \ar[r]&\mathcal O_V\ar[r]^-{\alpha}
& \mathcal O_{V'}\ar[r]& 0
}
\end{equation*}
such that $\dim \Supp \Ker\alpha\leq 0$ since 
$\dim \Nqlc(V, \omega)=0$ or $\Nqlc(V, \omega)=\emptyset$. 
Hence we have $H^i(V', \mathcal M^{\otimes l}|_{V'})=0$ for 
every $i\ne 0$ and $l>-r$. 
Since $\dim V'=\dim V>0$, it is obvious 
that $H^0(V', \mathcal M^{\otimes l}|_{V'})=0$ holds 
for every $l<0$. 
We consider 
\begin{equation*}
\chi (t):=\sum _{i=0}^{\dim V'}(-1)^i \dim _{\mathbb C} H^i(V', 
\mathcal M^{\otimes t}|_{V'}). 
\end{equation*} 
Then it is well known that $\chi (t)$ is a non-trivial polynomial of 
$\deg \chi (t)=\dim V'=\dim V$. By the above observation, 
$\chi (l)=0$ for $l \in \mathbb Z$ with $-r<l<0$. 
This implies that $r\leq \dim V+1$. 
We finish the proof. 
\end{proof}

Let us start the proof of Theorem \ref{c-thm1.4}. 

\begin{proof}[Proof of Theorem \ref{c-thm1.4}]
We put $\mathcal L:=f^*\mathcal A_{Y^\flat}$. 
Without loss of generality, we may assume that $\dim X\geq 1$. 
Since $\mathcal L$ is $\pi$-nef over $W$, 
$R$ is a $(K_X+\Delta)$-negative extremal ray of $\NE(X/Y; W)$. 
Therefore, by Theorem \ref{c-thm12.2} (3) and 
Theorem \ref{c-thm12.1}, 
after shrinking $Y$ around $W$ suitably, 
we obtain a contraction morphism 
$\varphi_R\colon X\to Z$ over $Y$ associated to $R$. 
It is sufficient to prove that $R\cdot \mathcal L=0$ holds. 
\setcounter{step}{0}
\begin{step}\label{c-thm1.4-step1}
In this step, we will treat the case where $\dim Z=0$. 

From now on, we assume that $\dim Z=0$ holds. 
Then $X$ is projective with $\rho (X)=1$ and $\mathcal L$ is a 
nef line bundle on $X$ in the usual sense. Suppose that 
$\mathcal L$ is ample. 
Then we obtain that $K_X+\Delta+r\mathcal L$ is numerically 
trivial for some $r>\dim X+1$ since 
$(K_X+\Delta+(\dim X+1)\mathcal L)\cdot R<0$ and $\rho(X)=1$. 
We note that $[X, K_X+\Delta]$ naturally becomes 
an irreducible projective quasi-log scheme with 
$\Nqlc(X, K_X+\Delta)=\emptyset$ 
(see, for example, \cite[6.4.1]{fujino-foundations}). 
Therefore, we get a contradiction by Lemma \ref{c-lem12.3}. 
This implies that $\mathcal L$ is numerically trivial, 
that is, $R\cdot f^*\mathcal A_{Y^\flat}=R\cdot 
\mathcal L=0$. 
This is what we wanted. 
\end{step}
\begin{step}\label{c-thm1.4-step2}
In this step, we will treat the case where $\dim Z\geq 1$. 

From now on, we assume that $\dim Z\geq 1$ holds. 
Then we can always take a point $P\in Z$ such that 
$\dim \varphi^{-1}_R(P)\geq 1$. 
We shrink $Z$ around $P$ and assume that $Z$ is 
Stein. 
Then we can take an effective $\mathbb R$-Cartier divisor 
$B$ on $Z$ such that $(X, \Delta+\varphi^*_RB)$ is log canonical 
outside $\varphi^{-1}_R(P)$, 
there exists a positive-dimensional log canonical 
center $C$ of $(X, \Delta+\varphi^*_RB)$ with 
$\varphi_R(C)=P$, and 
$\dim \Nlc(X, \Delta+\varphi^*_RB)=0$ or $\Nlc(X, 
\Delta+\varphi^*_RB)=\emptyset$. 
After shrinking $Z$ around $P$ suitably again, 
we can take a projective bimeromorphic 
morphism $f\colon Y\to X$ from a smooth complex variety $Y$ 
such that $f^{-1}(C)$ and the exceptional locus $\Exc(f)$ of 
$f$ are both simple normal crossing divisors on $Y$ and 
that the union of $f^{-1}(C)$, 
$\Exc(f)$, and $\Supp \left(f^{-1}_*(\Delta+\varphi^*_RB)\right)$ 
is a simple normal crossing divisor on $Y$ 
(see \cite[Theorem 13.2]{bierstone-milman}). 
We define $B_Y$ by the formula 
$K_Y+B_Y=f^*(K_X+\Delta+\varphi^*B)$. 
Then we see that 
$\Supp B_Y$ is a simple normal crossing divisor 
on $Y$. 
By shrinking $X$ around $C$, we assume that 
$(X\setminus C)\cap \Nlc(X, \Delta+\varphi^*_RB)=\emptyset$. 
Let $T$ be the union of the irreducible components 
of $B^{=1}_Y$ that are mapped to $C$ by $f$. 
We define $B_T$ by adjunction:~$K_T+B_T=(K_Y+B_Y)|_T$.  
We consider 
the following short exact sequence: 
\begin{equation*}
\begin{split}
0&\to \mathcal O_Y(\lceil -(B^{<1}_Y)\rceil-\lfloor B^{>1}_Y\rfloor-T)\to 
\mathcal O_Y(\lceil -(B^{<1}_Y)\rceil-\lfloor B^{>1}_Y\rfloor)\\ 
&\to 
\mathcal O_T(\lceil -(B^{<1}_T)\rceil-\lfloor B^{>1}_T\rfloor)\to 0. 
\end{split}
\end{equation*} 
We note that 
\begin{equation}\label{c-eq12.4}
\lceil -(B^{<1}_Y)\rceil-\lfloor B^{>1}_Y\rfloor-T-
(K_Y+\{B_Y\}+B^{=1}_Y-T)=-f^*(K_X+\Delta+\varphi^*B). 
\end{equation} 
By taking $R^if_*$, 
we have a long exact sequence: 
\begin{equation}\label{c-eq12.5}
\begin{split}
0&\longrightarrow f_*\mathcal O_Y(\lceil -(B^{<1}_Y)
\rceil-\lfloor B^{>1}_Y\rfloor-T)\longrightarrow 
\mathcal J_{\NLC}(X, \Delta+\varphi^*B)\\ 
&\longrightarrow 
f_*\mathcal O_T(\lceil -(B^{<1}_T)\rceil-\lfloor B^{>1}_T\rfloor)
\overset{\delta}{\longrightarrow} 
R^1f_* \mathcal O_Y(\lceil -(B^{<1}_Y)
\rceil-\lfloor B^{>1}_Y\rfloor-T)\longrightarrow \cdots. 
\end{split}
\end{equation}
The support of $f_*\mathcal O_T(\lceil -(B^{<1}_T)\rceil-
\lfloor B^{>1}_T\rfloor)$ 
is contained in $C$ since $f(T)=C$. On the 
other hand, 
any associated subvarieties of 
$R^1f_*\mathcal O_Y(\lceil -(B^{<1}_Y)\rceil-\lfloor B^{>1}_Y\rfloor-T)$ 
are not 
contained in $C$ by 
\eqref{c-eq12.4} and Theorem \ref{c-thm5.5} (i).  
Hence,  
the connecting homomorphism $\delta$ in \eqref{c-eq12.5} is zero. 
We put 
\begin{equation*}
\mathcal J:=f_*\mathcal O_Y(\lceil -(B^{<1}_Y)
\rceil-\lfloor B^{>1}_Y\rfloor-T). 
\end{equation*}  
Then it is an ideal sheaf contained in $\mathcal J_{\NLC}(X, \Delta+\varphi^*B)$. 
Let $X'$ denote the closed analytic subspace of $X$ defined 
by $\mathcal J$. 
By applying the snake lemma to the following commutative diagram: 
\begin{equation*}
\xymatrix{
&0\ar[r]& \mathcal J 
\ar[r]\ar@{^{(}->}[d]& \mathcal J_{\NLC}(X, \Delta+\varphi^*B)
\ar@{^{(}->}[d]\ar[r]& f_*\mathcal O_T(\lceil -(B^{<1}_T)\rceil-
\lfloor B^{>1}_T\rfloor)\ar[r]\ar[d] 
&0\\ 
&0\ar[r]& \mathcal O_X\ar@{=}[r] & \mathcal O_X\ar[r]& 0\ar[r]&0, 
}
\end{equation*}
we obtain the short exact sequence: 
\begin{equation*}
0\to f_*\mathcal O_T(\lceil -(B^{<1}_T)\rceil-\lfloor B^{>1}_T\rfloor)
\to \mathcal O_{X'} \to  
\mathcal O_{\Nlc(X, \Delta+\varphi^*B)}\to 0. 
\end{equation*}
Since $C$ is projective and $T$ is projective over $C$ by construction, 
\begin{equation*}
\left(X', (K_X+\Delta)|_{X'}, f\colon (T, B_T)\to X'\right)
\end{equation*} 
is a projective quasi-log scheme 
with $\Nqlc(X', (K_X+\Delta)|_{X'})=\Nlc(X, \Delta+\varphi^*_RB)$ 
by \cite[Theorem 4.9]{fujino-pull-back}. 
In particular, $\dim \Nqlc(X', (K_X+\Delta)|_{X'})=0$ 
or $\Nqlc(X', (K_X+\Delta)|_{X'})=\emptyset$ holds. 
We note that $X'=C$ holds set theoretically 
by construction. 
We put $\omega:=(K_X+\Delta)|_{X'}$. 
Then $-\omega$ is ample 
since $\varphi_R(X')=P$. 

Suppose that $R\cdot \mathcal L>0$ holds. 
Then $\mathcal L':=\mathcal L|_{X'}$ is ample 
and $\omega+r\mathcal L'$ is numerically trivial 
on $X'$ for some positive real number $r$ with $r>\dim X+1
>\dim X'+1$. 
This is a contradiction by Lemma \ref{c-lem12.3}. 
Hence we obtain $R\cdot \mathcal L=0$. 
Therefore, $R$ is a $(K_X+\Delta)$-negative 
extremal ray of $\NE(X/{Y^\flat}; g^{-1}(W))$.  
\end{step}

We finish the proof. 
\end{proof}

\begin{proof}[Proof of Corollary \ref{c-cor1.5}]
Let $P\in Y$ be any point. We put $W:=\{P\}$. 
Then the dimension of $N^1(X/Y; W)$ is finite by Theorem 
\ref{c-thm11.10}. 
Suppose that $K_X+\Delta+(\dim X+1)\mathcal A$ is not $\pi$-nef 
over $W$. 
Then there exists a $(K_X+\Delta+(\dim X+1)\mathcal A)$-negative 
extremal ray $R$ of $\NE(X/Y; W)$. 
We put $Y^\flat:=Y$, $\mathcal A_{Y^\flat}:=\mathcal A$, 
and $f:=id_Y$. 
Then we use Theorem \ref{c-thm1.4}. 
Thus we obtain $R\cdot \mathcal A=0$. 
This is a contradiction since $\mathcal A$ is $\pi$-ample 
over $W$. 
Therefore, $K_X+\Delta+(\dim X+1)\mathcal A$ is 
$\pi$-nef over $W$. 
Since $P$ is any point of $Y$, this means 
that $K_X+\Delta+(\dim X+1)\mathcal A$ is nef over 
$Y$. 
\end{proof}

We close this section with a remark on Theorem \ref{c-thm1.4} and 
Corollary \ref{c-cor1.5}. 

\begin{rem}\label{c-rem12.4} 
In Theorem \ref{c-thm1.4} and Corollary \ref{c-cor1.5}, 
we can replace $(\dim X+1)$ with $\dim X$ when $\pi(X)$ is not 
a point. 
We can check it easily by the proof of Theorem \ref{c-thm1.4}. 
\end{rem}

\section{Lengths of extremal rational curves}\label{c-sec13}

In this section, we will quickly explain that 
every extremal ray is spanned by a rational curve. 
Our result in this section generalizes Kawamata's 
famous result in \cite{kawamata}. 
We first prove the following theorem as an application of 
\cite[Theorem 1.12]{fujino-cone}. 

\begin{thm}\label{c-thm13.1}
Let $\varphi 
\colon X\to Z$ be a projective morphism of complex analytic spaces 
such that $X$ is a normal complex variety and let $\Delta$ be an 
effective $\mathbb R$-divisor on $X$ such that 
$K_X+\Delta$ is $\mathbb R$-Cartier. 
Assume that 
$-(K_X+\Delta)$ is $\varphi$-ample. 
Let $P$ be an arbitrary point of $Z$. 
Let $E$ be any positive-dimensional irreducible component 
of $\varphi^{-1}(P)$ such that $E\not\subset \Nlc(X, \Delta)$. 
Then $E$ is covered by possibly singular rational curves $\ell$ 
with 
\begin{equation*}
0<-(K_X+\Delta)\cdot \ell\leq 2 \dim E. 
\end{equation*} 
In particular, $E$ is uniruled. 
\end{thm}

In the proof of Theorem \ref{c-thm13.1}, we will use the theory of 
quasi-log schemes (see \cite[Chapter 6]{fujino-foundations}, 
\cite{fujino-pull-back}, \cite{fujino-cone}, and \cite{fujino-quasi}). 

\begin{proof}[Proof of Theorem \ref{c-thm13.1}]
If $\varphi(X)=P$, then $E=X$ obviously holds. 
In this case, the statement follows from 
\cite[Theorem 1.12]{fujino-cone} since we can see $[X, K_X+\Delta]$ 
as a projective quasi-log scheme 
(see, for example, \cite[6.4.1]{fujino-foundations}). 
Therefore, from now on, we may assume that $\varphi(X)\ne P$. 
We shrink $Z$ around $P$ and may assume that $Z$ is 
Stein. 
Then 
we can take an effective $\mathbb R$-Cartier divisor 
$B$ on $Z$ such that $E$ is a log canonical 
center of $(X, \Delta+\varphi^*B)$. 
After shrinking $Z$ around $P$ suitably again, 
we can take a projective bimeromorphic 
morphism $f\colon Y\to X$ from a smooth complex 
variety $Y$ such that $f^{-1}(E)$ is a simple 
normal crossing divisor on $Y$, 
\begin{equation*}
K_Y+B_Y:=f^*(K_X+\Delta+\varphi^*B), 
\end{equation*} 
and $\Supp B_Y$ is a simple normal crossing divisor 
on $Y$ (see \cite[Theorem 13.2]{bierstone-milman}). 
We may further assume that 
the support of the union of $f^{-1}(E)$ and $\Supp B_Y$ is 
also a simple normal crossing divisor on $Y$. 
Let $T$ be the union of the irreducible components 
of $B^{=1}_Y$ that are mapped to $E$ by $f$. 
We put $A:=\lceil -(B^{<1}_Y)\rceil$ and 
$N:=\lfloor B^{>1}_Y\rfloor$ and consider 
the following short exact sequence: 
\begin{equation*}
0\to \mathcal O_Y(A-N-T)\to \mathcal O_Y(A-N)\to 
\mathcal O_T(A-N)\to 0. 
\end{equation*} 
We note that 
\begin{equation}\label{c-thm13.1-eq1}
A-N-T-(K_Y+\{B_Y\}+B^{=1}_Y-T)=-f^*(K_X+\Delta+\varphi^*B). 
\end{equation} 
By taking $R^if_*$, 
we have a long exact sequence: 
\begin{equation}\label{c-thm13.1-eq2}
\begin{split}
0&\longrightarrow f_*\mathcal O_Y(A-N-T)\longrightarrow 
\mathcal J_{\NLC}(X, \Delta+\varphi^*B)\longrightarrow 
f_*\mathcal O_T(A-N)\\&\overset{\delta}{\longrightarrow} 
R^1f_* \mathcal O_Y(A-N-T)\longrightarrow \cdots. 
\end{split}
\end{equation}
The support of $f_*\mathcal O_T(A-N)$ 
is contained in $E$ since $f(T)=E$. On the 
other hand, 
any associated subvarieties of $R^1f_*\mathcal O_Y(A-N-T)$ are not 
contained in $E$ by \eqref{c-thm13.1-eq1} and 
Theorem \ref{c-thm5.5} (i).  
Hence, 
the connecting homomorphism $\delta$ in 
\eqref{c-thm13.1-eq2} is zero. 
We put $\mathcal J:=f_*\mathcal O_Y(A-N-T)$. 
Then it is an ideal sheaf contained in 
$\mathcal J_{\NLC}(X, \Delta+\varphi^*B)$. 
Let $X'$ denote the closed analytic subspace of $X$ defined 
by $\mathcal J$. 
Thus we obtain the following big commutative diagram. 
\begin{equation*}
\xymatrix{
&& 0 \ar[d]& 0\ar[d]& &\\ 
&0\ar[r]& \mathcal J 
\ar[r]\ar[d]& \mathcal J_{\NLC}(X, \Delta+\varphi^*B)
\ar[d]\ar[r]& f_*\mathcal O_T(A-N)\ar[r] 
&0\\ 
&& \mathcal O_X\ar[d]\ar@{=}[r] & \mathcal O_X\ar[d]& & \\ 
0\ar[r] & f_*\mathcal O_T(A-N)\ar[r] & \mathcal O_{X'} \ar[r]\ar[d]& 
\mathcal O_{\Nlc(X, \Delta+\varphi^*B)}\ar[d] \ar[r]& 0&&\\ 
&& 0&0 && 
}
\end{equation*}
We note that $X'=E\cup \Nlc(X, \Delta+\varphi^*B)$ holds 
set theoretically. On $X'$, by 
the above big commutative diagram, 
we see that $\Nlc(X, \Delta+\varphi^*B)$ is 
defined by the ideal sheaf $f_*\mathcal O_T(A-N)$. 
We write $T=T'+T''$, where $T''$ is the union of 
the irreducible components of $T$ mapped 
to $\Nlc(X, \Delta+\varphi^*B)$ by $f$. 
We put $\mathcal I:=f_*\mathcal O_{T''}(A-N-T')$. 
Then 
\begin{equation*}
\mathcal I=f_*\mathcal O_{T''}(A-N-T')\subset 
f_*\mathcal O_T(A-N)\subset \mathcal O_{X'}. 
\end{equation*} 
Since $\mathcal I\subset f_*\mathcal O_T(A-N)$, $\mathcal I$ 
is zero when 
it is restricted to $\Nlc(X, \Delta+\varphi^*B)$. 
Since $f(T'')\subset 
\Nlc(X, \Delta+\varphi^*B)$, $\mathcal I$ is zero on $X'\setminus 
\Nlc(X, \Delta+\varphi^*B)$. 
Thus, we have $\mathcal I=\{0\}$. 
By construction, we see that $E$ is a closed analytic subvariety 
of $X'$. 
Let $\mathcal I_E$ be the defining ideal sheaf of $E$ on $X'$. 
Then we obtain 
\begin{equation*}
\mathcal I_E\cap f_*\mathcal O_T(A-N)\subset 
f_*\mathcal O_{T''}(A-N-T')=\mathcal I=\{0\}. 
\end{equation*} 
Hence, we can see $f_*\mathcal O_T(A-N)$ as an 
ideal sheaf on $E$. 
Since $E$ is projective and $T$ is projective over $E$ by construction, 
\begin{equation*}
\left(E, (K_X+\Delta)|_E, f\colon (T, B_T)\to E\right)
\end{equation*} 
is a projective quasi-log scheme, where $K_T+B_T:=(K_Y+B_Y)|_T$, 
by \cite[Theorem 4.9]{fujino-pull-back}. 
Since $\varphi(E)=P$, 
$-(K_X+\Delta)|_E$ is ample. 
Thus, by \cite[Theorem 1.12]{fujino-cone}, 
$E$ is covered by possibly singular rational curves $\ell$ with 
$0<-(K_X+\Delta)\cdot \ell\leq 2\dim E$. 
In particular, this implies that $E$ is uniruled. 
\end{proof}

Theorem \ref{c-thm13.2} is an easy consequence of Theorem \ref{c-thm13.1}. 
It seems to be indispensable for the minimal model 
program with scaling. 

\begin{thm}[Lengths of extremal rational curves]\label{c-thm13.2}
Let $\pi\colon X\to Y$ be a projective morphism of 
complex analytic spaces such that 
$X$ is a normal complex variety and let $W$ be a compact subset of $Y$ 
such that the dimension of $N^1(X/Y; W)$ is finite. 
Let $\Delta$ be an effective $\mathbb R$-divisor on $X$ such that 
$K_X+\Delta$ is $\mathbb R$-Cartier. 
If $R$ is a $(K_X+\Delta)$-negative extremal ray of $\NE(X/Y; W)$ 
which is relatively ample at $\Nlc(X, \Delta)$, 
then there exists a rational curve $\ell$ spanning $R$ with 
\begin{equation*}
0<-(K_X+\Delta)\cdot \ell\leq 2\dim X. 
\end{equation*}
\end{thm}

\begin{proof}
By the cone and 
contraction theorem (see Theorems \ref{c-thm12.1} and \ref{c-thm12.2}), 
after shrinking $Y$ around $W$ suitably, 
we obtain a contraction morphism $\varphi\colon X\to Z$ 
over $Y$ associated to $R$. 
We note that $-(K_X+\Delta)$ is $\varphi$-ample and 
$\varphi\colon \Nlc(X, \Delta)\to \varphi(\Nlc(X, \Delta))$ 
is finite by construction. 
Therefore, we can find a rational curve $\ell$ in a fiber 
of $\varphi$ with $0<-(K_X+\Delta)\cdot \ell \leq 2\dim X$ by 
Theorem \ref{c-thm13.1}. 
This $\ell$ is a desired rational curve spanning $R$. 
\end{proof}

\section{On Shokurov's polytopes}\label{c-sec14}

In this section, we will discuss Shokurov's polytopes 
for projective morphisms of complex analytic spaces. 
Here, we will follow the presentation in 
\cite[Section 4.7]{fujino-foundations}. 
Let us recall the definition of {\em{extremal curves}}. 

\begin{defn}[Extremal curves]\label{c-def14.1} 
Let $\pi\colon X\to Y$ 
be a projective morphism of complex analytic spaces such that 
$X$ is a normal complex variety and let $W$ be a compact 
subset of $Y$ such that the dimension of $N^1(X/Y; W)$ is finite. 
A curve $\Gamma$ on $X$ is called {\em{extremal over $W$}} 
if the following properties hold. 
\begin{itemize}
\item[(i)] $\Gamma$ generates an extremal ray $R$ of $\NE(X/Y; W)$. 
\item[(ii)] There exists a $\pi$-ample 
line bundle $\mathcal H$ over some open neighborhood 
of $W$ such that 
\begin{equation*}
\mathcal H\cdot \Gamma=\min_{\ell} \{\mathcal H\cdot \ell\}, 
\end{equation*} 
where $\ell$ ranges over curves generating $R$. 
\end{itemize}
\end{defn}

By Lemma \ref{c-thm13.2}, we have: 

\begin{lem}\label{c-lem14.2}
Let $\pi\colon X\to Y$ be a projective morphism of complex 
analytic spaces and let $(X, \Delta)$ be a log canonical pair. 
Let $W$ be a compact subset of $Y$ such that 
the dimension of $N^1(X/Y; W)$ is finite. 
Let $R$ be a $(K_X+\Delta)$-negative extremal ray 
of $\NE(X/Y; W)$. 
If $\Gamma$ is an extremal curve over $W$ generating $R$, then 
\begin{equation*}
0<-(K_X+\Delta)\cdot \Gamma \leq 2\dim X
\end{equation*} 
holds. 
\end{lem}
\begin{proof}
By Theorem \ref{c-thm13.2}, 
we can take a rational curve $\ell$ spanning $R$ such that 
\begin{equation*}
0<-(K_X+\Delta)\cdot \ell \leq 2\dim X. 
\end{equation*} 
Let $\mathcal H$ be a line bundle as in Definition \ref{c-def14.1}. Then 
\begin{equation*}
\frac{-(K_X+\Delta)\cdot \Gamma}{\mathcal H\cdot \Gamma} 
=\frac{-(K_X+\Delta)\cdot \ell}{\mathcal H\cdot \ell}. 
\end{equation*} holds. 
Hence we obtain 
\begin{equation*}
0<-(K_X+\Delta)\cdot \Gamma =\left(-(K_X+\Delta)\cdot 
\ell\right)\cdot \frac{\mathcal H\cdot \Gamma}{\mathcal H\cdot 
\ell} \leq 2\dim X. 
\end{equation*}
This is what we wanted. 
\end{proof}

One of the main purposes of this section is to explain the 
following theorem, which is very well known and 
has already played an important role when $\pi\colon X\to Y$ 
is algebraic. 

\begin{thm}\label{c-thm14.3}
Let $\pi\colon X\to Y$ 
be a projective morphism of complex 
analytic spaces such that $X$ is a normal complex variety 
and let $W$ be a compact 
subset of $Y$ 
such that the dimension of $N^1(X/Y; W)$ is finite. 
Let $V$ be a finite-dimensional 
affine subspace of $\WDiv_{\mathbb R}(X)$, which 
is defined over the rationals. 
We fix an $\mathbb R$-divisor $\Delta\in \mathcal L(V; \pi^{-1}(W))$, 
that is, $\Delta\in V$ and $(X, \Delta)$ is log canonical at $\pi^{-1}(W)$. 
Then we can find positive real numbers $\alpha$ and $\delta$, which 
depend on $(X, \Delta)$ and $V$, with the following properties. 
\begin{itemize}
\item[(1)] If $\Gamma$ is any extremal curve over $W$ and 
$(K_X+\Delta)\cdot \Gamma>0$, 
then $(K_X+\Delta)\cdot \Gamma >\alpha$. 
\item[(2)] If $D\in \mathcal L(V; \pi^{-1}(W))$, 
$|\!| D-\Delta|\!|<\delta$, and $(K_X+D)\cdot \Gamma \leq 0$ for 
an extremal curve $\Gamma$ over $W$, 
then $(K_X+\Delta)\cdot \Gamma\leq 0$. 
\item[(3)] Let $\{R_t\}_{t\in T}$ be any set of extremal rays of $\NE(X/Y; W)$. 
Then 
\begin{equation*}
\mathcal N_T:=\{D\in \mathcal L(V; \pi^{-1}(W))\, |\, 
{\text{$(K_X+D)\cdot R_t\geq 0$ for every $t\in T$}}\} 
\end{equation*} 
is a rational polytope in $V$. 
In particular, 
\begin{equation*}
\mathcal N^\sharp _{\pi} (V; W):=\{\Delta\in 
\mathcal L(V; \pi^{-1}(W)) \, |\, 
{\text{$K_X+\Delta$ is nef over $W$}}\}
\end{equation*}
is a rational polytope. 
\end{itemize}
\end{thm}

\begin{proof}[Proof of Theorem \ref{c-thm14.3}]

Throughout this proof, we can freely shrink $Y$ around $W$ suitably. 
We first note that $\mathcal L(V; \pi^{-1}(W))$ is a rational polytope 
in $V$ (see \ref{c-say2.10}). 

(1) If $\Delta$ is a $\mathbb Q$-divisor, 
then we may assume that $m(K_X+\Delta)$ is 
Cartier for some positive integer $m$ by shrinking $Y$ around 
$W$ suitably. 
Therefore, the statement is obvious even if $\Gamma$ is not extremal. 
Hence, from now on, we assume 
that $\Delta$ is not a $\mathbb Q$-divisor. 
Then we can write $K_X+\Delta=\sum _j a_j (K_X+D_j)$ as 
in Lemma \ref{c-lem4.4}. This means 
that $a_j$ is a positive real number for every $j$ with 
$\sum _j a_j=1$ and that $D_j\in \mathcal L(V; \pi^{-1}(W))$ is 
a $\mathbb Q$-divisor for every $j$. 
Thus we have 
$(K_X+\Delta)\cdot \Gamma =\sum _j a_j (K_X+D_j)\cdot \Gamma$. 
If $(K_X+\Delta)\cdot \Gamma<1$, 
then 
\begin{align*}
-2\dim X\leq (K_X+D_{j_0})\cdot \Gamma &<\frac{1}{a_{j_0}}\left\{-\sum 
_{j\ne j_0}a_j (K_X+D_j)\cdot \Gamma+1\right\}\\ 
&\leq \frac{2\dim X+1}{a_{j_0}}
\end{align*}
for $a_{j_0}\ne 0$. 
This is because $(K_X+D_j)\cdot \Gamma \geq -2\dim X$ holds for 
every $j$ by Lemma \ref{c-lem14.2}. 
Thus there are only finitely many possibilities 
of the intersection numbers $(K_X+D_j)\cdot \Gamma$ for 
$a_j\ne 0$ when $(K_X+\Delta)\cdot \Gamma <1$. 
Therefore, the existence of $\alpha$ is obvious. 

(2) If we take $\delta$ sufficiently small, 
then, for every $D\in \mathcal L(V; \pi^{-1}(W))$ with 
$|\!|D-\Delta|\!|<\delta$, 
we can always find $D'\in \mathcal L(V; \pi^{-1}(W))$ such that 
\begin{equation*}
K_X+D =(1-s)(K_X+\Delta)+s(K_X+D')
\end{equation*} 
with 
\begin{equation*} 
0\leq s\leq \frac{\alpha}{\alpha+2\dim X}. 
\end{equation*} 
Since $\Gamma$ is extremal, we have 
$(K_X+D')\cdot \Gamma \geq -2\dim X$ for 
every $D'\in \mathcal L(V; \pi^{-1}(W))$ by Lemma \ref{c-lem14.2}. 
We assume that $(K_X+\Delta)\cdot \Gamma>0$. Then 
$(K_X+\Delta)\cdot \Gamma >\alpha$ by (1). 
Therefore, 
\begin{align*}
(K_X+D)\cdot \Gamma&=(1-s)(K_X+\Delta)\cdot \Gamma +s(K_X+D')\cdot \Gamma 
\\ &>(1-s)\alpha +s(-2\dim X)\geq 0. 
\end{align*} 
This is a contradiction. 
Hence, we obtain $(K_X+\Delta)\cdot \Gamma\leq 0$. 
We complete the proof of (2). 

(3) For every $t\in T$, we may assume that 
there is some $D(t)\in \mathcal L(V; \pi^{-1}(W))$ such that 
$(K_X+D(t))\cdot R_t<0$. 
Let $B_1, \ldots, B_r$ be the vertices of $\mathcal L(V; \pi^{-1}(E))$. 
We note that 
$(K_X+D)\cdot R_t<0$ for some $D \in \mathcal L(V; \pi^{-1}(W))$ 
implies $(K_X+B_j)\cdot R_t<0$ for some $j$. 
Therefore, we may assume that 
$T$ is contained in $\mathbb N$. 
This is because there are only countably many 
$(K_X+B_j)$-negative 
extremal rays for every $j$ by 
the cone theorem (see Theorem \ref{c-thm12.2}). 
We note that $\mathcal N_T$ is a closed 
convex subset of $\mathcal L(V; \pi^{-1}(W))$ by 
definition. 
If $T$ is a finite set, 
then the claim is obvious. 
Thus, we may assume that 
$T=\mathbb N$. 
By (2) and by the compactness of $\mathcal N_T$, 
we can take $\Delta _1, \ldots, \Delta_n\in \mathcal N_T$ and 
$\delta_1, \ldots, \delta_n>0$ such that 
$\mathcal N_T$ is covered by 
\begin{equation*} 
\mathcal B_i=\{D \in \mathcal L(V; \pi^{-1}(W))\, |\, |\!|D-\Delta_i|\!|<\delta_i\}
\end{equation*} 
and that if $D\in \mathcal B_i$ with $(K_X+D)\cdot R_t<0$ for 
some $t$, then $(K_X+\Delta_i)\cdot R_t=0$. 
If we put 
\begin{equation*} 
T_i=\{t\in T\, |\, (K_X+D)\cdot R_t<0 \ \text{for some} 
\ D\in \mathcal B_i\}, 
\end{equation*} 
then $(K_X+\Delta_i)\cdot R_t=0$ for every $t \in T_i$ by the above 
construction. 
Since $\{\mathcal B_i\}_{i=1}^n$ 
gives an open covering of $\mathcal N_T$, 
we have $\mathcal N_T=\bigcap _{1\leq i\leq n}
\mathcal N_{T_i}$ by the following claim. 
\setcounter{cla}{0}
\begin{cla}\label{c-thm14.3-claim1}
$\mathcal N_T=\bigcap _{1\leq i\leq n}\mathcal N_{T_i}$. 
\end{cla}
\begin{proof}[Proof of Claim \ref{c-thm14.3-claim1}]
We note that $\mathcal N_T\subset 
\bigcap _{1\leq i \leq n}\mathcal N_{T_i}$ is 
obvious. 
Suppose that 
$\mathcal N_T\subsetneq \bigcap _{1\leq i\leq n}\mathcal N_{T_i}$ holds. 
We take $D \in \bigcap _{1\leq i \leq n}\mathcal N_{T_i}\setminus 
\mathcal N_T$ which is very close to $\mathcal N_T$. 
Since $\mathcal N_T$ is covered by 
$\{\mathcal B_i\}_{i=1}^n$, 
there is some $i_0$ such that 
$D\in \mathcal B_{i_0}$. 
Since $D\not \in \mathcal N_T$, 
there is some $t_0\in T$ such that 
$(K_X+D)\cdot R_{t_0}<0$. 
Thus, $t_0\in T_{i_0}$. 
This is a contradiction 
because $D\in \mathcal N_{T_{i_0}}$. 
Therefore, we obtain the desired equality 
$\mathcal N_T=\bigcap _{1\leq i\leq n}
\mathcal N_{T_i}$. 
\end{proof}
Therefore, it is sufficient to 
see that 
each $\mathcal N_{T_i}$ is a rational polytope in $V$. 
By replacing $T$ with 
$T_i$, 
we may assume that 
there is some $D\in \mathcal N_T$ such that 
$(K_X+D)\cdot R_t=0$ for every $t\in T$. 

\begin{cla}\label{c-thm14.3-claim2}
If $\dim _{\mathbb R}\mathcal L(V; \pi^{-1}(W))=1$, 
then $\mathcal N_T$ is a rational polytope in $V$. 
\end{cla}

\begin{proof}[Proof of Claim \ref{c-thm14.3-claim2}]
As we explained above, we can take 
some $D\in \mathcal N_T$ such that 
$(K_X+D)\cdot R_t=0$ for every $t\in T$. 
Since $\dim _{\mathbb R} \mathcal L(V; \pi^{-1}(W))=1$, 
we can write  
\begin{equation*} 
K_X+D=b_1(K_X+D_1)+b_2(K_X+D_2)
\end{equation*} 
such that 
$K_X+D_i\in \mathcal L(V; \pi^{-1}(W))$, 
$D_i$ is a $\mathbb Q$-divisor, and 
$0\leq b_i \leq 1$ for $i=1, 2$, and 
$b_1+b_2=1$. 
By $(K_X+D)\cdot R_t=0$, 
we see that $b_1$ and $b_2$ are rational numbers. 
This implies that $D$ is a $\mathbb Q$-divisor. 
Therefore, $\mathcal N_T$ is a rational polytope in $V$. 
\end{proof}

Hence we assume $\dim _{\mathbb R}
\mathcal L(V; \pi^{-1}(W))>1$. 
Let $\mathcal L^1, \ldots, \mathcal L^p$ be the 
proper faces of $\mathcal L(V; \pi^{-1}(W))$. Then 
$\mathcal N_{T}^i=\mathcal N_T\cap \mathcal L^i$ is a rational 
polytope by induction on dimension. 
Moreover, for each $D''\in \mathcal N_T$ which is not $D$, 
there is $D'$ on some proper face of $\mathcal L(V; \pi^{-1}(W))$ such that 
$D''$ is on the line segment determined by $D$ and $D'$. 
Note that $(K_X+D)\cdot R_t=0$ for every $t\in T$. 
Therefore, if $D'\in \mathcal L^i$, 
then $D'\in \mathcal N_{T}^i$. 
Thus, $\mathcal N_T$ is the convex hull of $D$ and all 
the $\mathcal N_{T}^i$. 
There is a finite subset $T'\subset T$ such that 
\begin{equation*} 
\bigcup _i \mathcal N_T^i=\mathcal N_{T'}\cap (\bigcup _i\mathcal L^i). 
\end{equation*} 
Therefore, the convex hull of $D$ 
and $\bigcup_i \mathcal N_{T}^i$ is just $\mathcal N_{T'}$. 
We complete the proof of (3). 
\end{proof}

As an application of Theorem \ref{c-thm14.3}, 
we have: 

\begin{thm}\label{c-thm14.4} 
Let $\pi\colon X\to Y$ 
be a projective morphism of complex analytic spaces such that 
$X$ is a normal complex variety and let $W$ be a compact 
subset of $Y$ such that the dimension of $N^1(X/Y; W)$ is finite. 
Let $(X, \Delta)$ be a log canonical pair and let $H$ be an effective 
$\mathbb R$-Cartier $\mathbb R$-divisor on $X$ such that 
$(X, \Delta+H)$ is log canonical and that $K_X+\Delta+H$ is nef over 
$W$. 
Then, either $K_X+\Delta$ is nef over $W$ or there 
is a $(K_X+\Delta)$-negative extremal ray $R$ of $\NE(X/Y; W)$ such that 
$(K_X+\Delta+\lambda H)\cdot R=0$, 
where 
\begin{equation*}
\lambda:=\inf \{t\geq 0
\, |\, {\text{$K_X+\Delta+tH$ is nef over $W$}}\}. 
\end{equation*} 
Of course, $K_X+\Delta+\lambda H$ is nef over $W$. 
\end{thm}

\begin{proof}We assume 
that $K_X+\Delta$ is not $\pi$-nef over $W$. 
Let $\{R_j\}$ be the set of the $(K_X+\Delta)$-negative extremal rays 
of $\NE(X/Y; W)$. 
Let $C_j$ be an extremal curve over $W$ 
spanning $R_j$ for every $j$. 
We put 
$\mu=\underset{j}{\sup} \{\mu_j\}$, 
where 
\begin{equation*} 
\mu_j =\frac{-(K_X+\Delta)\cdot C_j}{H\cdot C_j}. 
\end{equation*} 
By definition, it is obvious that $\lambda=\mu$ and 
$0<\mu\leq 1$ hold. 
Hence it is sufficient to 
prove that $\mu=\mu_{j_0}$ for some $j_0$. 
By Lemma \ref{c-lem4.4}, after shrinking 
$Y$ around $W$ suitably, 
we can find effective $\mathbb Q$-divisors 
$\Delta_1, \ldots, \Delta_k$ and positive real numbers 
$r_1, \ldots, r_k$ with 
$\sum _{i=1}^kr_i=1$ such that $m(K_X+\Delta_i)$ is Cartier 
for every $i$, $\Delta=\sum _{i=1}^kr_i \Delta_i$ holds, and 
$(X, \Delta_i)$ is log canonical for every $i$. 
Therefore, by Lemma \ref{c-lem14.2}, we can write 
\begin{equation*}
-(K_X+\Delta)\cdot C_j =\sum _{i=1}^{l}\frac{r_in_{ij}}{m}>0, 
\end{equation*} 
where 
$n_{ij}$ is an integer with $n_{ij}\leq 2m\dim X$ for 
every $i$ and $j$ since $C_j$ is extremal over $W$. 
If $(K_X+\Delta+H)\cdot R_{j_0}=0$ for some 
$j_0$, then there are nothing to 
prove since $\lambda=1$ and 
$(K_X+\Delta+H)\cdot R=0$ with $R=R_{j_0}$. 
Thus, we assume that $(K_X+\Delta+H)\cdot R_j>0$ for every $j$. 
We put $F=\Supp (\Delta+H)$. Let 
$F=\sum _k F_k$ be the irreducible decomposition. 
We put $V=\bigoplus _k \mathbb R F_k$, 
\begin{equation*} 
\mathcal L(V; \pi^{-1}(W)): 
=\{D\in V\, |\, (X, D) \ \text{is log canonical at $\pi^{-1}(W)$}\}, 
\end{equation*}  
and 
\begin{equation*}
\mathcal N:=
\{D\in \mathcal L(V; \pi^{-1}(W))\, |\, (K_X+D)\cdot R_j\geq 0 
\ \text{for every $j$}\}. 
\end{equation*}  
Then $\mathcal N$ is a rational polytope in $V$ 
by Theorem \ref{c-thm14.3} (3) and $\Delta+H$ is in the relative 
interior of $\mathcal N$ by the above assumption. 
Therefore, after shrinking $Y$ around $W$ suitably again, 
we can write 
\begin{equation*}
K_X+\Delta+H=\sum _{p=1}^{q}r'_p(K_X+D_p), 
\end{equation*} 
where $r'_1, \ldots, r'_q$ are positive real numbers such that 
$\sum _p r'_p=1$, $(X, D_p)$ is log canonical for every $p$, 
$m'(K_X+D_p)$ is Cartier for some positive 
integer $m'$ and every $p$, 
and $(K_X+D_p)\cdot C_j> 0$ for every $p$ and $j$. 
So, 
we obtain 
\begin{equation*}
(K_X+\Delta+H)\cdot C_j=\sum _{p=1}^{q}\frac{r'_p n'_{pj}}{m'} 
\end{equation*}  
with $0< n'_{pj}=m'(K_X+D_p)\cdot C_j\in \mathbb Z$.  
Note that $m'$ and $r'_p$ are independent of $j$ for 
every $p$. 
We also note that 
\begin{align*}
\frac{1}{\mu_j}=\frac{H\cdot C_j}{-(K_X+\Delta)\cdot C_j}&=
\frac{(K_X+\Delta+H)\cdot C_j}{-(K_X+\Delta)\cdot C_j}+1
\\ &
=\frac{m\sum _{p=1}^q r'_p n'_{pj}}{m'\sum _{i=1}^lr_i n_{ij}}+1. 
\end{align*} 
Since 
\begin{equation*} 
\sum _{i=1}^{l}\frac{r_i n_{ij}}{m}>0 
\end{equation*} 
for every $j$ and $n_{ij}\leq 2m\dim X$ with $n_{ij}\in \mathbb Z$ for 
every $i$ and $j$, 
the number of the set $\{n_{ij}\}_{i, j}$ is finite. 
Thus, 
\begin{equation*} 
\inf _j \left\{\frac{1}{\mu_j}\right\}=\frac{1}{\mu_{j_0}}
\end{equation*}  
for some $j_0$. Therefore, we obtain $\mu=\mu_{j_0}$. 
We finish the proof. 
\end{proof}

\section{Basepoint-free theorem for $\mathbb R$-divisors}\label{c-sec15}

In this section, we will discuss the basepoint-free theorem 
for $\mathbb R$-divisors, although we do not need it in this paper. 
The proof of Theorem \ref{c-thm15.1} needs the cone theorem 
(see Theorem \ref{c-thm12.2}). Hence Theorem \ref{c-thm15.1} 
looks much deeper than the basepoint-free theorem 
for Cartier divisors (see Theorem \ref{c-thm9.1}). 

\begin{thm}[Basepoint-free theorem for 
$\mathbb R$-divisors]\label{c-thm15.1}
Let $\pi\colon X\to Y$ be a projective 
morphism of complex analytic spaces such that $X$ is 
a normal complex variety and let $W$ be a compact subset of $Y$ 
such that the dimension of $N^1(X/Y; W)$ is finite. 
Let $\Delta$ be an effective $\mathbb R$-divisor 
on $X$ such that $(X, \Delta)$ is log canonical. 
Let $D$ be an $\mathbb R$-Cartier $\mathbb R$-divisor 
defined on some open neighborhood of $\pi^{-1}(W)$ such that $D$ is 
$\pi$-nef over $W$. 
Assume that $aD-(K_X+\Delta)$ 
is $\pi$-ample over $W$ for some positive 
real number $a$. 
Then there exists an open neighborhood $U$ of $W$ such that 
$D$ is semiample over $U$. 
\end{thm}

Theorem \ref{c-thm15.1} is an application of the cone 
theorem (see Theorem \ref{c-thm12.2}) and 
the basepoint-free theorem for Cartier divisors 
(see Theorem \ref{c-thm9.1}).  

\begin{proof}[Proof of Theorem \ref{c-thm15.1}]
Without loss of generality, 
by replacing $D$ with $aD$, we may assume that 
$a=1$ holds. 
By replacing $Y$ with a relatively compact open neighborhood 
of $W$, we may further assume that 
$\Supp D$ has only finitely many irreducible components and that 
$D$ is a globally 
$\mathbb R$-Cartier $\mathbb R$-divisor on $X$. 
We consider 
\begin{equation*}
F=\{ z\in \NE(X/Y; W)\, | \, D\cdot z=0\}. 
\end{equation*}
Then $F$ is a face of $\NE(X/Y; W)$ and 
$(K_X+\Delta)\cdot z<0$ holds for $z\in F$. 
We take an ample $\mathbb R$-line bundle $\mathcal A$ 
on $X$ such that 
$D-(K_X+\Delta+\mathcal A)$ is still $\pi$-ample 
over $W$.  
Let $R$ be a $(K_X+\Delta)$-negative extremal ray 
of $\NE(X/Y; W)$ such that 
$R\subset F$. Then $R$ is automatically 
a $(K_X+\Delta+\mathcal A)$-negative 
extremal ray of $\NE(X/Y; W)$ 
since $D\cdot R=0$ and $D-(K_X+\Delta+\mathcal A)$ is 
$\pi$-ample over $W$. 
Therefore, $F$ contains only finitely many $(K_X+\Delta)$-negative 
extremal rays $R_1, \ldots, R_k$ of $\NE(X/Y; W)$. 
Thus, $F$ is spanned by 
the extremal rays $R_1, \ldots, R_k$. 
Let $\Supp D=\sum _j D_j$ be the irreducible 
decomposition of $\Supp D$.
Then we consider the finite-dimensional real vector space $V=\underset {j}
\bigoplus \mathbb RD_j$.  
In this situation, we can easily check that 
\begin{equation*}
\mathcal R:=\{ B\in V\, | \, \text{$B$ is a globally $\mathbb R$-Cartier 
$\mathbb R$-divisor 
and $B\cdot z=0$ for every $z\in F$} \} 
\end{equation*} 
is a rational affine subspace of $V$ with $D\in \mathcal R$. 
As in Step \ref{c-thm12.2-step5} in the 
proof of Theorem \ref{c-thm12.2}, 
we put 
\begin{equation*}
\mathcal C_F:=\NE(X/Y; W)_{(K_X+\Delta+\mathcal A)\geq 0}+
\sum _{R_j\not\subset F}R_j. 
\end{equation*} 
and 
\begin{equation*}
\mathcal R^+:=\{B\in \mathcal R\, |\, \text{$B$ is positive 
on $\mathcal C_F\setminus \{0\}$}\}. 
\end{equation*} 
We note that $\NE(X/Y; W)_{\Nlc (X, \Delta)}=\emptyset$ 
since $(X, \Delta)$ is log canonical. 
Then $\mathcal R^+$ is a non-empty open subset 
of $\mathcal R$ with $D\in \mathcal R^+$. 
Hence we can find positive real numbers $r_1, r_2, \ldots, r_m$ and 
globally $\mathbb Q$-Cartier $\mathbb Q$-divisors 
$B_1, B_2, \ldots, B_m \in \mathcal R^+$ such that 
$D=\sum _{i=1}^{m}r_iB_i$ and  
$B_i-(K_X+\Delta)$ is $\pi$-ample over $W$ 
for every $i$. 
We note that $B_i$ is automatically $\pi$-nef over $W$ for every $i$ since 
$B_i\in \mathcal R^+$. 
By the basepoint-free theorem for 
Cartier divisors (see Theorem \ref{c-thm9.1}), 
there exists a relatively compact open neighborhood 
$U$ of $W$ such that 
$B_i$ is $\pi$-semiample over $U$ for every $i$. 
Therefore, $D=\sum _{i=1}^mr_i B_i$ is $\pi$-semiample over $U$. 
This is what we wanted. 
\end{proof}

Theorem \ref{c-thm15.1} will play an important role 
in the study of minimal models of complex analytic spaces. 

\section{Proof of Main theorem}\label{c-sec16}
In this final section, we will prove Theorem \ref{c-thm1.2}, 
which is the main theorem of this paper. 

\begin{proof}[Proof of Theorem \ref{c-thm1.2}]
By Theorem \ref{c-thm12.2} (1) and (2), 
we obtain the following equality 
\begin{equation*}
\NE(X/Y; W)=\NE(X/Y; W)_{(K_X+\Delta)\geq 0} 
+\NE(X/Y; W)_{\Nlc (X, \Delta)}+\sum R_j
\end{equation*} 
satisfying (1), (2), and (3) in Theorem \ref{c-thm1.2}. 
By Theorem \ref{c-thm12.2} (3), 
$F$ in Theorem \ref{c-thm1.2} (4) is rational and 
hence contractible at $\Nlc(X, \Delta)$. 
Thus, by the contraction theorem 
(see Theorem \ref{c-thm12.1}), 
we obtain the desired contraction morphism 
$\varphi_F\colon X\to Z$ over $Y$ after shrinking 
$Y$ around $W$ suitably. 
By Theorem \ref{c-thm13.2}, 
we see that (5) holds. 
We note that (6) is nothing but Theorem \ref{c-thm14.4}. 
Finally, we will prove (7). 
We may assume that $K_X+\Delta$ is not $\pi$-nef 
over $W$. 
Then we can take a small positive real number $\varepsilon$ 
such that $K_X+\Delta+\varepsilon \mathcal H$ is not $\pi$-nef 
over $W$. 
By the cone theorem (see Theorem \ref{c-thm12.2} (2)), 
there exist only finitely many $(K_X+\Delta+\varepsilon \mathcal H)$-negative 
extremal rays $R_1, \ldots, R_k$ of 
$\NE(X/Y; W)$. We put 
\begin{equation*}
\mu_i:=\frac{-(K_X+\Delta)\cdot R_i}{\mathcal H\cdot R_i} 
\end{equation*} 
for every $i$. 
Then it is obvious that $\lambda =\max_{1\leq i\leq k} \mu_i$. 
If $\lambda=\mu_{i_0}$ holds for $1\leq i_0\leq k$, 
then $(K_X+\Delta+\lambda \mathcal H)\cdot R_{i_0}=0$. 
By construction, $K_X+\Delta+\lambda \mathcal H$ is $\pi$-nef 
over $W$. 
We finish the proof of Theorem \ref{c-thm1.2}. 
\end{proof}

\end{document}